\documentclass[twoside,11pt]{article}

\usepackage{blindtext}

\usepackage{amsthm}
\usepackage{amsmath}
\usepackage{amsfonts}
\usepackage{amssymb}
\usepackage{jmlr2e}
\usepackage{graphicx}
\usepackage{esint}
\usepackage{subfigure}
\usepackage{amsmath}
\usepackage{bm}
\usepackage{color}
\usepackage{latexsym}
\usepackage{multicol}
\usepackage{multirow}

\usepackage{lastpage}
\jmlrheading{23}{2022}{1-\pageref{LastPage}}{xx; Revised xx}{xx}{21-0000}{ $^*$ Corresponding author}

\ShortHeadings{Accelerated Primal-dual Mirror Dynamics for Constrained Optimizations}{Y. Zhao, X. F. Liao, X. He and C. Li}
\firstpageno{1}

\begin{document}

\title{Accelerated Primal-Dual Mirror Dynamics for Centrailized and Distributed Constrained Convex Optimization Problems}

\author{\name You Zhao \email zhaoyou1991sdtz@163.com \\
       \addr College of Computer Science\\
       Chongqing University\\
        Chongqing 400044, PR. China
       \AND
       \name Xiaofeng Liao$^*$ \email xfliao@cqu.edu.cn \\
       \addr College of Computer Science\\
       Chongqing University\\
        Chongqing 400044, PR. China
        \AND
       \name Xing He \email hexingdoc@swu.edu.cn \\
       \addr School of Electronics and Information Engineering\\
       Southwest University\\
        Chongqing 400715, PR. China
          \AND
       \name Chaojie Li \email chaojie.li@unsw.edu.au\\
       \addr School of Electrical Engineering and Telecommunications\\
       niversity of New South Wales\\
        Kensington, NSW 2052, Australia
        }
\editor{My editor}

\maketitle

\begin{abstract}
	This paper investigates two accelerated primal-dual mirror dynamical approaches for smooth and nonsmooth convex optimization problems with affine and closed, convex set constraints. In the smooth case, an accelerated primal-dual mirror dynamical approach (APDMD) based on accelerated mirror descent and primal-dual framework is proposed and accelerated convergence properties of primal-dual gap, feasibility measure and the objective function value along with trajectories of APDMD are derived by the Lyapunov analysis method. Then, we extend APDMD into two distributed dynamical approaches to deal with two types of distributed smooth optimization problems, i.e., distributed constrained consensus problem (DCCP) and distributed extended monotropic optimization (DEMO) with accelerated convergence guarantees. Moreover, in the nonsmooth case, we propose a smoothing accelerated primal-dual mirror dynamical approach (SAPDMD) with the help of smoothing approximation technique and the above APDMD. We further also prove that primal-dual gap, objective function value and feasibility measure of SAPDMD have the same accelerated convergence properties as APDMD by choosing the appropriate smooth approximation parameters. Later, we propose two smoothing accelerated distributed dynamical approaches to deal with nonsmooth DEMO and DCCP to obtain accelerated and efficient solutions. Finally, numerical experiments are given to demonstrate the effectiveness of the proposed accelerated mirror dynamical approaches.
\end{abstract}

\begin{keywords}
   machine learning, accelerated mirror dynamical approaches, smooth and nonsmooth optimization, constrained optimization, smoothing approximation, distributed approaches
\end{keywords}

\section{Introduction}

\subsection{Problem statement}\label{sec:Problem}
\label{sec:Problem}
In this paper, we consider the following convex constrained optimization problem: 
\begin{equation}\label{P1}
	\underset{x\in \mathbb{R}^n}{\min}\,\,f\left( x \right),\,\, \mathrm{s}.\mathrm{t}. \,\, Ax=b, x\in \mathcal{X},
\end{equation}
where 
\begin{equation}\label{Assume}
	\begin{cases}
		\mathcal{X} \,\,\mathrm{is} \,\,\mathrm{a}\,\,\mathrm{closed} \,\,\mathrm{and} \,\,\mathrm{convex} \,\,\mathrm{set},% \,\,\mathrm{in}\,\, \mathrm{real}\,\,\mathrm{Hilbert}\,\,\mathrm{space},\\
		\\
		f:\mathcal{X} \rightarrow \mathbb{R}\,\,\mathrm{may} \,\,\mathrm{be}\,\,\mathrm{a}\,\,\mathrm{smooth} \,\,\mathrm{or} \,\,\mathrm{nonsmooth} \,\,\mathrm{convex} \,\,\mathrm{function},\\
		A: \mathcal{X} \rightarrow \mathbb{R}^m \,\,\mathrm{is} \,\,\mathrm{a}\,\, \mathrm{continuous} \,\,\mathrm{linear} \,\,\mathrm{operator}\,\, \mathrm{and}\,\, b\in \mathbb{R}^m,\\
		\mathrm{The} \,\,\mathrm{optimal}\,\,\mathrm{solution} \,\,\mathrm{set} \,\,\mathrm{of} \,\,\mathrm{problem}\,\,\mathrm{\eqref{P1}}  \,\,\mathrm{is} \,\,\mathrm{non}\mathrm{empty}.                     \\
	\end{cases} 
\end{equation}
This problem covers many optimization problems in various applied fields such as machine learning, sparse signal reconstruction, image deblurring, resource allocation, saddle point, Markov decision processes, regularized empirical risk minimization, and supervised machine learning. (see, e.g. \cite{Boyd2011DistributedOA, lin2020accelerated, Darbon2021AcceleratedNP, OConnor2014PrimalDualDB, Nandwani2019APD, Zhang2017, Tiapkin2022PrimalDualSM}).

%\begin{remark}\label{R2}
%	The DEMO problem \eqref{P3} encompasses many problems in economic and machine learning due to its universal expressions. For instances, it extends the optimization model in the resource allocation problem by allowing the more general equation constraints \cite{}. It also includes the model presented in \cite{} and generalizes the model to the distributed constrained optimal consensus problem \cite{} by allowing heterogeneous constraints. Moreover, this problem can be seen as another form of the widely studied linear algebraic equations for distributed computation in \cite{} and dsitributed sparse signal recovery.	
%\end{remark}

\subsection{Historical presentation} The inertial (scend-order or accelerated) dynamical approaches are increasingly popular for solving the unconstrained optimization problem
\begin{equation} \label{uncon_optimization}
	\underset{x\in \mathbb{R}^n}{\min}  f\left( x \right).
\end{equation}

To deal with problem \eqref{uncon_optimization}, \cite{Polyak1964SomeMO} first proposed the heavy ball with friction dynamical approach
\begin{equation} \label{heavy_ball}
	(\mathrm{HBF}) \,\,\,\,\ddot{x}\left( t \right) +\eta \dot{x}\left( t \right) +\nabla f\left( x \right) =0,
\end{equation}
where $\eta>0$ is a damping parameter. Later, \cite{Alvarez2000OnTM} studied the asymptotic behavior of \eqref{heavy_ball} with a time independent parameter $\eta$ when $f\left( x \right)$ is convex.  \cite{Jendoubi2015OnDS} studied some convergence properties of the HBF \eqref{heavy_ball} with a constant $\eta$ when $f\left( x \right)$ is a nonconvex function. \cite{Aujol2022ConvergenceRO, Aujol2020ConvergenceRO} investigated the convergence rate of HBF dynamical system and its corresponding discrete algorithm with a fixed parameter $\eta$ under the condition that $f\left( x \right)$ satisfies quasi-strongly convex and Lojasiewicz properties, respectively. \cite{Su2016ADE} first revealed that the HBF \eqref{heavy_ball} with $\eta =\frac{\alpha}{t},\alpha \geq 3$ can be regarded as the continuous-time limit of the Nesterov's accelerated gradient algorithm and it has an accelerated convergence rate, i.e, $f\left( x\left( t \right) \right) -\min f\left( x \right) =O\left( \frac{1}{t^2} \right)$. \cite{Elloumi2017AsymptoticFA} further improved the convergence rate of it to $f\left( x\left( t \right) \right) -\min f\left( x \right) =o\left( \frac{1}{t^2} \right)$. When $\eta =\frac{\alpha}{t}, 0<\alpha \leq3$, the convergence rate  $f\left( x\left( t \right) \right) -\min f\left( x \right) =O\left( t^{-\frac{2\alpha}{3}} \right) $ was estimated by \cite{Attouch2019RateOC} and \cite{Apidopoulos2018TheDI} for smooth and nonsmooth convex functions $f\left( x \right)$. \cite{Cabot2012AsymptoticsFS} studied the asymptotic behavior of the HBF \eqref{heavy_ball} when $\eta =\frac{\alpha}{t^{\gamma}},0<\gamma <1$. Moreover, \cite{Wibisono2016AVP} studied a series of accelerated (inertial) dynamical approaches to the problem \eqref{uncon_optimization} by the Bregman Lagrangian based on the calculus of variations. \cite{Kovachki2021ContinuousTA} studied the behavior of momentum methods for solving problem \eqref{uncon_optimization} from a continuous-time perspective.
\cite{wilson2021} proposed several Lyapunov functions for analyzing the accelerated (momentum) algorithms to solve the problem \eqref{uncon_optimization}. By the manifold with curvature bounded from below, \cite{Alimisis2020ACP} proposed a Riemannian variant of accelerated gradient dynamical approach.

Recently, to solve the problem \eqref{uncon_optimization} with a closed and convex set constraint $\mathcal{X}$, i.e., 
\begin{equation} \label{set_optimization}
	\underset{x\in \mathbb{R}^n}{\min} f\left( x \right),\,\, \mathrm{s}.\mathrm{t}. \,\, x\in \mathcal{X} \subset \mathbb{R} ^n,
\end{equation}
many inertial dynamical approaches have been studied. Based on the projection operators, \cite{he2016inertial} proposed an inertial dynamical approach to solve the problem \eqref{set_optimization} with non-convex objective functions. Later, based on the work in \cite{he2016inertial} and smoothing approximation technique, two smoothing inertial hydrodynamic  approaches were proposed to solve constrained nonconvx $L_{p-q}$ minimization problem to reconstruct the sparse signal in \cite{Zhao2018SmoothingIP}. Combining the dynamical approach of mirror descent with the continuous version of Nesterov's acceleration algorithm, \cite{Krichene2015AcceleratedMD} proposed an accelerated mirror dynamical approach for problem \eqref{set_optimization} as follows: 
\begin{equation}\label{mirror}
	\begin{cases}
		\dot{X}=\frac{\gamma}{t}\left( \nabla \psi ^*\left( Z \right) -X \right),\\
		\dot{Z}=-\frac{t}{\gamma}\nabla f\left( X \right), \,\,      \\
		X\left( 0 \right) =x_0, Z\left( 0 \right) =z_0, \mathrm{with} \nabla \psi ^*\left( z_0 \right) =x_0,\\
	\end{cases}	
\end{equation}
where $\gamma \geq 2$ and $\nabla \psi ^*$ is the gradient of conjugate of $\psi$ (see \eqref{mirr} ). 

In addition, in order to solve the problem \eqref{uncon_optimization} with an affine constraint, i.e, 
\begin{equation} \label{aff_optimization}
	\underset{x\in \mathbb{R} ^n}{\min}\,\,  f\left( x \right),\,\, \mathrm{s}.\mathrm{t}. \,\, Ax=b,
\end{equation}
many inertial dynamical approaches under primal-dual framework were extensively investigated. \cite{Zeng2019DynamicalPA} first proposed a second-order dynamical primal-dual approach for solving the problem \eqref{aff_optimization} and proved that the convergence rate of the gap between the Lagrangian function and its optimal value is $O\left( \frac{1}{t^2} \right) $. The corresponding inertial dynamical approach is then extended to solve the distributed optimization problems. Later, \cite{Bo2021ImprovedCR} improved the convergence rates of the works in \cite{Zeng2019DynamicalPA} and provided the weakly convergence analysis to a primal-dual optimal solution of the problem \eqref{aff_optimization}. \cite{He2021ConvergenceRO} studied an inertial primal-dual dynamical approach without/with perturbations for separable convex optimization problems with an affine constraint. \cite{Attouch2021FastCO} proposed a temporally rescaled inertial augmented Lagrangian system (TRIALS) with three time-varying parameters (i.e., viscous damping, extrapolation and temporal scaling) to address separable smooth/nonsmooth convex optimization problems with an affine constraint, and presented the fast convergence properties of TRIALS. In addition, \cite{Luo2021AcceleratedPM} further proposed a ``second-order"+``first-order" primal-dual dynamical approaches for solving problem \eqref{aff_optimization} with accelerated convergence guarantees.

However, many practical problems in machine learning, sparse signal reconstruction, image deblurring, resource allocation and so on , not only have affine constraints, but also have set constraints, as in problem \eqref{P1}. The problem \eqref{P1} can be regarded as a generalization for the problems \eqref{uncon_optimization}, \eqref{set_optimization}, \eqref{aff_optimization}, \eqref{P2} and \eqref{P3}. In recent years, some first order dynamical approaches  based on primal-dual framework and projection were proposed by \cite{Liu2017ACN,  Qu2019OnTE, Yi2016InitializationfreeDA, zeng2018distributed} to solve the problem \eqref{P1} only with convergence analysis or ergodic $O\left( \frac{1}{t} \right) $ convergence rate. To the best of our knowledge, accelerated (inertial) dynamical approaches based on primal-dual framework with non-ergodic $O\left( \frac{1}{t^2} \right) $ convergence rates are rarely involved. It is worth noting that the inertial dynamical approaches mentioned above were used for solving the problems \eqref{uncon_optimization}, \eqref{set_optimization} and \eqref{aff_optimization} and they cannot be used directly for the problem \eqref{P1}. 

This paper aims to investigate accelerated primal-dual dynamical approaches based on mirror descent and smoothing approximation methods for solving problem \eqref{P1} with an accelerated non-ergodic convergence rate $O\left( \frac{1}{t^2} \right)$. Our contributions are summarized as follows:
\begin{itemize}
	\item
	For the problem \eqref{P1} in the smooth case, an accelerated primal-dual mirror dynamical approach (APDMD) is proposed for the first time, it is used to solve problem \eqref{P1} with accelerated convergence rates of Lagrangian and objective functions? gaps. We provide an interpretation of the APDMD from different perspectives (i.e., neurodynamic approach, Hamilton's system, game theory and control theory). Moreover, based on the properties of conjugate function and Cauchy-Lipschitz-Picard theorem, the feasibility (i.e, ensuring that the trajectories of solutions always satisfy the set constraint), existence and uniqueness of the global solution of APDMD are obtained. Last, applying the APDMD to DCCP and EDMO leads to two distributed APDMDs (ADPDMD and ADMD) with accelerated convergence guarantees.

	\item  For the problem \eqref{P1} in the nonsmooth case, a smoothing accelerated primal-dual mirror dynamical approach (SAPDMD) for the problem \eqref{P1} is also proposed based on the smoothing approximation technique and APDMD, and it has accelerated convergence rates of Lagrangian and objective function gaps. We provide a comparative explanation between SAPDMD and state-of-the-art dynamical approaches besed on differential inclusion, Moreau-Yosida regularization and directional derivative methods. Moreover, we further analyze the convergence properties of SAPSMD and provide smooth parameter selection condition. Last, applying the SAPDMD to DCCP and EDMO leads to two distributed SAPDMDs (SADPDMD and SADMD) with accelerated convergence gurantees.
	
	\item 	
	Last but not least, the asymptotic analysis and the obtained results in this paper can be straight forwardly transferred to inertial dynamical approaches for the problems \eqref{uncon_optimization}, \eqref{set_optimization} and \eqref{aff_optimization}. 
\end{itemize}

The paper is organized as follows. Section~2 introduces some preliminaries of convex analysis, saddle point theorem, dual distance, Bregman divergence, projection operators, graph theory and smoothing approximation. Then, we propose the APDMD to solve problem \eqref{P1} in the smooth case, and discuss some convergence properties of APDMD in Section~3. In Section~4, we propose a SAPDMD (smoothing version of APDMD) for solving the problem \eqref{P1} in the nonsmooth case, and provide a detailed discussion of the convergence properties of SAPDMD and then extend the SAPDMD to address the DCCP and EDMO in the nonsmooth case. Section~5 provides several experiments to demonstrate the theoretical results. Finally, we conclude this paper in Section~6.

\section{Preliminaries}\label{sec:Mathpr} 
This section gives some essential mathematical preliminaries.

\subsection{Convex analysis} 
A function $f:\mathcal{X} \rightarrow \mathbb{R}$ is convex, if it satisfies $ \theta f\left( x \right) +\left( 1-\theta \right) f\left( z \right) \ge f\left( \theta x+\left( 1-\theta \right) z \right),\forall x,z\in \mathcal{X}, x\ne z$, and $\theta \in \left[0,1 \right] $.
The subdifferential $g_{f}\left( x \right)$ of $f\left( x \right)$ with respect to $x \in \mathcal{X}$ is defined by $g_f\left( x \right) =\left\{ p\in \mathbb{R} ^n|f\left( z \right) -f\left( x \right) \ge p^T \right. \left. \left( z-x \right) ,\forall x,z\in \mathcal{X} \right\} $, and the element $\partial f\left( x \right) $ of $g_{f}\left( x \right)$ is called subgradient of $f\left( x \right)$. In addition, if $f$ is smooth, the subgradient $\partial f\left( x \right) $ reduces to gradient $\nabla f\left( x \right) $. %?? $\footnote{We use the supscript $t$ for parameters in the $t^{th}$ round}$

\subsection{Saddle point Theorem}\label{saddle}
The (augmented) Lagrangian $\mathcal{L}_{\beta}$ $:$ $\mathcal{X} \times \mathbb{R} ^m \rightarrow \mathbb{R}$ with respect to the problem \eqref{P1} with $\beta \geq0$ is defined by 
\begin{equation}\label{AugLag}
	\mathcal{L} _{\beta}\left( x,\lambda \right) =f\left( x \right) +\lambda ^T\left( Ax-b \right) +\frac{\beta}{2}\left\| Ax-b \right\| ^2, 
\end{equation}
Then, $\left( x^*,\lambda ^* \right) $ is an optimal solution of the problem \eqref{P1} if and only if it is a saddle point of $\mathcal{L} _{\beta}$, i.e.,
\begin{equation}\label{sd}
	\mathcal{L}_{\beta}\left( x^*,\lambda \right) \leq \mathcal{L} _{\beta}\left( x^*,\lambda ^* \right) \leq \mathcal{L} _{\beta}\left( x,\lambda ^* \right) , \forall \left( x,\lambda \right) \in \mathcal{X} \times \mathbb{R}^m.
\end{equation}

\subsection{Dual distance and Bregman divergence}\label{mirr} 
Define  $\psi: \mathcal{X} \rightarrow \mathbb{R}$ and let $\mathcal{X} $ be a closed and convex set , then, by the \emph{Legendre-Fenchel transform}, the conjugate function of $\psi$ is 
\begin{equation*}
	\begin{split}
		\psi ^*\left( u \right) =\underset{x\in \mathcal{X}}{\mathrm{sup}}\left\{ u^Tx-\psi \left( x \right) \right\}.
	\end{split}
\end{equation*}
If $\psi \left( x \right) $ is a proper, lower semi-continuous and convex function, then for all  $u\in \mathcal{X} ^*$, we have
\begin{equation*}
	\begin{split}
		\psi ^{\textsc{**}}\left( u \right) =\underset{u\in \mathcal{X} ^*}{\mathrm{sup}}\left\{ x^Tu-\psi ^*\left( u \right) \right\} =\psi \left( x \right),
	\end{split}
\end{equation*}
from the Fenchel's duality theorem.

Assume that $\psi$ and $\psi ^*$ are proper and convex functions, then they're subdifferentiable, i.e., $\partial \psi \left( x \right)$, $\partial \psi ^* \left( u\right)$ exist in the relative interior of their domains $\mathcal{X}$, $\mathcal{X} ^*$, respectively \cite{Rockafellar1970ConvexA}. In addition, using the condition that $\psi ^{**}\left( u \right) =\psi \left( x \right) $, one has 
\begin{equation*}
	\begin{split}
		\partial \psi ^* \left( u \right) +\partial \psi \left( x \right) =u^Tx\Leftrightarrow u\in \partial \psi \left( x \right) \Leftrightarrow x\in \partial \psi ^*\left( u \right),		
	\end{split}
\end{equation*}
which implies $\partial \psi ^*\left( u \right) =arg\underset{x\in \mathcal{X}}{\max}\left\{ u^Tx-\psi \left( x \right) \right\}$. Since $\mathrm{dom}\psi =\mathcal{X} $,  then $\partial \psi ^*\left( u \right) \subset \mathcal{X} $ , i.e., the set-valued mapping $\partial\psi ^*$ can map $\mathcal{X} ^*$ into $\mathcal{X}$. In order to make $\partial \psi ^*\left( u \right) \subset \mathcal{X} $ be unique, i.e., it maps to a single-value, in other words, the $\psi ^*\left( u \right) $ is differentiable for any $u\in \mathcal{X} ^*$, the following definitions are needed (see \cite{krichene2016lyapunov}).
\begin{definition}
	A convex function $\psi $ is cofinite if its epigraph does not consist of any non-vertical half-line.
\end{definition}
\begin{definition}
	A convex function $\psi $ is essentially strictly convex if it is strictly convex and subdifferentiable on all convex subsets.
\end{definition}
\begin{lemma}
	If  $\psi$ and $\psi ^*$ are proper, convex, and closed, so they are inverses of each other, then $\psi ^*$ is finite and differentiable on $\mathcal{X} ^*$ if and only if $\psi $ is essentially strictly convex and cofinite.
\end{lemma}

\begin{lemma} \textbf{Bregman divergence}: The distance between $h$ at $y \in \mathcal{X} $ and its first-order Taylor series approximation at $z \in \mathcal{X}$ is given by : $$D_h\left( y,z \right) =h\left( y \right) -h\left( z \right) -\nabla h\left( z \right) ^T\left( y-z \right) , \forall y, z\in \mathcal{X},$$
	which is nonnegative if $h$ is convex and is an approximation to the Hessian metric, i.e., $D_h\left( y,z \right) \approx \frac{1}{2}\left( y-z \right) ^T\nabla ^2h\left( x \right) \left( y-x \right), \forall y, z \in \mathcal{X} $ when $y$ is close to $z$. %The Bregman divergence enjoys the following property:
	% The Bregman three-point identity: 
	%$$D_h\left( \mathsf{u} ,\mathsf{z} \right) -D_h\left( \mathsf{w} ,\mathsf{z} \right) -D_h\left( \mathsf{u} ,\mathsf{w} \right) =\left( \nabla h\left( \mathsf{w} \right) -\nabla h\left( \mathsf{z} \right) \right) ^T\left( \mathsf{u} -\mathsf{w} \right) ,\forall \mathsf{u} ,\mathsf{z} ,\mathsf{w} \in \mathcal{X},$$
	%ii): For the fenchel conjugate $h^*$ to $h$, it satisfies
	%  $$D_{h^*}\left( y,z \right) =D_h\left( \nabla h^*\left( z \right) ,\nabla h^*\left( y \right) \right) , \forall y, z\in \mathcal{X}.$$
\end{lemma}
\subsection{Projection operators} \label{proj} Define the projection operator of a closed and convex set $\mathcal{X}$ at $x$  be $P_{\mathcal{X}}\left( u \right) =\arg\underset{z\in \mathcal{X}}{\min}\left\| u-z \right\| $.
  A basic property of the projection operator $P_{\mathcal{X}}$ is 
  \begin{align*}
	\left( P_{\mathcal{X}}\left( x \right) -z \right) ^T\left( P_{\mathcal{X}}\left( x \right) -x \right) \leq 0,\forall x\in \mathbb{R}^n, z\in \mathcal{X}.
	 \end{align*}
%The normal cone of set $\mathcal{X}$ is
%	\begin{align*}
%			\mathcal{N} _{\mathcal{X}}\left( x \right) =\mathrm{cl(}\bigcup_{\nu \ge 0}{\nu \partial d_{\mathcal{X}}\left( x \right)})
%			=\left\{ v\in \mathbb{R^n}|v^T\left( z-x \right) \leq 0,\forall z\in \mathcal{X} \right\} ,
%		\end{align*}
%where the symbol $\text{cl}(S)$ represents closure of set $S$, $d_{ \mathcal{X}}\left( x \right) =\underset{z\in \mathcal{X}}{\min}\lVert x-z \rVert $.
\begin{lemma} 
	%The projection operators $P_{\mathcal{X}}$ have closed-form solutions for some special convex set $\mathcal{X}$, which is described as follows?
	
	1): If $\mathcal{X}$ is a \textbf{box} set, i.e. $\mathcal{X}=\left\{ u\in \mathbb{R}^n\mid\underline{u}_i\le u_i\le \bar{u}_i,\,i=1,...,n \right\}$, then, for any $i$, we have 
%	$P_{\mathcal{X}}\left( u_i \right) =\max \left\{ \min \left\{ u_i,\bar{u}_i \right\} ,\underline{u_i} \right\}$.
	$$P_{\mathcal{X}}\left(u_i\right)=\begin{cases}
			\underline{u}_i,&		u_i<\underline{u}_i,\\
			u_i,&		\underline{u}_i\le u_i\le \bar{u}_i,\\
			\bar{u}_i,&		u_i>\bar{u}_i.\\
		\end{cases}$$
		
		2): If $\mathcal{X}$ is a \textbf{Sphere} set, i.e., $\mathcal{X}=\left\{ u\in \mathbb{R}^n\mid \left\| u-z \right\| \le r, z\in \mathbb{R}^n, r>0 \right\} $, then, the projection operator is 
		$P_{\mathcal{X}}\left( u \right) =\begin{cases}
			u,&		\mathrm{if}\,\,\,\left\| u-z \right\| \le r,\\
			z+\frac{r\left( u-z \right)}{\left\| u-z \right\|},&		\mathrm{if}\,\,\,\left\| u-z \right\| >r.\\
		\end{cases}$
		
		3): If $\mathcal{X}$ is an \textbf{affine} set, i.e., $\mathcal{X} =\left\{ u\in \mathbb{R}^n\mid Au=b, A\in \mathbb{R}^{m\times n} \right\} $, then, the projection operator is $P_{\mathcal{X}}\left( u \right) =u+A^{\dagger}\left( b-Au \right)$, where $A^{\dagger}$ 
		is Moore-Penrose pseudoinverse of $A$ if $\mathrm{Rank}\left( A \right) <m$, and  $A^{\dagger}=A^T\left( AA^T \right) ^{-1}$ when $\mathrm{Rank}\left( A \right) =m$. 
		
		4): If  $\mathcal{X}$ is a \textbf{half-space} set, i.e., $\mathcal{X}=\left\{ \begin{array}{c}	u\in \mathbb{R}^n\\ \end{array} \right. \mid \left. a^Tu\leq b,a\ne 0 \right\} $, then, the projection operator is $P_{\mathcal{X}}\left( u \right) =\begin{cases}
			u+\frac{b-a^Tu}{\left\| a \right\| ^2}a,&		\mathrm{if}\,\,a^Tu>b,\\
			u,&		\mathrm{if}\,\,a^Tu\leq b.\\
		\end{cases}$
		%4): Assume that $\mathcal{X}$ is a \textbf{hyperplane} set, i.e., $\mathcal{X}=\left\{ \begin{array}{c}	u\in \mathbb{R}^n\\ \end{array} \right. \mid \left. a^Tu=b,a\ne 0 \right\} $. The projection operator is 
		%$P_{\mathcal{X}}\left( u \right) =u+\frac{b-a^Tu}{\left\| a \right\| ^2}a$.
		
		For more information regarding the projection operators, (please refer to \cite{Bauschke2011ConvexAA, Parikh2014ProximalA}).
	\end{lemma}
	\subsection{Graph Theory} 
	An undirected communication topology graph is a triplet $\mathcal{G}=\left(\mathcal{V},\mathcal{E},\mathcal{A} \right)$ with node set $\mathcal{V}=\left\{ \nu_1, \nu_2,..., \nu_n\right\}$, edge set $\mathcal{E}\subseteq \mathcal{V}\times \mathcal{V}$ and connection matrix $\mathcal{A}=\left\{\mathrm{a}_{ij} \right\} _{n\times n}$ with nonnegative elements $\mathrm{a}_{ij}=\mathrm{a}_{ji}>0$ if $\left( i,j \right) \in \mathcal{E}$, and $\mathrm{a}_{ij}=\mathrm{a}_{ji}=0$ otherwise. The coupling of agents in an undirected graph is unordered, which means that there exists information exchange for both agent $i$ and agent $j$. A path in an undirected graph between agent $i$ and agent $j$ is a sequence of edges of the form $\left( i,i_1 \right) $, $\left(i_1,i_2 \right) $, $\dots$, $\left( i_s,j\right) $, where $i$, $i_1$, $\cdots $, $i_s$, $j$ denote different agents. Let $\mathcal{N}_i=\left\{ j|\left( i,j \right) \in \mathcal{E} \right\}$ be an agent $i$'s neighbor set. The undirected graph $\mathcal{G}$ is connected if there exists a path between any pair of distinct nodes $v_i$ and $v_j$ $(i,j=1,2,...,n)$ (see \cite{Godsil2001AlgebraicGT}).
	
	\subsection{Smoothing approximation} \label{SA} The main characteristic of the smoothing method is to approximate the nonsmooth function with a parameterized smoothing function. In this paper, we adopt a smoothing function, which is defined as follows:
	
	\begin{definition} 
		Let $ \hat{f}: \mathcal{X} \times \left [0, \bar{\mu}\right] \rightarrow \mathbb{R}$ with $\bar{\mu}>0$ be a smoothing function of the convex function $f:\mathcal{X}\rightarrow \mathbb{R}$, if $\hat{f}\left( \cdot ,\mu \right) $ is continuous differentiable for any fixed $\bar{\mu}\ge\mu>0$, then it enjoys the following properties ( see \cite{Bian2020ASP,Chen2012SmoothingMF} \label{Defi_sm}  )
		
		(i) (approximation property) $\underset{x\rightarrow w,\mu \rightarrow 0}{\lim}\hat{f}\left( x ,\mu \right) =f\left( w \right), \forall \,\,x\in \mathcal{X}$;
		
		(ii) (convexity) For any fixed $\mu>0$, $\hat{f}\left( x,\mu \right) $ is a convex function of $x$ in $ \mathcal{X} $;
		
		(iii) (gradient consistency) $\big\{\underset{x\rightarrow w,\mu \rightarrow 0}{\lim}\nabla _x\hat{f}\big(x,\mu \big) \big\} \subseteq \partial f\big(w \big), \forall \,\,x\in \mathcal{X} $;
		
		(iv) (gradient boundedness and Lipschitz continuous of $\mu$)
		There exists a positive constant $\kappa _{\hat{f}}>0$ such that
		\begin{equation*}
			\begin{split}
				\left| \nabla _{\mu}\hat{f}\left(x,\mu \right) \right|\leq\kappa _{\hat{f}}\ \ \ \forall \,\,\mu \in \left[0,\bar{\mu}\right], \forall \,\,x\in \mathcal{X},
			\end{split}
		\end{equation*}
		and for any $x\in \mathcal{X}$,$u_1,\mu _2\in \left[ 0,\bar{\mu} \right]$, it follows	
		$$\left| \hat{f}\left( x,\mu _1 \right) -\hat{f}\left( x,\mu _2 \right) \right|\leq \kappa _{\hat{f}}\left| \mu _1-\mu _2 \right|;$$
		
		(v) (Lipschitz continuity with respect to $x$) there is a constant $\ell$ such that for any fixed $ \mu \in \left[0,\bar{\mu}\right)$, $\nabla _x\hat{f}\left( x,\cdot \right) $ is Lipschitz continuous with respect to $x$ on $\mathcal{X}$ with a Lipschitz constant $\frac{\ell}{\mu}$.
	\end{definition}
	
	In addition, the \eqref{Defi_sm} \textit{(iv)} implies that $\left| \hat{f}\left( x,\mu \right) -f\left( x \right) \right|\le \kappa _{\hat{f}}\mu,\,\, \forall\,\, 0<\mu \le \bar{\mu}, \,\,x\in \mathcal{X}$.
	
	The smoothing function satisfying the above conditions in \eqref{Defi_sm} enjoys the following properties (see \cite{Bian2014NeuralNF}):
	
	\begin{lemma} 	
		1): If $\hat{f}_1,\hat{f}_2,...,\hat{f}_n$ are smoothing functions of $f_1,f_2,...,f_n$, then $\sum_{i=1}^n{c_i\hat{f}_i}$ is a smoothing function of $\sum_{i=1}^n{c_if_i}$ with $\kappa _{\sum_{i=1}^m{c_i\hat{f}_i}}=\sum_{i=1}^m{c_i\kappa _{\hat{f}_i}}$ when $c_i\ge 0$ and $f_i$ is regular for any $i=1,2,...,n$.	
		
		2): If $\varPhi: \mathcal{X}\rightarrow \mathbb{R}$ is locally Lipschitz, $\varPsi: \mathbb{R}\rightarrow \mathbb{R}$ is continuously differentiable and globally Lipschitz with a Lipschitz constant $l_{\varPsi}$. Let $\hat{\varPhi}$ be a smoothing function of $\varPhi$, then $\varPsi \left( \hat{\varPhi} \right) $ is a smoothing function of $\varPsi \left( \varPhi \right) $ with $\kappa _{\varPsi \left( \hat{\varPhi} \right)}=l_{\varPsi}\kappa _{\hat{\varPhi}}$.
		
		3):  Let $\varPhi: \mathbb{R}^m\rightarrow \mathbb{R}$ be regular and $\varPsi: \mathbb{R}^n\rightarrow \mathbb{R}^m$ be continuously differentiable. If $\hat{\varPhi}$ is a smoothing function of $\varPhi$, then $\hat{\varPhi}\left( \varPsi \right) $ is a smoothing function of $\varPhi\left( \varPsi \right) $ with 	$\kappa _{\hat{\varPhi}\left( \varPsi \right)}=\kappa _{\hat{\varPhi}}$.
	\end{lemma}
	
	\begin{example} \label{example1} The existing results in \cite{Chen2012SmoothingMF} provide some theoretical basis to construct smoothing functions that satisfy the conditions in  \eqref{Defi_sm}. A smoothing function for the $g\left( s \right) =\max \left\{ 0,s \right\}$ is given by 
		\begin{align}\label{sm1}
			\hat{g}\left( s,\mu \right) =\begin{cases}
				\max \left\{ 0,s \right\} , \mathrm{if} \left| s \right|>\mu,\\
				\frac{\left( s+\mu \right) ^2}{4\mu},\,\,\,\,\,\,\,\,\,\,\,\mathrm{if} \left| s \right|\leq \mu,\\
			\end{cases}
		\end{align}
		with a $\kappa _{\hat{g}}=\frac{1}{4}$. Note that $\hat{g}\left( s,\mu \right)$ is convex and nondecreasing of $s$ with any fixed $0<\mu \le \bar{\mu}$, is also nondecreasing with respect to $\mu$ for any fixed $s$, and $\underset{\mu \rightarrow 0}{\lim}\hat{g}\left( s,\mu \right) =\max \left\{ 0,s \right\}$ (see Figure \eqref{fig:sm} (left)).
		
		Many nonsmooth convex functions in applications can be reformulated by the $\max \left\{ s,0 \right\}$, such as
		\begin{equation*}
			\begin{split}
				\max \left\{ s,x \right\} =s+\max \left\{ s-x,0 \right\},
				\\
				\min \left\{ s,x \right\} =s-\max \left\{ s-x,0 \right\},
				\\
				\mathrm{mid}\left\{ s,x,w \right\} =\min \left\{ \max \left\{ x,s \right\} ,w \right\},
				\\
				\theta\left( s \right)=\left| s \right|=\max \left\{ s,0 \right\} +\max \left\{ -s,0 \right\},
			\end{split}
		\end{equation*}
		where the smoothing approximation function of $\theta \left( s \right) =\left| s \right|$ is 
		\begin{align}\label{sm2}
			\hat{\theta}\left( s,\mu \right) =\begin{cases}
				\left| s \right|, \,\,\,\,\,\,\,\,\,\,\,\mathrm{if} \left| s \right|>\frac{\mu}{2},\\
				\frac{s^2}{\mu}+\frac{\mu}{4},\mathrm{if} \left| s \right|\leq \frac{\mu}{2},\\
			\end{cases}
		\end{align}
		where $\underset{\mu \rightarrow 0}{\lim}\hat{\theta}\left( s,\mu \right) =\lvert s \rvert$, and $\kappa _{\hat{\theta}}=\frac{1}{4}$ (It is used in \eqref{non-experiment}).
	\end{example}
	 As can be seen from Figure \eqref{fig:sm} (right) that $\hat{\theta}\left( s,\mu \right)$ is also convex and nondecreasing for any fixed $0<\mu \le \bar{\mu}$, and nondecreasing with respect of $s$ for any fixed $s$.
	\begin{figure*}[htbp]
		\centering
		\includegraphics[width=7cm,height=5cm]{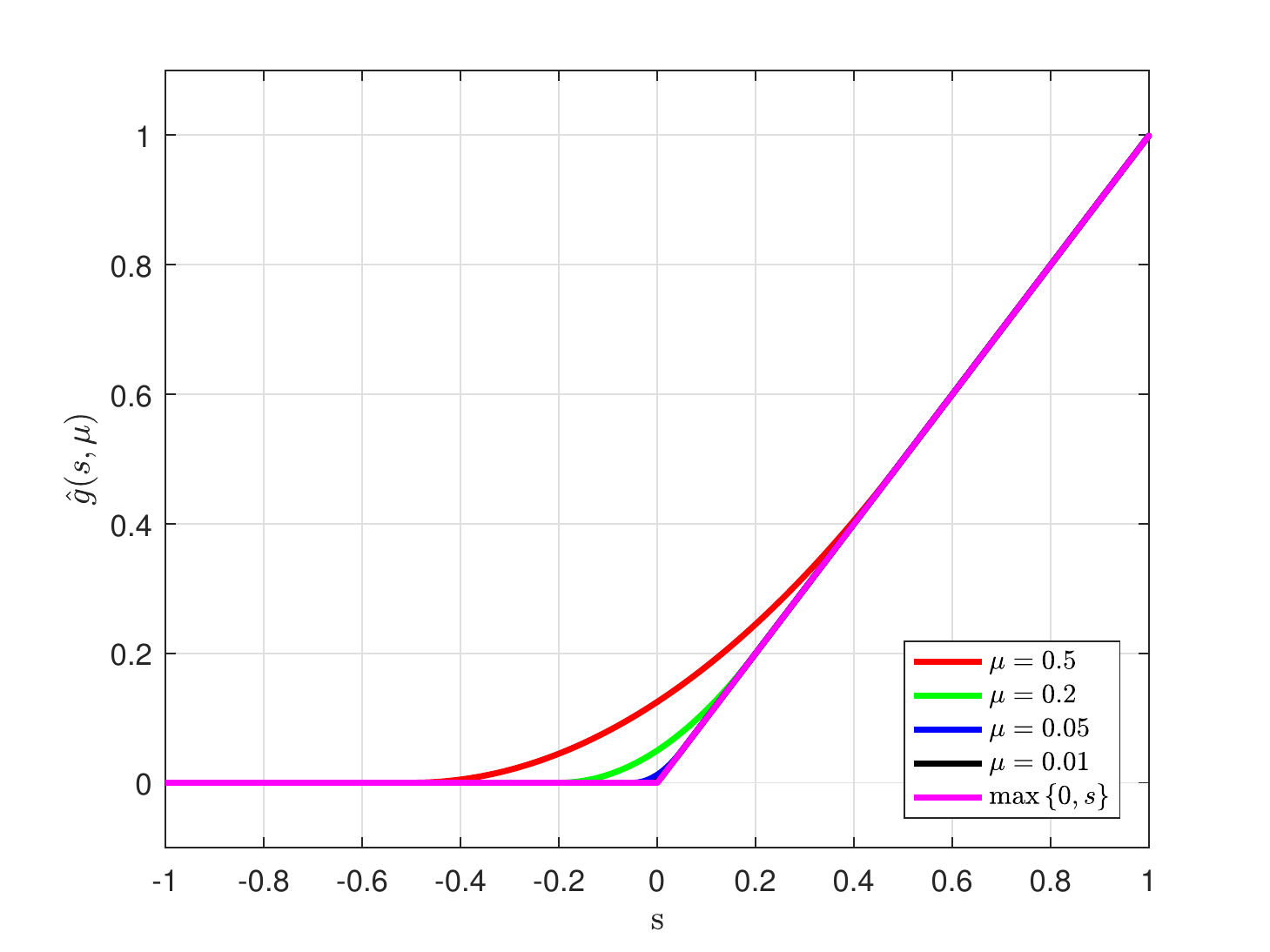}
		\centering
		\includegraphics[width=7cm,height=5cm]{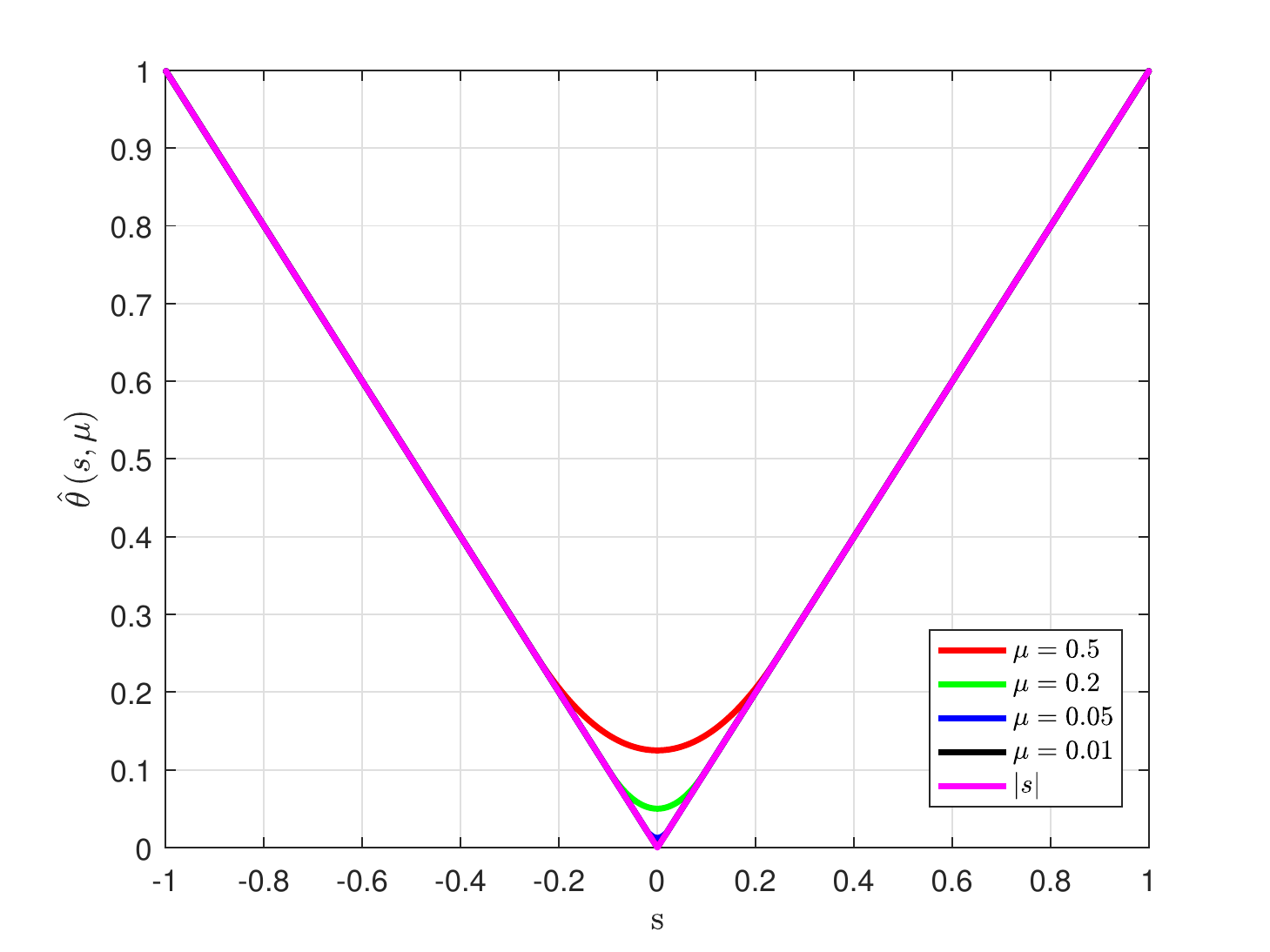}
		\caption{(left) The smoothing function $\hat{g}\left( s,\mu \right)$ with different $\mu$. (right) The smoothing function $\hat{\theta}\left( s,\mu \right)$ with different $\mu$.}
		\label{fig:sm}
	\end{figure*}
	\subsection{Notation} Let $\mathbb{R} ^n$ ($\mathbb{R} ^{m\times n}$) be the $n$-dimensional (or $m$-by-$n$) real vectors ( or real matrices), and $I_n$ be a $n \times n$ identity matrix. For vectors $x$, $y\in \mathbb{R}^n$, $x^Ty=\sum_{i=1}^n{x_iy_i}$, and the superscript $T$ represents transpose. $\mathrm{col}\left( x_1,...,x_n \right) =\left( x_{1}^{T},...,x_{n}^{T} \right) ^T$, $\left\| \cdot \right\| $ is the Euclidean norm, $\left\| \cdot \right\| _1$ is the 1-norm. $\mathrm{L}_{loc}^{1}\left( \left[ t_0,+\infty \right) \right) $ represents the local Lebesgue integrable functions on $\left[ t_0,+\infty \right)$. Let $A_i\in \mathbb{R} ^{p_i\times q_i} (p_i, q_i>0), i=1,...,n$, then we define $\bar{A}=\mathrm{bl}\mathrm{diag}\left\{ A_1,...,A_n \right\} \in \mathbb{R} ^{\sum_{i=1}^n{p_i}\times\sum_{i=1}^n{q_i}}$, which means a block diagonal matrix.  For $x_i\in \mathbb{R} ^m$, $i=1,...,n$, we use $\mathrm{col}\left( x_1,...,x_n \right) \in \mathbb{R} ^{mn}$ to denote an $mn$ column vector.
	
\section{Optimization approaches for problem (\ref{P1}) in the smooth case} In this section, we propose an APDMD approach to address the problem (\ref{P1}) with smooth and convex objective function. Then, we extend it to solve ECCP (\ref{P2}) and DEMO (\ref{P3}) in the smooth case, to obtain ADPDMD (\ref{ADPDMD}) and ADMD (\ref{ADMD}), respectively. To our best knowledge, there are no accelerated mirror dynamical approaches for the problem (\ref{P1}), distributed accelerated dynamical approaches for ECCP (\ref{P2}) and DEMO (\ref{P3}) only with smooth convex objective functions.

To make the well-posedness of the problem (\ref{P1}), some appropriate assumptions are needed, which are fairly standard.

{\bf Aassume} \label{assume1}
	\,\,\,\,\,\,\,\,\,\,\,\,\,\,\,\,\,\,\,\,\,\,\,\,\,\,\,\,\,\,\,\,\,\,\,\,\,\,\,\,\,\,\,\,\,\,\,\,\,\,\,\,\,\,\,\,\,\,\,\,\,\,\,\,\,\,\,\,\,\,\,\,\,\,\,\,\,\,\,\,\,\,\,\,\,\,\,\,\,\,\,\,\,\,\,\,\,\,\,
	
	1. The objective function $f\left( x \right) $ is smooth and convex on an open set containing $\mathcal{X}$, where $\mathcal{X}$ is a closed and convex set;
	
	2. The function $\psi $ is proper, essentially strictly convex and cofinite (see (\ref{mirr}));
	
	3. (Slater's condition) There exists a  vector $x\in\mathrm{int}\left( \mathcal{X} \right)$ that satisfies $Ax=b$.

The formulation (\ref{P1}) covers two important network optimization problems as follows:

\textbf{Scenario 1}: \textit{Distributed Constrained Consensus Problem (DCCP)}. Consider a network of $n$ agents over an undirected graph $\mathcal{G} $. The aims of it is to cooperate for seeking the minimum of the following problem:
\begin{equation}\label{P2}
	\begin{split}
		&\underset{x\in \mathbb{R} ^{nm}}{\min}  f\left( x \right) =\sum_{i=1}^n{f_i\left( x_i \right)},
		\\
		&\mathrm{s}.\mathrm{t}.\,\,Lx=0 \,\,(i.e.,x_i=x_j, i=1,...,n), x_i\in \mathcal{X} _i\subset \mathbb{R} ^m,
	\end{split}
\end{equation}
where $x=\mathrm{col}\left( x_1,...,x_n \right)\in \mathbb{R}^{nm}$, and $\mathcal{X} =\Pi _{i=1}^{n}\mathcal{X} _i\subset \mathbb{R}^{nm}$. $L=L_n\otimes I_m\in \mathbb{R}^{nm\times nm}$ and $L_n\in \mathbb{R}^{n\times n}$ is the Laplacian matirx of $\mathcal{G}$, $Lx=0$ is applied to ensure the consensus of $x_i=x_j, i,j=1,2,...,n$ in a distributed way since the agnet $i$ only accesses its local function $f_i: \mathcal{X}_i\rightarrow \mathbb{R}$ and the constraint $\mathcal{X} _i\in \mathbb{R}^m$.

%in the areas of optimization and machine learning because of the general expression. For instance, if the constrained set $\mathcal{X} =R^n$, the problem becomes the classical convex optimization with an affine equality constraint []--[] i.e.,
%\begin{remark}\label{R1}
%	The problem X is a universal distributed optimization model and has been extensively studied in the literatures [], it has many applications such as distributed sparse signal recovery [], wireless sensor network routing [] , distributed machine learning.
%\end{remark}
\textbf{Scenario 2}: \textit{Distributed Extended Monotropic Optimization (DEMO)}. Consider a network of $n$ agents reciprocating information over a graph $\mathcal{G}$. There exists a local objective function $f_i\left( x_i \right) : \mathcal{X} _i\rightarrow \mathbb{R}$ and a local feasible constraint set $\mathcal{X} _i \subset \mathbb{R}^{p_i}, i=1,...,n$. Let $x=\mathrm{col}\left( x_1,...,x_n \right) $, $\mathcal{X} =\Pi _{i=1}^{n}\mathcal{X} _i\subset \mathbb{R}^{\sum_{i=1}^n{p_i}}$, then, the DEMO problem is
\begin{equation}\label{P3}
	\begin{split}
		\underset{x\in \mathbb{R} ^{\sum_{i=1}^n{p_i}}}{\min} f\left( x \right) =\sum_{i=1}^n{f_i\left( x_i \right)},
		\,\,\mathrm{s}.\mathrm{t}. \,\,\sum_{i=1}^n{A_ix_i}=\sum_{i=1}^n{d_i},\,\, x_i\in \mathcal{X} _i,\,\, i=1,2,...,n,
	\end{split}
\end{equation}
where $A=\left[A_1,...,A_n \right] \in \mathbb{R}^{m\times \sum_{i=1}^n{p_i}}$ and $A_i\in \mathbb{R}^{m\times p_i}$.

To ensure the well-posedness of DCCP (\ref{P2}) and DEMO (\ref{P3}), some appropriate assumptions need to be made on them, which are fairly standard as follows:

{\bf Aassume 1: } \label{assume2}	
	1. The objective function $f\left( x \right)$ is $\sum_{i=1}^m{f_i\left( x_i \right)}$ and for all $i=1,...,n$, $f_i\left( x_i \right) $ is smooth and  convex on an open set containing $\mathcal{X}_i$, and $\mathcal{X}_i$ is a closed and convex set;
	
	2. For all $i=1,...,n$, $\psi_i $ is proper, essentially strictly convex and cofinite (see (\ref{mirr}));
	
	3. The communication graph $\mathcal{G}$ is connected and undirected;
	
	4. The Slater's condition of DCCP (\ref{P2}) and DEMO (\ref{P3}) is satisfied.

\subsection{APDMD for problem \eqref{P1} with smooth convex objective functions}\label{sec:main}

Inspired by the accelerated mirror descent in \cite{Krichene2015AcceleratedMD} and primal-dual dynamical approach in \cite{feijer2010stability}, we propose the following APDMD:
\begin{align}\label{PDM}
	\begin{cases}
		\dot{x}\left( t \right) =\frac{\alpha}{t}\left( \nabla \psi ^*\left( u\left( t \right) \right) -x\left( t \right) \right) ,
		\\
		\dot{u}\left( t \right) =-\frac{t}{\alpha}\left( \nabla f\left( x\left( t \right) \right) +\beta A^T\left( Ax\left( t \right) -b \right) +A^Tv\left( t \right) \right) ,
		\\
		\dot{\lambda}\left( t \right) =\frac{\alpha}{t}\left( v\left( t \right) -\lambda\left( t \right) \right), 
		\\
		\dot{v}\left( t \right) =\frac{t}{\alpha}\left( A\nabla \psi ^*\left( u\left( t \right) \right) -b \right) .
		\\
		x\left( t_0 \right) =x_0, u\left( t_0 \right) =u_0\,\,\mathrm{with}\,\,\nabla \psi ^*\left( u_0 \right) =x_0\in \mathcal{X}, \lambda \left( t_0 \right) =\lambda _0,v\left( t_0 \right) =v_0,
	\end{cases}    
\end{align} 
where $t\geq t_0>0$, $\beta>0$, and $\alpha \geq 2$. The illustration of the APDMD \eqref{PDM} is in \eqref{fig:APDMD} (left).
% ordinary differential equations (ODEs)

Note that the APDMD \eqref{PDM} can be equivalently transformed into the following second-order dynamical approach: 

\begin{align}\label{PDM1}
	\begin{cases}
		\ddot{x}\left( t \right) +\frac{\alpha +1}{t}\dot{x}\left( t \right) +\nabla ^2\psi ^*\left( \nabla \psi \left( x\left( t \right) +\frac{t}{\alpha}\dot{x}\left( t \right) \right) \right) 
		\\
		\,\,\,\,\,\,\,\,\,\,\,\,\,\times \left( \nabla f\left( x\left( t \right) \right) +\beta A^T\left( Ax\left( t \right) -b \right) +A^T\left( \lambda \left( t \right) +\frac{t}{\alpha}\dot{\lambda}\left( t \right) \right) \right) =0,
		\\
		\ddot{\lambda}\left( t \right) +\frac{\alpha +1}{t}\dot{\lambda}\left( t \right) -A\left( x\left( t \right) +\frac{t}{\alpha}\dot{x}\left( t \right) \right) +b=0,
		\\
		x\left( t_0 \right) =x_0\in \mathcal{X}, \dot{x}\left( t_0 \right) =\dot{x}_0, \lambda \left( t_0 \right) =\lambda _0, \dot{\lambda}\left( t_0 \right) =\dot{\lambda}_0,                                  \\
	\end{cases}
\end{align} 
where the Hessian term $\nabla ^2\psi ^*\left( \nabla \psi \left( x\left( t \right) +\frac{t}{\alpha}\dot{x}\left( t \right)  \right) \right)$ is nonlinear transformation. It applies to $\nabla f\left( x\left( t \right)  \right) +\beta A^T\left( Ax\left( t \right) -b \right) +A^T\left( \lambda \left( t \right) +\frac{t}{\alpha}\dot{\lambda} \left( t \right) \right) $ to guarantee the trajectory of $x$ is inside the feasible set $\mathcal{X}$ in the intuitive understanding (the rigorous proof process in \eqref{feasibility1}). The form in  \eqref{PDM1} allows us to understand more intuitively why the proposed \eqref{PDM} is called the primal-dual method that is with primal variable $x$ and dual variable ${\lambda}$.
\begin{figure}[!htbp]
	\centering
	\includegraphics[width=15cm,height=6cm]{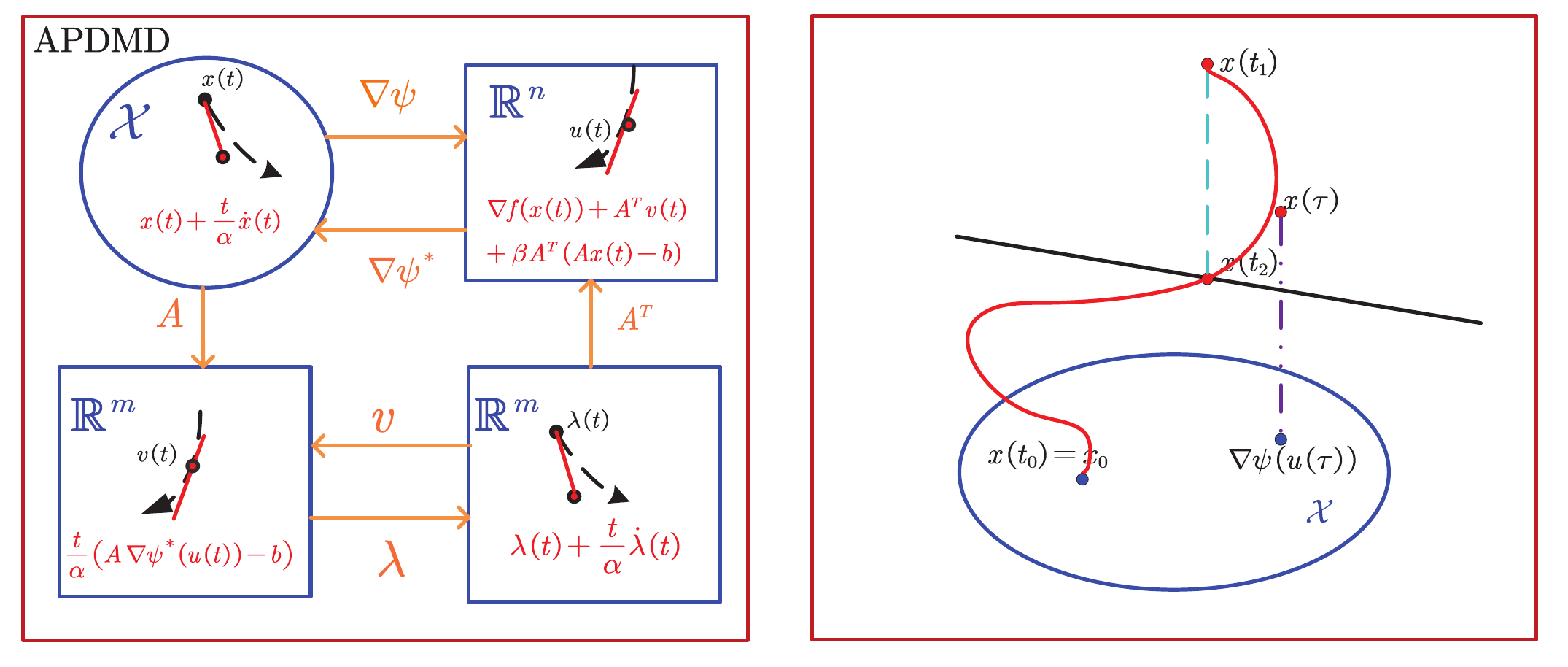}
	\caption{(left) Illustration of the APDMD \eqref{PDM}. (right) The feasibility of $x\left( t \right)$ to APDMD \eqref{PDM}.}
	\label{fig:APDMD}
\end{figure}

\subsection{Interpretation of the APDMD} The APDMD \eqref{PDM} (i.e., \eqref{PDM1}) can be interpreted from different perspectives: neurodynamic approach, Hamilton's system, game theory and control theory.
\begin{figure} [!htbp]
	\centering
	\includegraphics[width=15cm,height=11cm]{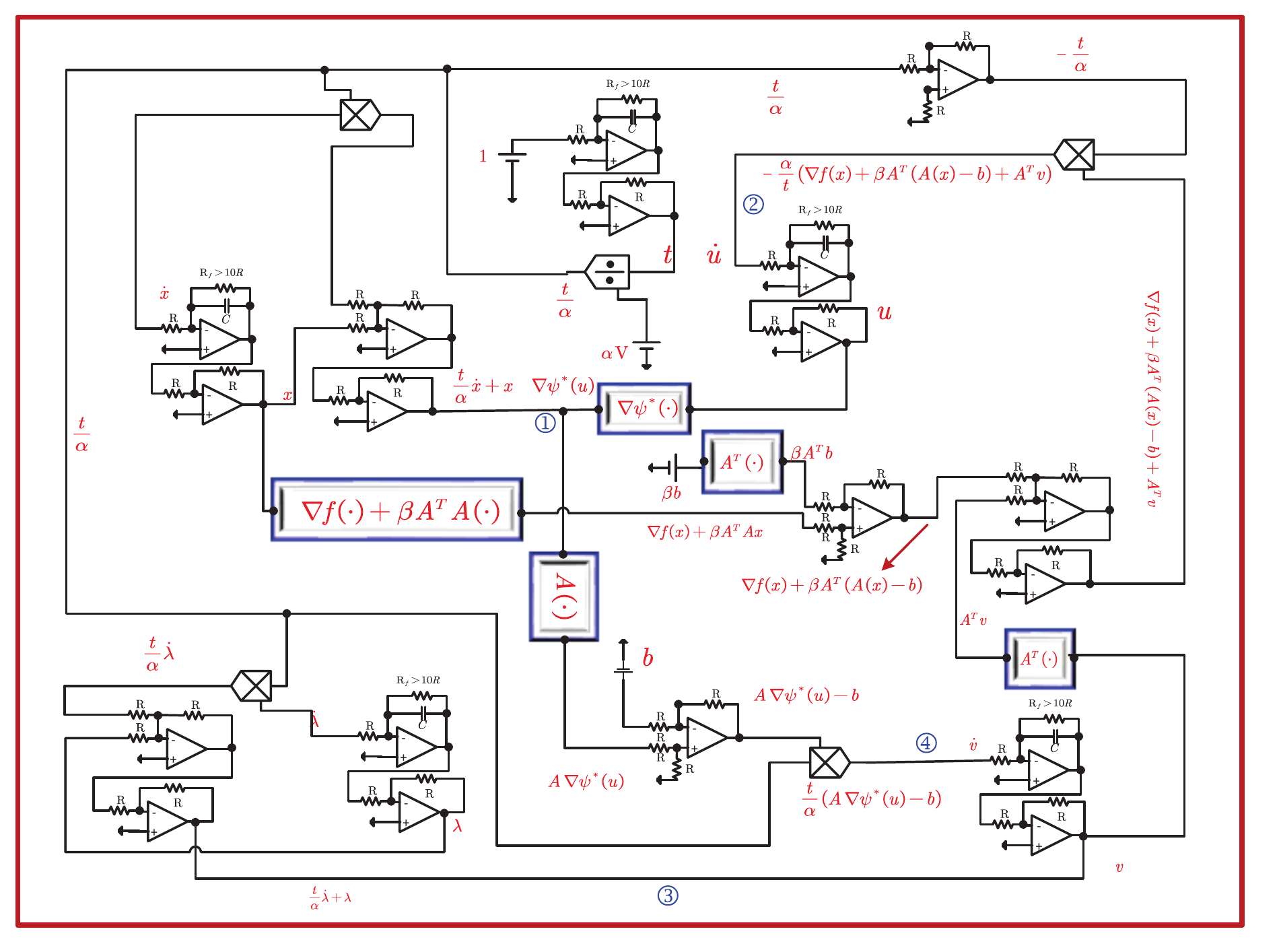}
	\caption{Circuit architecture diagram of APDMD \eqref{PDM} with
		\text{\textcircled{1}}$\frac{t}{\alpha}\dot{x}+x=\nabla \psi ^*\left( u \right)$, \textcircled{2}$\dot{u}=-\frac{t}{\alpha}\left( \nabla f\left( x \right) +\beta A^T\left( Ax-b \right) +A^Tv \right)$,
		\textcircled{3} $\frac{t}{\alpha}\dot{\lambda}+\lambda =v$, 
		\textcircled{4} $\dot{v}=\frac{t}{\alpha}\left( A\nabla \psi ^*\left( u \right) -b \right)$.}
	\label{fig:circuit}
\end{figure}
\begin{itemize}
\item \textbf{Neurodynamic approach perspective}. The APDMD \eqref{PDM} be regarded as a neurodynamic approach (describe the dynamic behavior of neurons), thus APDMD can be implemented by analog circuit like the Hopfield neural network in \cite{Hopfield1986ComputingWN}, as shown in Figure \eqref{fig:circuit}. The circuit in Figure\eqref{fig:circuit} is designed according to APDMD \eqref{PDM} by using analog adders, analog subtracters, analog multipliers, analog integrators , etc. When the circuit is turned on, the stable value of the voltage at position $x$ in Figure\eqref{fig:circuit} is the stable point to the APDMD \eqref{PDM} (i.e., an optimal solution to the problem \eqref{P1}). For more information about neurodynamic approaches and their circuits, please see \cite{kennedy1988neural}.
\end{itemize}

\begin{itemize}
\item \textbf{Hamiltonian system perspective}. The Hamiltonian-based system approach is used to design accelerated dynamical approaches for solving problem \eqref{uncon_optimization}, which has recently been widely studied in \cite{Diakonikolas2021GeneralizedMM,Wibisono2016AVP}, and it plays a key role in designing of APDMD \eqref{PDM} and the development of our Lyapunov analysis method.

Designing the first Hamiltonian (time-dependent) system 
\begin{align}
	&H_1\left( \bar{x}\left( t \right) ,u\left( t \right) ,\eta \left( t \right) \right) =\psi ^*\left( u\left( t \right) \right) +\varGamma \left( \eta \left( t \right) \right) \left( f\left( \frac{\bar{x}\left( t \right)}{\eta \left( t \right)} \right) \right. \nonumber
	\\
	&\,\,\,\,\,\,\,\,\,\,\,\,\,\,\,\,\,\,\,\,\,\,\,\,\,\,\,+\beta \left\| A\left( \frac{\bar{x}\left( t \right)}{\eta \left( t \right)} \right) -b \right\| ^2\left. +v\left( t \right) ^T\left( A\left( \frac{\bar{x}\left( t \right)}{\eta \left( t \right)} \right) -b \right) \right) \nonumber	
\end{align}
with $\bar{x}\left( t \right) =\eta \left( t \right) x\left( t \right) $, $\eta \left( t \right) =t^{\alpha}$ and $ \varGamma \left( \eta \left( t \right) \right) =\frac{\left( \eta \left( t \right) \right) ^{\frac{2}{\alpha}}}{\alpha ^2}$. 

The corresponding continuous-time dynamics of $H_1$ is 
\begin{align*}
\frac{d}{dt}\left( \bar{x}\left( t \right) \right) =&\dot{\eta}\left( t \right) \frac{d\bar{x}\left( t \right)}{d\eta \left( t \right)}=\dot{\eta}\left( t \right) \nabla _uH_1\left( \bar{x}\left( t \right) ,u\left( t \right) ,\eta \left( t \right) \right) =\dot{\eta}\left( t \right) \nabla \psi ^*\left( u\left( t \right) \right) 
\\
&\Rightarrow \dot{x}\left( t \right) =\frac{\alpha}{t}\left( \nabla \psi ^*\left( u\left( t \right) \right) -x\left( t \right) \right);
\end{align*}
\begin{align*}
\frac{d}{dt}u\left( t \right) =&\dot{\eta}\left( t \right) \frac{du\left( t \right)}{d\eta \left( t \right)}=-\varGamma \left( \eta \left( t \right) \right) \dot{\eta}\left( t \right) \nabla _{\bar{x}}H_1\left( \bar{x}\left( t \right) ,u\left( t \right) ,\eta \left( t \right) \right) 
\\
&\Rightarrow \dot{u}\left( t \right) =-\frac{\alpha}{t}\left( \nabla f\left( x\left( t \right) \right) +\beta A^T\left( Ax\left( x\left( t \right) \right) -b \right) +A^Tv\left( t \right) \right).
\end{align*}
Furthermore, setting the second Hamiltonian (time-dependent) system be  
\begin{align*}
H_2\left( \bar{\lambda}\left( t \right) ,v\left( t \right) ,\eta \left( t \right) \right) =\frac{1}{2}\left\| v\left( t \right) \right\| ^2-\varGamma \left( \eta \left( t \right) \right) \left( \frac{\bar{\lambda} \left( t \right)}{\eta \left( t \right)} \right) ^T\left( Au\left( t \right) -b \right),
\end{align*}	
where $\bar{\lambda}\left( t \right) =\eta \left( t \right) \lambda \left( t \right) $. 

The corresponding continuous-time dynamics of $H_2$ is 
\begin{align*}
\frac{d}{dt}\left( \bar{\lambda}\left( t \right) \right) =&\dot{\eta}\left( t \right) \frac{d\bar{\lambda}\left( t \right)}{d\eta \left( t \right)}=\dot{\eta}\left( t \right) \nabla _vH_2\left( \bar{\lambda}\left( t \right) ,v\left( t \right) ,\eta \left( t \right) \right) =\dot{\eta}\left( t \right) v\left( t \right) 
\\
&\Rightarrow \dot{\lambda}\left( t \right) =\frac{\alpha}{t}\left( v\left( t \right) -\lambda \left( t \right) \right) 
\\
\frac{d}{dt}v\left( t \right)=&\dot{\eta}\left( t \right) \frac{dv\left( t \right)}{d\eta \left( t \right)}=-\varGamma \left( \eta \left( t \right) \right) \dot{\eta}\left( t \right) \nabla _{\bar{\lambda}}H_2\left( \bar{\lambda}\left( t \right) ,v\left( t \right) ,\eta \left( t \right) \right) 
\\
&\Rightarrow \dot{v}\left( t \right) =-\frac{t}{\alpha}\left( b-Au\left( t \right) \right).
\end{align*}	
From the above conclusions, the APDMD \eqref{PDM} can be obtained directly. 

\item \textbf{Game theoretic standpoint} \cite{Attouch2021FastCO}. Let us consider $x\left( t \right)$ and $\lambda \left( t \right) $ as two players who compete with each other. Briefly, we identify players by their actions. We can see that each player anticipates the opponent's movement in APDMD \eqref{PDM1}. In the coupling term, the player $\lambda \left( t \right) $ takes account of the anticipated position of the player $x\left( t \right) $, i.e., $x\left( t \right) +\frac{\alpha}{t}\dot{x}\left( t \right) $. Conversely, for player $x\left( t \right)$, it takes account of the anticipated position of the player $\lambda \left( t \right)$,  i.e., $\lambda \left( t \right) +\frac{t}{\alpha}\dot{\lambda}\left( t \right) $.

\item \textbf{Control theoretic view}. The APDMD \eqref{PDM1} can also be associated with control theory and state derivative feedback. Let $\chi \left( t \right) =\mathrm{col}\left( x\left( t \right),\lambda \left( t \right) \right)$, the APDMD \eqref{PDM1} can be written in the following formulation $$\ddot{\chi}\left( t \right) +\frac{\alpha +1}{t}\dot{\chi}\left( t \right) =\varUpsilon \left( t,\chi \left( t \right) ,\dot{\chi}\left( t \right) \right),$$
with an operator $\varUpsilon$, i.e., a feedback control term which takes the constraint into account. It is a function of the state $\chi \left( t \right)$, its derivative $\dot{\chi}\left( t \right)$ and time $t$. For a comprehensive understanding of state derivative feedback, the readers can consult reference \cite{Michiels2009StabilizabilityAS}.
\end{itemize}

%\begin{lemma}% ?????????????
%	Under Assumption X, the equilibrium point $\left( x^*,u^*,\lambda ^* \right) $ of PDMD \eqref{PDM1} satisfies 
%\begin{subequations}
%\begin{align}
%			\nabla ^2\psi ^*\left( u^* \right) \left( \nabla f\left( x^* \right) +A^T\lambda ^* \right) =&0, \label{KKT:1}\\
%			\nabla \psi ^*\left( u^* \right) =&x^*,\label{KKT:2}\\
%			Ax^*=&b\label{KKT:3},         
%\end{align}
%\end{subequations}
%where $ x^* \in \mathcal{X}$ is an optimal solution of problem \eqref{P1}. 
%\begin{proof}
%According to the \textit{Feasibility} in \eqref{Lem3.3}, we have $\nabla \psi ^*\left( u^* \right) =x^*$ i.e., \eqref{KKT:2} holds. In addition, from PDMD \eqref{PDM1}, we further get $\nabla ^2\psi ^*\left( u^* \right) \left( \nabla f\left( x^* \right) +A^T\lambda ^* \right) =0$ and $\nabla \psi ^*\left( u^* \right) =x^*$. Next XXX.
%\end{proof}
%\end{lemma}

\subsection{Feasibility, existence and uniqueness of strong global solution to APDMD} \label{solution1} In this subsection, we illustrate the feasibility, the existence and uniqueness of the strong global solution $x\left(t\right)$ for APDMD  (\ref{PDM}) by the  Cauchy-Lipschitz-Picard theorem in \cite{Bolte2003SurDS}.

\begin{lemma}\label{feasibility1}
	For any initial values $\left(x\left( t_0 \right) ,u\left( t_0 \right), \lambda \left( t_0 \right), \upsilon \left( t_0 \right)\right) \in \mathcal{X} \times \mathbb{R}^n\times \mathbb{R}^m\times \mathbb{R}^m$, the variable $x\left( t \right) \in  \mathcal{X} , \forall \,\,t\geq t_0> 0$, i.e., the solution $x\left( t \right)$ is feasible.
\end{lemma}
	\begin{proof}	
		Inspired by the work in \cite{krichene2016lyapunov}, we provid a rigorous proof for the feasibility by contradiction. Suppose there exist $t_1>0$ and $x\left(t_1 \right) \notin \mathcal{X}$. Since $\mathcal{X}$ is closed and convex, by the separation theorem, there exists a hyperplane that strictly separates $x\left( t_1 \right) $ and the set $\mathcal{X}$. That is, there exist $\omega$, $a\in \mathbb{R}^n$ such that $\left( x\left( t_1 \right) -a \right) ^T\omega >0$ and $\left( x-a \right) ^T\omega <0, \forall x\in \mathcal{X} $. Let $d\left( x\left( t \right) \right) =\left( x-a \right) ^T\omega$. Note that the trajectory $x\left( t \right) $ is continuous, then $t\rightarrow d\left( x\left( t \right) \right)$ is continuous, and $d\left( x\left( t_0 \right) \right) =\left( x\left( t\right)-a \right) ^T \omega <0$, $d\left( x\left( t_1 \right) \right)=\left( x\left( t_1 \right)-a \right) ^T\omega >0$, thus, there is a $t_2$ such that $d\left( x\left( t_2 \right) \right) =0$, i.e., $d\left( x\left( t \right) \right) >0, t\in \left( t_2,t_1 \right] $ and $d\left( x\left( t \right) \right) <0$, $t\in \left[ t_0,t_2 \right) $, which implies $t_2=\mathrm{sub}\left\{ t:d\left( x\left( t \right) \right) \leq0 \right\} $. We have $d\left( x\left( t_1 \right) \right) -d\left( x\left( t_0 \right) \right) >0$, according to Taylor's theorem and $\dot{x}\left( t \right)=\frac{\alpha}{t}\left( \nabla \psi ^*\left( u \left( t \right)\right) -x \left( t \right)\right)$, there exists $\tau \in \left[ t_2,t_1 \right]$, such that 
		\begin{align*}
			d\left( x\left( t_1 \right) \right) -d\left( x\left( t_2 \right) \right) =&\dot{d}\left( x\left( \tau \right) \right) 
			\,\,                        =\left( \dot{x}\left( \tau \right) \right) ^T\omega =\omega ^T\frac{\alpha}{\tau}\left( \nabla \psi ^*\left( u\left( \tau \right) \right) -x\left( \tau \right) \right)
			\\
			\,\,                        =&\frac{\alpha}{\tau}\left( d\left( \nabla \psi ^*\left( u\left( \tau \right) \right) \right) -d\left( x\left( \tau \right) \right) \right) <0,
		\end{align*}
		since $d\left( \nabla \psi ^*\left( u_0 \right) \right) \in \mathcal{X} $.  It leads to a contradiction, which concludes the proof and is shown in Figure. \ref{fig:APDMD} (right).		
	\end{proof}	
		
\begin{definition}\label{stro_solu}
	Let $x:\left[ t_0,+\infty \right) \rightarrow \mathcal{X}$, $\lambda: \left[ t_0,+\infty \right) \rightarrow \mathbb{R}^m$ with the corresponding product space structure $\mathcal{X} \times \mathbb{R}^m$ and $t_0>0$. $\left( x,u,\lambda ,v \right) $ is a strong global solution of APDMD \eqref{PDM} if it satisfies the following properties:
	
	(i) $x,u:\left[ t_0,+\infty \right) \rightarrow \mathbb{R}^n$, $\lambda ,v:\left[ t_0,+\infty \right) \rightarrow \mathbb{R}^m$ are locally absolutely continuous,
	
	(ii) The APDMD holds for almost every $t\geq t_0$,
	
	(iii) $x\left( t_0 \right) =x_0\in \mathcal{X}$, $u\left( t_0 \right) =u_0$, $\lambda \left( t_0 \right) =\lambda _0$, $v \left( t_0 \right) =v _0$.

	A mapping $x:\left[ t_0,+\infty \right) \rightarrow \mathbb{R}^n$ is called locally absoutely continuous if it is absolutely continuous on every compact interval $\left[ t_0,\mathcal{T} \right] $ with $\mathcal{T} >t_0$. For the absolutely continuous function $x:\left[ t_0,+\infty \right) \rightarrow \mathbb{R}^n$, it has the following equivalent characterizations:
	
	(a) There exists $\mathsf{x} :\left[ t_0,T \right) \rightarrow R^n$  an integrable function, such that 
	$$x\left( t \right) =x\left( t_0 \right) +\int_{t_0}^t{\mathsf{x} \left( s \right) ds}, \forall t\in \left[ t_0,\mathcal{T} \right];$$
	
	(b) $x$ is a continuous function and its distribuional derivative is Lebesque integrable on the interval  $\left[ t_0,\mathcal{T} \right]$;
	
	(c) For every $\varepsilon >0$, there exists $\pi >0$, such that for every finite family $I_k=\left( a_k,b_k \right) $ from $\left[ t_0,T \right] $, the following  implication is valid:
	$$\left[ I_k\cap I_j=\oslash \,\,\mathrm{and} \sum_k{\left| b_k-a_k \right|}<\pi \right] \Rightarrow \left[ \sum_k{\left\| x\left( b_k \right) -x\left( a_k \right) \right\| <\varepsilon} \right].$$
\end{definition}
\begin{theorem}\label{The4.4}	
	For any initial values $\left( x\left( t_0 \right) ,u\left( t_0 \right), \lambda \left( t_0 \right), \upsilon \left( t_0 \right) \right)\in \mathcal{X} \times \mathbb{R}^n\times \mathbb{R}^m\times \mathbb{R}^m$, there exists a unique strong global solution of the APDMD.
\end{theorem}
	\begin{proof}
		Let $Y\left( t \right) =\left( x\left( t \right) ,u\left( t \right) ,\lambda \left( t \right) ,v\left( t \right) \right) $, then the APDMD can be rewritten as follows:
	\begin{equation}\label{ex-un}
		\begin{cases}
			\dot{Y}\left( t \right) =F\left( t,Y\left( t \right) \right) \,\,                    \\
			Y\left( t_0 \right) =\left( x\left( t_0 \right) ,u\left( t_0 \right) , \lambda \left( t_0 \right), \upsilon \left( t_0 \right) \right),
		\end{cases}
	\end{equation}
	where $F:\left[ t_0,+\infty \right) \times \mathcal{X} \times \mathbb{R}^n\times \mathbb{R} ^m\times \mathbb{R} ^m, F\left( t,Y \right) =\left( \frac{\alpha}{t}\left( \nabla \psi ^*\left( u \right) -x \right) , \right. -\frac{t}{\alpha}\left( \nabla f\left( x \right) \right. 
	\\
	\left. +\beta A^T\left( Ax-b \right) -A^Tv \right) ,\frac{\alpha}{t}\left( v-\lambda \right) ,\frac{t}{\alpha}\left( A\nabla \psi ^*\left( u \right) -b \right) \left. ,\frac{\alpha}{t}\left( v-\lambda \right) ,\frac{t}{\alpha}\left( A\nabla \psi ^*\left( u \right) -b \right) \right) $.
	
	According to the Cauchy-Lipschitz-Picard theorem, there exists a unique strong global solution for APDMD (\ref{PDM}), i.e., (\ref{PDM1}), if the following conditions \textbf{(I)} and \textbf{(II)} are fulfilled.
	
	\textbf{(I)}: For every $t\in \left[ t_0,+\infty \right) $, the mapping $F\left( t,\cdot \right) $ is $\mathfrak{l} \left( t \right)$-Lipschitz continuous and $\mathfrak{l} \left( \cdot \right) \in \mathrm{L}_{\mathrm{loc}}^{1}\left( \left[ t_0,+\infty \right) \right) $.
	
	\textbf{(II)}: For any $Y\in \mathcal{X} \times \mathbb{R}^n\times \mathbb{R}^m\times \mathbb{R}^m$, we have $F\left( \cdot ,Y \right) \in \mathrm{L}_{\mathrm{loc}}^{1}\left( \left[ t_0,+\infty \right),\mathcal{X} \times \mathbb{R}^n,\right. 
	\\
	\left. \times \mathbb{R}^m\times \mathbb{R}^m \right)$.
	
	The proof of \textbf{(I)}. Let $t\in \left[ t_0,+\infty \right)$ be fixed and utilize the Lipschitz continuous of $\nabla \psi ^*$, $\nabla f$ and $\left\| \mathsf{X} +\mathsf{Y} \right\| ^2\leq 2\left\| \mathsf{X} \right\| ^2+2\left\| \mathsf{Y} \right\| ^2$ for any vectors $\mathsf{X}$ and $\mathsf{Y}$, then, for any $Y$ and $\hat{Y}$, we have 
	\begin{align*}
		&\left\| F\left( t,Y \right) -F\left( t,\hat{Y} \right) \right\| 
		\\
		\le& \left( \frac{t^2}{\alpha ^2} \right. \left( 4\beta ^2\delta _{\max}\left( A^TA \right) ^2+4\mathfrak{l} _{f}^{2}+\delta _{\max}\left( A^TA \right) \mathfrak{l} _{\psi ^*}^{2} \right. 
		\left. \,\,\,\left. +2\delta _{\max}\left( A^TA \right) \right) +\frac{\alpha ^2\left( 6+2\mathfrak{l} _{\psi ^*}^{2} \right)}{t^2} \right) ^{\frac{1}{2}}\left\| Y-\hat{Y} \right\|,
	\end{align*}
	where $\delta _{\max}\left(A\right) $ is the maximum singular value of matrix $A$ and the inequality holds from the Lipschitz continuous properties of $\nabla f $ and $\nabla \psi ^*$ that have Lipschitz constants $\mathfrak{l} _f$ and $\mathfrak{l} _{\psi ^*}$. 
	Let $\mathfrak{l} \left( t \right)$ be $\left( \frac{t^2\left( 4\beta ^2\delta _{\max}\left( A^TA \right) ^2+4\mathfrak{l} _{f}^{2}+\delta _{\max}\left( A^TA \right) \mathfrak{l} _{\psi ^*}^{2}+2\delta _{\max}\left( A^TA \right) \right)}{\alpha ^2} \right. 
	\left. +\frac{\alpha ^2\left( 6+2\mathfrak{l} _{\psi ^*}^{2} \right)}{t^2} \right) $, one has 
	$\left\| F\left( t,Y \right) -F\left( t,\hat{Y} \right) \right\| \leq \mathfrak{l} \left( t \right) \left\| Y-\hat{Y} \right\|.$ 
	
	Note that $\mathfrak{l} \left( t \right) $ is continuous on $\left[ t_0,+\infty \right) $. Hence $\mathfrak{l} \left( \cdot \right) $ is integrable on $\left[ t_0,\mathcal{T} \right] $ for all $0<t_0<\mathcal{T} <+\infty$.
	
	The proof of \textbf{(II)}. For arbitrary $Y\in \mathcal{X} \times \mathbb{R}^n\times \mathbb{R}^m\times \mathbb{R}^m$ and $0<t_0<\mathcal{T}$, we have 
	\begin{align*}
		&\int_{t_0}^{\mathcal{T}}{\left\| F\left( t,Y\left( t \right) \right) \right\|}dt
		\\
		=&\int_{t_0}^{\mathcal{T}}{\begin{array}{c}
						\begin{array}{c}
								\left( \frac{\alpha ^2}{t^2}\left\| \nabla \psi ^*\left( u\left( t \right) \right) -x\left( t \right) \right\| ^2+\frac{\alpha ^2}{t^2}\left\| v\left( t \right) -\lambda \left( t \right) \right\| ^2 \right.\\
							\end{array}\\
				\end{array}}
				\\
				&+\left. \frac{t^2}{\alpha ^2}\left\| A\left( \nabla \psi ^*\left( u\left( t \right) \right) -x^* \right) \right\| ^2+\frac{t^2}{\alpha ^2}\left\| \nabla f\left( x\left( t \right) \right) +\beta A^TA\left( x\left( t \right) -x^* \right) -A^Tv\left( t \right) \right\| ^2 \right) ^{\frac{1}{2}}dt
				\\
				\leq& \int_{t_0}^{\mathcal{T}}{\begin{array}{c}
								\begin{array}{c}
										\left( \frac{2\alpha ^2}{t^2}\left\| \nabla \psi ^*\left( u\left( t \right) \right) \right\| ^2 \right. +\frac{4t^2\beta ^2\delta _{\max}\left( A^TA \right) ^2}{\alpha ^2}\left\| x\left( t \right) \right\| ^2\\
									\end{array}\\
						\end{array}}
						\\
						&+\frac{2\alpha ^2}{t^2}\left\| v\left( t \right) \right\| ^2+\frac{2\alpha ^2}{t^2}\left\| \lambda \left( t \right) \right\| ^2+\frac{2\alpha ^2}{t^2}\left\| x\left( t \right) \right\| ^2+\frac{2t^2\delta _{\max}\left( A^TA \right)}{\alpha ^2}\left\| \nabla \psi ^*\left( u\left( t \right) \right) \right\| ^2
						\\
						&+\frac{2t^2\delta _{\max}\left( A^TA \right)}{\alpha ^2}\left\| x^* \right\| ^2+\frac{4t^2}{\alpha ^2}\left\| \nabla f\left( x\left( t \right) \right) \right\| ^2+\frac{4t^2\delta _{\max}\left( A^TA \right)}{\alpha ^2}\left\| v\left( t \right) \right\| ^2
						\\
						&\left. \begin{array}{c}
							\begin{array}{c}
									+\frac{4t^2\beta ^2\delta _{\max}\left( A^TA \right) ^2}{\alpha ^2}\left\| x^* \right\| ^2\\
								\end{array}\\
						\end{array} \right) ^{\frac{1}{2}}dt
						\\
						\leq& \sqrt{\left\| \nabla \psi ^*\left( u \right) \right\| ^2+\left\| x \right\| ^2+\left\| \lambda \right\| ^2+\left\| v \right\| ^2+\left\| x^* \right\| ^2+\left\| \nabla f\left( x\left( t \right) \right) \right\| ^2}
						\\
						&\times \left( \frac{8\alpha ^2}{t^2}+\frac{2t^2}{\alpha ^2}\left( 4\delta _{\max}\left( A^TA \right) +4\beta ^2\delta _{\max}\left( A^TA \right) ^2+2 \right) \right),
					\end{align*}
					and the conclusion holds by the continuity of the following function 
					$$t\rightarrow \left( \frac{8\alpha ^2}{t^2}+\frac{2t^2}{\alpha ^2}\left( 4\delta _{\max}\left( A^TA \right) +4\beta ^2\delta _{\max}\left( A^TA \right) ^2+2 \right) \right) ^{\frac{1}{2}}  \text{on}\,\,\left[ t_0,\mathcal{T} \right].$$
					
In view of the above statements \textbf{(I)} and \textbf{(II)}, the existence and uniqueness of a strong global solution for the dynamical system (\ref{ex-un}) can be obtained. This leads directly to  the existence and uniqueness of a strong global solution for APDMD.
\end{proof}

\subsection{Convergence rate of APDMD} With the help of the Lyapunov analysis method with Bregman divergence function, we will illustrate the accelerated convergence properties of the APDMD \eqref{PDM} as follows.

\begin{theorem} \label{Theorem 2.4}
	Suppose Assumption 3.1 holds. Let $\left( x\left( t \right) ,\lambda \left( t \right) \right) $ and $\left( x^*,\lambda ^* \right)$ be the solution trajectory and optimal solution for APDMD \eqref{PDM1} and problem \eqref{P1}, respectively. Then for any $\left( x\left( t_0 \right) ,u\left( t_0 \right) ,\lambda\left( t_0 \right) ,\upsilon\left( t_0 \right) \right)  \in \mathcal{X} \times \mathbb{R}^n\times \mathbb{R} ^m\times \mathbb{R} ^m$ with $\alpha \geq2$, the following statements are true:
	
	\textit{(I)}: The trajectories $x\left( t \right)$ and $\lambda \left( t \right)$ of APDMD \eqref{PDM1} are bounded for any $t\geq t_0>0$.
	
	(II): Let $\mathcal{L} _{\beta}\left( x\left( t \right) ,\lambda \left( t \right) \right) =f\left( x\left( t \right) \right) +\lambda \left( t \right) ^T\left( Ax\left( t \right) -b \right) +\frac{\beta}{2}\left\| Ax\left( t \right) -b \right\| ^2$ with $\beta \geq 0$. Then, we have 
	\begin{subequations}
		\begin{align}
			&\mathcal{L} _{\beta}\left(x\left( t \right),\lambda ^* \right) -\mathcal{L} _{\beta}\left( x^*,\lambda\left( t \right) \right) \leq \frac{\alpha ^2V\left( t_0 \right)}{t^2},\label{4.6}
			\\
			&\left\| Ax\left( t \right) -b \right\| ^2\leq \frac{2\alpha ^2V\left( t_0 \right)}{\beta t^2}, \label{4.61}
			\\
			&\int_0^{+\infty}{t\left( \mathcal{L} _{\beta}\left( x\left( t \right) ,\lambda ^* \right) -\mathcal{L} _{\beta}\left( x^*,\lambda \left( t \right) \right) \right) dt}<+\infty,\label{4.7}
			\\
			&\int_{t_0}^{+\infty}{t\left\| Ax\left( t \right) -b \right\| ^2}dt<+\infty.\label{4.71}
		\end{align}
	\end{subequations}
	(III): Let $\lambda ^*=\begin{cases}
		0, \,\,\,\,\,\,\,\,\,\,\,\,\,\,\,\,\,\,\,\,\,  if\,\,Ax\left( t \right) -b=0,\\
		\frac{Ax\left( t \right) -b}{\left\| Ax\left( t \right) -b \right\|},if\,\,Ax\left( t \right) -b\ne 0,\\
	\end{cases}$. For any $t\geq t_0>0$, it follows that 
	\begin{subequations}
		\begin{align}
			&0\le f\left( x\left( t \right) \right) -f\left( x^* \right) +\left\| Ax\left( t \right) -b \right\| +\frac{\beta}{2}\left\| Ax\left( t \right) -b \right\| ^2\le \frac{\alpha ^2V\left( t_0 \right)}{t^2},\label{4.8}
			\\
			&-\frac{\alpha \sqrt{V\left( t_0 \right)}}{t}\leq f\left( x\left( t \right) \right) -f\left( x^* \right) \leq\frac{\alpha ^2V\left( t_0 \right)}{t^2}.\label{4.9}
		\end{align}
	\end{subequations}
\end{theorem}
\begin{proof}
%To lighten notation, we drop the dependence on the time variable $t$. 
Consider a Lyapunov function $V:\left[ t_0, T \right) \rightarrow \left[ 0,+\infty \right) $ which is given by 
\begin{align}\label{V}
V\left( t \right) =&V_1\left( t \right) +V_2\left( t \right) +V_3\left( t \right),
\end{align}
where
\begin{equation*}
\begin{cases}
  V_1\left( t \right) =\frac{t^2}{\alpha ^2}\left( \mathcal{L} _{\beta}\left( x\left( t \right) ,\lambda ^* \right) -\mathcal{L} _{\beta}\left( x^*,\lambda \left( t \right) \right) \right)\\
	\,\,\,\,\,\,\,\,\,\,\,\,\,\,\,\,=\frac{t^2}{\alpha ^2}\left( f\left( x\left( t \right) \right) -f\left( x^* \right) +\frac{\beta}{2}\left\| Ax\left( t \right) -b \right\| ^2+A^T\lambda ^*\left( x\left( t \right) -x^* \right) \right) ,\\
		V_2\left( t \right) =D_{\psi ^*}\left( u\left( t \right) ,u^* \right) =\psi ^*\left( u\left( t \right) \right) -\psi ^*\left( u^* \right) -\nabla \psi ^*\left( u^* \right) ^T\left( u\left( t \right) -u^* \right) ,\\
		V_3\left( t \right) =D_h\left( v\left( t \right) ,\lambda ^* \right) =\frac{1}{2}\left\| v\left( t \right) -\lambda ^* \right\| ^2,\\
\end{cases}
\end{equation*}
the $D_{\psi ^*}\left( u, u^* \right)$, $D_h\left( v,\lambda^* \right) $ are two Bregman divergences associated with suitable functions $\psi ^*$ and $h$ that are determined by different requirements, respectively. 
%\begin{align}\label{V1} %(???)
%V_1\left( t \right) =&\frac{t^2}{\alpha ^2}\left( \mathcal{L} _{\beta}\left( x,\lambda ^* \right) -\mathcal{L} _{\beta}\left( x^*,\lambda \right) \right) 
%\\
%=&\frac{t^2}{\alpha ^2}\left( f\left( x \right) -f\left( x^* \right) +\frac{\beta}{2}\left\| Ax-b \right\| ^2+A^T\lambda ^*\left( x-x^* \right) \right)\nonumber, 
%\\
%V_2\left( t \right)=&D_{\psi ^*}\left( u,u^* \right) =\psi ^*\left( u \right) -\psi ^*\left( u^* \right) -\nabla \psi ^*\left( u^* \right) ^T\left( u-u^* \right),
%\\
%V_3\left( t \right) =&D_h\left( v,\lambda ^* \right) =\frac{1}{2}\left\| v-\lambda ^* \right\| ^2,
%\end{align}

Note that the Laypunov function $V_1\left( t \right)$ is positive and radially unbounded since 
$\mathcal{L} _{\beta}\left( x\left( t \right) ,\lambda ^* \right) -\mathcal{L} _{\beta}\left( x^*,\lambda \left( t \right) \right) \geq 0, \forall \left( x\left( t \right) ,\lambda \left( t \right) \right) \in \mathcal{X} \times \mathbb{R} ^m $ in \eqref{sd}, $\alpha \geq 2$ and $t\geq t_0>0$. For $V_2\left( t \right)$ and  $V_3\left( t \right)$, since $\psi ^*\left( \cdot \right)$ and $h\left( \cdot \right) =\frac{1}{2}\left\| \cdot \right\| ^2$ are convex \textit{(Assumption 3.1 holds}), we get $D_{\psi ^*}\left( u\left( t \right),u^* \right) \geq 0, D_h\left( v\left( t \right),\lambda ^* \right) \geq 0$, moreover, $\left\| u\left( t \right) \right\| \rightarrow +\infty , V_2\left( t \right) \rightarrow +\infty $ and $\left\| v\left( t \right) \right\| \rightarrow +\infty , V_3\left( t \right) \rightarrow +\infty $. Thus, both $V_2\left( t \right)$ and  $V_3\left( t \right)$ are positive and radially unbounded.

The derivatives of $V_1\left( t \right)$, $V_2\left( t \right)$ and $V_3\left( t \right) $ along the trajectory of APDMD with $x\left( t \right) \in \mathcal{X} $ are given by
\begin{equation}\label{V1}
\begin{split}
\dot{V}_1\left( t \right) =&\frac{2t}{\alpha ^2}\left( \mathcal{L} _{\beta}\left( x\left( t \right) ,\lambda ^* \right) -\mathcal{L} _{\beta}\left( x^*,\lambda \left( t \right) \right) \right) +\left( \nabla \psi ^*\left( u\left( t \right) \right) -x\left( t \right) \right) ^T
\\
\,\,&\times \frac{t}{\alpha}\left( \nabla f\left( x\left( t \right) \right) +\frac{\beta}{2}A^T\left( Ax\left( t \right) -b \right) +A^T\lambda ^* \right) 
\\
=&\frac{2t}{\alpha ^2}\left( f\left( x\left( t \right) \right) -f\left( x^* \right) +\frac{\beta}{2}\left\| Ax\left( t \right) -b \right\| ^2+\left( A^T\lambda ^* \right) ^T\left( x\left( t \right) -x^* \right) \right) 
\\
&+\frac{t}{\alpha}\left( \nabla f\left( x\left( t \right) \right) +\beta A^TA\left( x\left( t \right) -x^* \right) +A^T\lambda ^* \right) ^T
\\
&\times \left( \nabla \psi ^*\left( u\left( t \right) \right) -\nabla \psi ^*\left( u^* \right) \right) +\left( \nabla \psi ^*\left( u^* \right) -x\left( t \right) \right) ^T
\\
&+\frac{t}{\alpha}\left( \nabla f\left( x\left( t \right) \right) +\beta A^TA\left( x\left( t \right) -x^* \right) +A^T\lambda ^* \right),
\end{split}
\end{equation}
\begin{equation}\label{V2}
\begin{split}
\dot{V}_2\left( t \right) =&\left( \nabla \psi ^*\left( u\left( t \right) \right) -\nabla \psi ^*\left( u^* \right) \right) ^T\dot{u}\left( t \right) 
\\
=&-\frac{t}{\alpha}\left( \nabla \psi ^*\left( u\left( t \right) \right) -\nabla \psi ^*\left( u^* \right) \right) ^T
\\
&\times \left( \nabla f\left( x\left( t \right) \right) +\beta A^T\left( Ax\left( t \right) -b \right) +A^Tv\left( t \right) \right) 
\\
=&-\frac{t}{\alpha}\left( \nabla \psi ^*\left( u\left( t \right) \right) -\nabla \psi ^*\left( u^* \right) \right) ^T
\\
&\times \left( \nabla f\left( x\left( t \right) \right) +\beta A^TA\left( x\left( t \right) -x^* \right) +A^T\lambda ^* \right) 
\\
&-\frac{t}{\alpha}\left( \nabla \psi ^*\left( u\left( t \right) \right) -\nabla \psi ^*\left( u^* \right) \right) ^T\left( A^Tv\left( t \right) -A^T\lambda ^* \right),
\end{split}
\end{equation}
\begin{align}\label{V3}
\dot{V}_3\left( t \right) =\left( v\left( t \right) -\lambda ^* \right) ^T\dot{v}=\frac{t}{\alpha}\left( v\left( t \right) -\lambda ^* \right) ^TA\left( \nabla \psi ^*\left( u\left( t \right) \right) -x^* \right).
\end{align}

Adding \eqref{V1}, \eqref{V2} and  \eqref{V3} together, and using $Ax^*=b$, we have
\begin{equation}\label{Vend}
\begin{split}
\dot{V}\left( t \right) =&\dot{V}_1\left( t \right) +\dot{V}_2\left( t \right) +\dot{V}_3\left( t \right) 
\\
=&\frac{2t}{\alpha ^2}\left( f\left( x\left( t \right) \right) -f\left( x^* \right) +\frac{\beta}{2}\left\| Ax\left( t \right) -b \right\| ^2+\left( A^T\lambda ^* \right) ^T\left( x\left( t \right) -x^* \right) \right) 
\\
&+\frac{t}{\alpha}\left( \nabla f\left( x\left( t \right) \right) +\frac{\beta}{2}A^T\left( Ax\left( t \right) -b \right) +A^T\lambda ^* \right) ^T\left( x^*-x\left( t \right) \right) 
\\
\le &\frac{2t}{\alpha ^2}\left( \mathcal{L} _{\beta}\left( x\left( t \right) ,\lambda ^* \right) -\mathcal{L} _{\beta}\left( x^*,\lambda \right) \right) -\frac{t}{\alpha}\left( \mathcal{L} _{\beta}\left( x^*,\lambda ^* \right) -\mathcal{L} _{\beta}\left( x\left( t \right) ,\lambda ^* \right) \right) 
\\
&-\frac{\beta t}{2\alpha}\left( Ax\left( t \right) -b \right) A\left( x^*-x\left( t \right) \right) 
\\
=&-t\frac{\alpha -2}{\alpha}\left( \mathcal{L} _{\beta}\left( x\left( t \right) ,\lambda ^* \right) -\mathcal{L} _{\beta}\left( x^*,\lambda \left( t \right) \right) \right)-\frac{\beta t}{2\alpha}\left\| Ax\left( t \right) -b \right\| ^2\le 0,
\end{split}
\end{equation}
where the second equation holds from $\nabla \psi ^*\left( u^* \right) =x^*$, the first inequality is satisfied due to the convexity of $\mathcal{L} _{\beta}\left(x\left( t \right),\lambda ^* \right) $ of $x\left( t \right)$ with fixed $\lambda ^*$, the third equation holds because of $\mathcal{L} _{\beta}\left(x^*,\lambda\left( t \right) \right) =\mathcal{L} _{\beta}\left( x^*,\lambda ^* \right) $ and the last inequality is established since $\alpha \geq 2$ and $\mathcal{L}_{\beta}\left( x\left( t \right),\lambda^* \right) -\mathcal{L} _{\beta}\left(x^*,\lambda\left( t \right)\right) \geq 0$. 

Since the Lyapunov function $V\left( t \right)$ is positive and radically unbounded with any $t\geq t_0>0$. This implies that the trajectories of $x\left( t \right)$ and $\lambda \left( t \right)$ are bounded for any $t\geq t_0>0$. The proof of (I) therefore completed.

\textit{(II)}: From \eqref{Vend}, we have that $\dot{V}\left( t \right)\leq0$, i.e., ${V}\left( t \right)$ is nonincreasing on $\left[ t_0,+\infty \right)$. Thus, for any $t\geq t_0>0$ it yields to $\frac{t^2}{\alpha ^2}\left( \mathcal{L} _{\beta}\left( x\left( t \right),\lambda ^* \right) -\mathcal{L} _{\beta}\left( x^*,\lambda\left( t \right) \right) \right) =V_1\left( t \right) \leq V\left( t \right) \leq V\left( t_0 \right)$, i.e., $\mathcal{L} _{\beta}\left( x,\lambda ^* \right) -\mathcal{L} _{\beta}\left( x^*,\lambda \right) \leq \frac{\alpha ^2V\left( t_0 \right)}{t^2}$, thus the condition \eqref{4.6} holds. In addition, the \eqref{4.6} further implies $\left\| Ax\left( t \right) -b \right\| ^2\leq \frac{2\alpha ^2V\left( t_0 \right)}{\beta t^2}$, i.e., \eqref{4.61} is satisfied.

Since $\dot{V}\left( t \right) \leq -t\frac{\alpha -2}{\alpha ^2}\left( \mathcal{L} _{\beta}\left( x\left( t \right) ,\lambda ^* \right) -\mathcal{L} _{\beta}\left( x^*,\lambda \left( t \right) \right) \right) -\frac{\beta t}{2\alpha}\left\| Ax\left( t \right) -b \right\| ^2$, thus, integrating the aboved inequality from $t_0$ to $+\infty $ and rearranging it, we have
\begin{align*}
\frac{\alpha -2}{\alpha ^2}&\int_{t_0}^{+\infty}{t\left( \mathcal{L} _{\beta}\left( x\left( t \right) ,\lambda ^* \right) -\mathcal{L} _{\beta}\left( x^*,\lambda \left( t \right) \right) \right)}dt+\frac{\beta}{2\alpha}\int_{t_0}^{+\infty}{t\left\| Ax\left( t \right) -b \right\| ^2}dt
\\
=&V\left( t_0 \right) -V\left( +\infty \right) \leq V\left( t_0 \right) <+\infty,
\end{align*}
and it combines with $\mathcal{L} _{\beta}\left( x\left( t \right) ,\lambda ^* \right) -\mathcal{L} _{\beta}\left( x^*,\lambda \left( t \right) \right) \geq 0, \forall x\left( t \right) \in \mathcal{X}$ leads to 
\\ $\int_{t_0}^{+\infty}{t\left\| Ax\left( t \right) -b \right\| ^2}dt<+\infty $, i.e., the condition \eqref{4.7} holds. Moreover, \eqref{4.7} combines with $\left\| Ax\left( t \right) -b \right\| ^2\geq 0$ implies $\int_{t_0}^{+\infty}{t\left\| Ax\left( t \right) -b \right\| ^2}dt<+\infty $, i.e., the \eqref{4.71} holds.

\textit{(III)}: The equation $\lambda ^*=\begin{cases}
		0, \,\,\,\,\,\,\,\,\,\,\,\,\,\,\,\,\,\,\,\,\,\,  if\,\,Ax\left( t \right) -b=0,\\
		\frac{Ax\left( t \right) -b}{\left\| Ax\left( t \right) -b \right\|},if\,\,Ax\left( t \right) -b\ne 0,\\
	\end{cases}$ and \eqref{4.6} yields $0\le f\left( x\left( t \right) \right) -f\left( x^* \right) +\left\| Ax\left( t \right) -b \right\| +\frac{\beta}{2}\left\| Ax\left( t \right) -b \right\| ^2\le \frac{\alpha ^2V\left( t_0 \right)}{t^2}$. 
	
	Observing that the upper-bound of \eqref{sd} in \eqref{saddle} implies
	\begin{equation}\label{v_in4}
	\begin{split}
	f\left( x\left( t \right) \right) -f\left( x^* \right)
	\geq& -\left( Ax\left( t \right) -b \right) ^T\lambda ^* 
	\\
	\geq& -\left\| Ax\left( t \right) -b \right\| \geq -\frac{\alpha \sqrt{V\left( t_0 \right)}}{t}.
	\end{split}
	\end{equation}
	
	Combining the \eqref{v_in4}, $\left\| Ax\left( t \right) -b \right\| \geq 0$ and  $f\left( x\left( t \right) \right) -f\left( x^* \right) \leq \frac{\alpha ^2V\left( t_0 \right)}{t^2}$, we deduce that  $-\frac{\alpha \sqrt{V\left( t_0 \right)}}{t}\leq f\left( x\left( t \right) \right) -f\left( x^* \right) \leq \frac{\alpha ^2V\left( t_0 \right)}{t^2}$ with any $t\geq t_0>0$. Finally, the conclusion (III) holds.
\end{proof}

\subsection{Examples of the APDMD}\label{Sec2.5} In this subsection, we will illustrate some examples of the APDMD (\ref{PDM}, i.e., (\ref{PDM1}, in other words, the APDMD can reduce to some classical and new dynamical approaches when choosing different constraint set $\mathcal{X}$, i.e.,  $\mathcal{X}$ is an Euclidean space, $\mathcal{X}$ is a positive-orthant constrained set, $\mathcal{X}$ is a unit simplex set and $\mathcal{X}$ is a closed and convex set, and its projection operator  $P_{\mathcal{X}}$ has a closed form solution (the detailed description is given in 

\subsection{Examples of the APDMD}\label{Sec2.5} In this subsection, we will illustrate some examples of the APDMD \eqref{PDM}, i.e., \eqref{PDM1}, in other words, the APDMD can reduce to some classical and new dynamical approaches when choosing different constraint set $\mathcal{X}$, i.e.,  $\mathcal{X}$ is an Euclidean space, $\mathcal{X}$ is a positive-orthant constrained set, $\mathcal{X}$ is a unit simplex set and $\mathcal{X}$ is a closed and convex set, and its projection operator  $P_{\mathcal{X}}$ has a closed form solution.

\subsection{Examples of the APDMD}
\textbf{Case 1:} If $\mathcal{X}$ is an Euclidean space, i.e., $\mathcal{X} =\mathbb{R}^n$. Let $\psi \left( x\left( t \right)+\frac{\alpha}{t}\dot{x}\left( t \right) \right)$$=\frac{1}{2}\left\| x\left( t \right)+\frac{\alpha}{t}\dot{x}\left( t \right) \right\| ^2$, then, one has $\nabla \psi \left( x\left( t \right)+\frac{\alpha}{t}\dot{x}\left( t \right) \right) =u\left( t \right)$, $\psi ^*\left( u\left( t \right) \right) =\frac{1}{2}\left\| u\left( t \right) \right\| ^2$, $\nabla \psi ^*\left( u \left( t \right)\right)=u\left( t \right)$, $\nabla ^2\psi ^*\left( \nabla \psi \left( x\left( t \right)\left( t \right)+\frac{\alpha}{t}\dot{x}\left( t \right) \right) \right) \equiv I_n$. The APDMD \eqref{PDM1} reduces to the classical accelerated primal-dual dynamical approaches \cite{He2021ConvergenceRO, Zeng2019DynamicalPA, Bo2021ImprovedCR, Attouch2021FastCO}.
\begin{equation}\label{PDM_case1}
	\begin{cases}
			\ddot{x}\left( t \right) +\frac{\alpha +1}{t}\dot{x}\left( t \right) 
			\\
			\,\,\,\,\,\,\,\,\,\,\,\,\,+\left( \nabla f\left( x\left( t \right) \right) +\beta A^T\left( Ax\left( t \right) -b \right) +A^T\left( \lambda \left( t \right) +\frac{t}{\alpha}\dot{\lambda}\left( t \right) \right) \right) =0,
			\\
			\ddot{\lambda}\left( t \right) +\frac{\alpha +1}{t}\dot{\lambda}\left( t \right) -A\left( x\left( t \right) +\frac{t}{\alpha}\dot{x}\left( t \right) \right) +b=0,
			\\
			x\left( t_0 \right) =x_0\in \mathcal{X}, \dot{x}\left( t_0 \right) =0, \lambda \left( t_0 \right) =\lambda _0, \dot{\lambda}\left( t_0 \right) =0.
		\end{cases}
\end{equation} 
%Note that \eqref{PDM_case1} reduces to the Nesterov ODE studied in []--[], $\ddot{x}+\frac{\alpha +1}{t}\dot{x}+\nabla f\left( x \right) =0$ when the affine constraint $Ax=b$ does not exist.

\textbf{Case 2:} \cite{Banerjee2005ClusteringWB} If $\mathcal{X}$ is a positive-orthant constrained set, i.e., $\mathcal{X} =\mathbb{R}_{+}^{n}$. Let the negative entropy function $\psi \left( x \right) =\sum_{i=1}^n{x_i\ln x_i}$ be a distance generating function in this case. Then, one has $\nabla \psi \left( x\right) =\mathrm{col}\left( 1+\ln x_1,...,1+\ln x_n \right)$ and $\psi ^*\left( u\right) =\sum_{i=1}^n{e^{u_i-1}}$, $\nabla \psi ^*\left( u \right) =\mathrm{col}\left( e^{u_1-1},...,e^{u_n-1} \right)$, $\nabla ^2\psi ^*\left( u\left( t \right) \right) =\mathrm{diag}\left( e^{u_1-1},..., e^{u_n-1} \right) $. Thus, for all $x\in \mathbb{R}_{+}^{n}$, we have
\begin{equation*}
	\begin{split}
			&\nabla ^2\psi ^*\left( \nabla \psi \left( x\left( t \right) +\frac{t}{\alpha}\dot{x}\left( t \right) \right) \right) 
			\\
			=&\mathrm{diag}\left( x_1\left( t \right) +\frac{t}{\alpha}x_1\left( t \right) ,...,x_n\left( t \right) +\frac{t}{\alpha}{t}x_n\left( t \right) \right) =\mathrm{diag}\left( x\left( t \right) +\frac{t}{\alpha}\dot{x}\left( t \right) \right).
		\end{split}
\end{equation*}
Correspondingly, APDMD \eqref{PDM1} is turned into 
\begin{equation}\label{PDM_case2}
	\begin{cases}
			\ddot{x}\left( t \right) +\frac{\alpha +1}{t}\dot{x}\left( t \right) +\mathrm{diag}\left( x\left( t \right) +\frac{t}{\alpha}\dot{x}\left( t \right) \right) \left( \nabla f\left( x\left( t \right) \right) \right. 
			\\
			\left. \,\,\,\,\,\,\,\,\,\,\,\,\,\,\,\,\,\,\,\,\,\,\,\,\,\,\,\,\,\,\,\,\,\,\,\,\,\,\,\,\,+\beta A^T\left( Ax\left( t \right) -b \right) +A^T\left( \lambda \left( t \right) +\frac{t}{\alpha}\dot{\lambda}\left( t \right) \right) \right) =0,
			\\
			\ddot{\lambda}\left( t \right) +\frac{\alpha +1}{t}\dot{\lambda}\left( t \right) -A\left( x\left( t \right) +\frac{t}{\alpha}\dot{x}\left( t \right) \right) +b=0,
			\\
			x\left( t_0 \right) =x_0\in \mathcal{X}, \dot{x}\left( t_0 \right) =\dot{x}_0, \lambda \left( t_0 \right) =\lambda _0, \dot{\lambda}\left( t_0 \right) =\dot{\lambda}_0,
		\end{cases}
\end{equation} 

In addition, for the constrained set $\mathcal{X} =\mathbb{R}_{+}^{n}$, the distance function can also be selected as $\psi \left( x \right) =-\sum_{i=1}^n{\ln x_i}$. Then, we have $\nabla \psi \left( x  \right) =\mathrm{col}\left( -\frac{1}{x_1} ,...,-\frac{1}{x_n}\right) $ and the Bregman divergence of $\psi$ for $x $ and $y $ is $D_{\psi}\left( x ,y\right) =\sum_{i=1}^n{\left( \frac{x_i}{y_i}\ln \frac{x_i}{y_i}-1 \right)}$ (named the \textbf{Itakura-Saito} divergence). Further we have $\psi ^*\left( u\right) =-\sum_{i=1}^n{\left( 1+\ln \left( -u_i\right) \right)}$ on $\mathbb{R}_{-}^{n}$, $\nabla \psi ^*\left( u \right) =\mathrm{col}\left( -\frac{1}{u_1},...,-\frac{1}{u_n} \right)$ and $\nabla ^2\psi ^* =\mathrm{diag}\left( \frac{1}{\left( u_1 \right) ^2},...,\frac{1}{\left( u_n \right) ^2} \right) $ on $\mathbb{R}_{-}^{n}$. To sum up, one has 
\begin{equation}\label{ISD}
\begin{split}
&\nabla ^2\psi ^*\left( \nabla \psi \left( x\left( t \right)+\frac{t}{\alpha}\dot{x}\left( t \right) \right) \right)
\\
=&\mathrm{diag}\left( \left( x_1\left( t \right)+\frac{t}{\alpha}\dot{x}_1\left( t \right) \right) ^2,...,\left( x_n+\frac{t}{\alpha}\dot{x}_n \left( t \right)\right) ^2 \right).
\end{split}
\end{equation}
 
\textbf{Case 3:} \cite{Beck2003MirrorDA} If $\mathcal{X}$ is a unit simplex set, i.e., $\mathcal{X} =\bigtriangleup=\left\{ x\in \mathbb{R}_{+}^{n}|\sum_{i=1}^n{x_i=1} \right\}$. Considering a distance-generating function $\psi \left( x \right) =\sum_{i=1}^n{x_i\ln x_i}+I_{\mathcal{X}}\left( x \right)$, where $\sum_{i=1}^n{x_i\ln x_i}$ is the negative entropy function and $I_{\mathcal{X}}\left( x \right)$ is the indicator function on $\mathcal{X}$ and its Bregman divergence is $D_{\psi}\left( x,y \right) =\sum_{i=1}^n{x_i\ln \frac{x_i}{yi}}$ between the vectors $x$, $y$ (called \textbf{Kullback-Leibler} divergence). Then, we have
\begin{align*}
	&\nabla \psi \left( x \right) =\mathrm{col}\left( 1+\ln x_1,...,1+\ln x_n \right), \forall x\in \mathcal{X}, 
	\\
	&\psi ^*\left( u \right) =\ln \left( \sum_{i=1}^n{e^{u_i}} \right), \nabla \psi ^*\left( u \right) =\mathrm{col}\left( \frac{e^{u_1}}{\sum_{i=1}^n{e^{u_i}}},...,\frac{e^{u_n}}{\sum_{i=1}^n{e^{u_i}}} \right), 
	\\
	&\left[ \nabla ^2\psi ^*\left( u \right) \right] _{ij}=\frac{e^{u_i}\sum_{i=1}^n{e^{u_i}}-e^{u_j}e^{u_i}}{\left( \sum_{i=1}^n{e^{u_i}} \right) ^2}, \forall \,\,i,j\in n.	
\end{align*}
Furthermore, for any $i,j =1,...,n$, we have
\begin{align*}
	&\left[ \nabla ^2\psi ^*\left( \nabla \psi \left( x\left( t \right) +\frac{t}{\alpha}\dot{x}\left( t \right) \right) \right) \right] _{ij}
\\
%=&\frac{\left( x_i\left( t \right) +\frac{\alpha}{t}\dot{x}_i\left( t \right) \right) \sum_{i=1}^n{\left( x_i\left( t \right) +\frac{\alpha}{t}\dot{x}_i\left( t \right) \right)}-\left( x_i\left( t \right) +\frac{\alpha}{t}\dot{x}_i\left( t \right) \right) \left( x_j\left( t \right) +\frac{\alpha}{t}\dot{x}_j\left( t \right) \right)}{\left( \sum_{i=1}^n{\left( x_i\left( t \right) +\frac{\alpha}{t}\dot{x}_i\left( t \right) \right)} \right) ^2}
%\\
=&\left( x_i\left( t \right) +\frac{t}{\alpha}\dot{x}_i\left( t \right) \right) -\left( x_i\left( t \right) +\frac{t}{\alpha}\dot{x}_i\left( t \right) \right) \left( x_j\left( t \right) +\frac{t}{\alpha}\dot{x}_j\left( t \right) \right).
\end{align*}
%where the second equation follows from that $\sum_{i=1}^n{\left( x_i\left( t \right) +\frac{\alpha}{t}\dot{x}_i\left( t \right)\right)}=1$.

Therefore, the APDMD \eqref{PDM1} reduces to
\begin{align}\label{PDM_case3}
\begin{cases}
   \ddot{x}_i\left( t \right) +\frac{\alpha +1}{t}\dot{x}_i\left( t \right) +\left( x_i\left( t \right) +\frac{t}{\alpha}\dot{x}_i\left( t \right) \right) \left( \left( \nabla f\left( x\left( t \right) \right) +\beta A^T\left( Ax\left( t \right) -b \right) \right. \right. 
\\
\left.\,\,\,\,\,\,\,\,\,\,\,\,\,\,+A^T\left( \lambda \left( t \right) +\frac{t}{\alpha}\dot{\lambda}\left( t \right) \right) \right) _i\begin{array}{c}
		-\left( x\left( t \right) +\frac{t}{\alpha}\dot{x}\left( t \right) \right) ^T\\
	\end{array}\left( \nabla f\left( x\left( t \right) \right) \right. 
	\\
	\left. \left.\,\,\,\,\,\,\,\,\,\,\,\,\,\, +\beta A^T\left( Ax\left( t \right) -b \right) +A^T\left( \lambda \left( t \right) +\frac{t}{\alpha}\dot{\lambda}\left( t \right) \right) \right) \right) =0,\forall i=1,...,n,
		\\
	\ddot{\lambda}\left( t \right) +\frac{\alpha +1}{t}\dot{\lambda}\left( t \right) -A\left( x\left( t \right) +\frac{t}{\alpha}\dot{x}\left( t \right) \right) +b=0,
		\\
	x\left( t_0 \right) =x_0\in \mathcal{X}, \dot{x}\left( t_0 \right) =\dot{x}_0, \lambda \left( t_0 \right) =\lambda _0, \dot{\lambda}\left( t_0 \right) =\dot{\lambda}_0,
	\end{cases}
	\end{align}
	where the $\left( \nabla f\left( x\left( t \right) \right) +\beta A^T\left( Ax\left( t \right) -b \right) +A^T\left( \lambda \left( t \right) +\frac{t}{\alpha}\dot{\lambda}\left( t \right) \right) \right) _i$ is the $i$-th element of $\nabla f\left( x\left( t \right) \right) +\beta A^T\left( Ax\left( t \right) -b \right) +A^T\left( \lambda \left( t \right) +\frac{t}{\alpha}\dot{\lambda}\left( t \right) \right)$.
	
	\textbf{Case 4:} If $\mathcal{X}$ is a closed and convex set, and its projection operator  $P_{\mathcal{X}}$ has a closed form solution, moreover, setting $\psi \left( x \right)$ be $\frac{1}{2}\left\| x \right\| ^2+I_{\mathcal{X}}\left( x \right)$, one has $\partial \psi \left( x \right) =x+\mathcal{N} _{\mathcal{X}}\left( x \right)$ and $D_{\psi}\left( x,y \right) =\frac{1}{2}\left\| x-y \right\| ^2, \forall x, y\in \mathcal{X}$. Furthermore, one has $\psi ^*\left( u \right) =\frac{1}{2}\left( \left\| u \right\| ^2-\left\| u-P_{\mathcal{X}}\left( u \right) \right\| ^2 \right)$, $\nabla \psi ^*\left( u \right) =P_{\mathcal{X}}\left( u \right) $. According to the above discussion, the APDMD \eqref{PDM} reduces to a new accelerated primal-dual projection dynamical (APDPD) approach : 
	\begin{equation}\label{APDPD}
		\begin{cases}
				\dot{x}\left( t \right)=\frac{\alpha}{t}\left( P_{\mathcal{X}}\left( u\left( t \right) \right) -x\left( t \right) \right),\\
				\dot{u}\left( t \right)=-\frac{t}{\alpha}\left( \nabla f\left( x\left( t \right) \right) +\beta A^T\left( Ax\left( t \right)-b \right) +A^Tv\left( t \right) \right),\\
				\dot{\lambda}\left( t \right)=\frac{\alpha}{t}\left( v\left( t \right)-\lambda\left( t \right) \right),\\
				\dot{v}\left( t \right)=\frac{t}{\alpha}\left( AP_{\mathcal{X}\left( t \right)}\left( u\left( t \right) \right) -b \right),\\	
					x\left( t_0 \right) =x_0, u\left( t_0 \right) =u_0\,\,\mathrm{with}\,\,x_0=P_{\mathcal{X}}\left( u_0 \right) \in \mathcal{X}; \lambda \left( t_0 \right) =\lambda _0, v\left( t_0 \right)=0.
			\end{cases}    
	\end{equation} 

\subsection{APDMD for DCCP \eqref{P2} in the smooth case}
In this subsection, inspired by the ADPDMD \eqref{PDM1} and distributed consensus theorem, an accelerated distributed primal-dual mirror dynamical (ADPDMD) approach for DCCP \eqref{P2} in the smooth case is proposed and discussed.

For DCCP \eqref{P2} with smooth convex objective functions, the proposed APDMD with $\alpha \geq2$ is given by   

\begin{equation}\label{ADPDMD1}
	\begin{cases}
		\dot{x}_i\left( t \right) =\frac{\alpha}{t}\left( \nabla \psi _{i}^{*}\left( u_i\left( t \right) \right) -x_i\left( t \right) \right),\\
		\dot{u}_i\left( t \right) =-\frac{t}{\alpha}\left( \nabla f_i\left( x_i\left( t \right) \right) +\beta \sum_{j=1}^n{\mathrm{a}_{ij}\left( x_i\left( t \right) -x_j\left( t \right) \right)} \right.\\
		\left. \,\,\,\,\,\,\,\,\,\,\,\,\,\,\,\,\,\,\,\,\,+\sum_{j=1}^n{\mathrm{a}_{ij}\left( \lambda _i\left( t \right) +\frac{t}{\alpha}\dot{\lambda}_i\left( t \right) -\lambda _j\left( t \right) -\frac{t}{\alpha}\dot{\lambda}_j\left( t \right) \right)} \right) ,\\
		\dot{\lambda}_i\left( t \right) =\frac{\alpha}{t}\left( v_i\left( t \right) -\lambda _i\left( t \right) \right),
		\\
		\dot{v}_i\left( t \right) =\frac{t}{\alpha}\sum_{j=1}^n{\mathrm{a}_{ij}\left( \nabla \psi _{i}^{*}\left( u_i\left( t \right) \right) -\nabla \psi _{j}^{*}\left( u_j\left( t \right) \right) \right)},\\
		x_i\left( t_0 \right) =x_{i,0},\nabla \psi _{i}^{*}\left( u_i\left( t_0 \right) \right) =x_{i,0},
		\\
		\lambda _i\left( t_0 \right) =\lambda _{i,0},v_i\left( t_0 \right) =v_{i,0}, i=1,...,n.\\
	\end{cases}
\end{equation}

Defining  $L=L_n\otimes I_m$, where $L_n$ is the Laplacian matrix of graph $\mathcal{G}$ and let $x=\mathrm{col}\left( x_1,...,x_n\right) \in \mathbb{R}^{nm}$, $\dot{x}=\mathrm{col}\left( \dot{x}_1,...,\dot{x}_n\right) \in \mathbb{R}^{nm}$, $\lambda =\mathrm{col}\left( \lambda _1,...,\lambda _n\right) \in \mathbb{R}^{nm}$, $\dot{\lambda}=\mathrm{col}\left( \dot{\lambda}_1,...,\dot{\lambda}_n\right) \in \mathbb{R}^{nm}$, $\nabla f\left( x \right) =\mathrm{col}\left( \nabla f_1\left( x_1\right) ,...,\nabla f_n\left( x_n \right) \right) \in \mathbb{R}^{nm}$, $\nabla \psi ^*\left( u \right) =\mathrm{col}\left( \nabla \psi _{1}^{*}\left( u_1 \right) ,...,\nabla \psi _{n}^{*}\left( u_n \right) \right) \in \mathbb{R} ^{nm}$ and $\mathcal{X} =\prod_{i=1}^n{\mathcal{X} _i}$ is the Cartesian product of set $\mathcal{X} _i$, $i=1,...,n$. Then the compact formula of APDMD \eqref{ADPDMD1} is given by 

\begin{equation}\label{ADPDMD}
	\begin{cases}
		\dot{x}\left( t \right) =\frac{\alpha}{t}\left( \nabla \psi ^*\left( u\left( t \right) \right) -x\left( t \right) \right),
		\\
		\dot{u}\left( t \right) =-\frac{t}{\alpha}\left( \nabla f\left( x\left( t \right) \right) +\beta Lx\left( t \right) +Lv\left( t \right)\right),
		\\
		\dot{\lambda}\left( t \right) =\frac{\alpha}{t}\left( v\left( t \right) -\lambda \left( t \right) \right), \dot{v}\left( t \right)=\frac{t}{\alpha}L\nabla \psi ^*\left( u\left( t \right) \right),
		\\
		x\left( t_0 \right) =x_0, \nabla \psi ^*\left( u\left( t_0 \right) \right) =x_0,\lambda \left( t_0 \right) =\lambda _0, v\left( t_0 \right) =v_0.               \\
	\end{cases}
\end{equation}

Next, we will illustrate the accelerated convergence rate of the ADPDD \eqref{ADPDMD} by the Lyapunov analysis method.

\begin{theorem} Suppose Assumption 3.2 holds and let $\left( x\left( t \right) ,\lambda \left( t \right) \right) $ and $\left( x^*,\lambda ^* \right)$ be a solution trajectory and an optimal solution for ADPDMD \eqref{ADPDMD} and DCCP \eqref{P2}, respectively. Then for any $\left( x\left( t_0 \right) ,u\left( t_0 \right) ,\lambda \left( t_0 \right) ,\upsilon \left( t_0 \right) \right) \in \mathcal{X} \times \mathbb{R} ^{nm}\times \mathbb{R} ^{nm}\times \mathbb{R} ^{nm}$, the following statements are true:
	
	\textit{(I)}: The trajectories $x\left( t \right)$ and $\lambda \left( t \right)$ of ADPDMD \eqref{ADPDMD} are bounded for any $t\geq t_0>0$.
	
	(II): Let $\mathsf{L} _{\beta}\left( x\left( t \right)\lambda\left( t \right) \right) =f\left( x\left( t \right)\right) +\frac{\beta}{2}x\left( t \right)^TLx\left( t \right)+ \lambda \left( t \right) ^TLx\left( t \right)$, one has 
	\begin{subequations}
		\begin{align}
			&\mathsf{L} _{\beta}\left( x\left( t \right),\lambda ^* \right) -\mathsf{L} _{\beta}\left( x^*,\lambda \left( t \right)\right) \leq \frac{\alpha ^2\mathsf{V} \left( t_0 \right)}{t^2},\label{conver2,1}
			\\
			&x\left( t \right) ^TLx\left( t \right) \leq \frac{2\alpha ^2\mathsf{V} \left( t_0 \right)}{\beta t^2}, \label{conver2,2}
			\\
			&\int_{t_0}^{+\infty}{t\left( \mathsf{L} _{\beta}\left( x\left( t \right) ,\lambda ^* \right) -\mathsf{L} _{\beta}\left( x^*,\lambda \left( t \right) \right) \right) dt}<+\infty, \label{conver2,3}
			\\
			&\int_{t_0}^{+\infty}{t\left\| Ax\left( t \right) -b \right\| ^2}dt<+\infty.\label{conver2,4}
		\end{align}
	\end{subequations}
	(III): Let $\lambda ^*=\begin{cases}
		0, \,\,\,\,\,\,\,\,\,\,\,\,\,\,\,\,\,\,\,\,\,\,\,\,\,\,\,  if\,\,Lx\left( t \right)=0,\\
		\frac{x\left( t \right)}{\sqrt{x\left( t \right) ^TLx\left( t \right)}},if\,\,Lx\left( t \right)\ne 0,\\
	\end{cases}$. For any $t\geq t_0>0$, it follows that 
	\begin{subequations}
		\begin{align}
			&0\leq f\left( x\left( t \right) \right) -f\left( x^* \right) +\sqrt{x\left( t \right) ^TLx\left( t \right)}\leq \frac{\alpha ^2V\left( t_0 \right)}{t^2},\label{conve_2,5}
			\\
			&-\frac{\alpha \sqrt{2\mathsf{V} \left( t_0 \right)}}{t\sqrt{\beta}}\leq f\left( x\left( t \right) \right) -f\left( x^* \right) \leq \frac{\alpha ^2V\left( t_0 \right)}{t^2}.\label{conve_2,6}
		\end{align}
	\end{subequations}
\end{theorem}

\begin{proof}
	Consider the following candidate Lyapunov function:
	\begin{align}\label{energy_2}
		\mathsf{V} \left( t \right) =\mathsf{V} _1\left( t \right) +\mathsf{V} _2\left( t \right) +\mathsf{V} _3\left( t \right),
	\end{align}
	with 
	\begin{align*}	
		\begin{cases}
			\mathsf{V} _1\left( t \right) =\frac{t^2}{\alpha ^2}\left( \mathsf{L} _{\beta}\left( x\left( t \right) ,\lambda ^* \right) -\mathsf{L} _{\beta}\left( x^*,\lambda \left( t \right) \right) \right)\\
			\,\,\,\,\,\,\,\,\,\,\,\,\,\,\,\,=\frac{t^2}{\alpha ^2}\left( f\left( x\left( t \right) \right) -f\left( x^* \right) +\frac{\beta}{2}x\left( t \right) Lx\left( t \right) ^2+L\lambda ^* \right) ,\\
			\mathsf{V} _2\left( t \right) =D_{\psi ^*}\left( u\left( t \right) ,u^* \right) =\psi ^*\left( u\left( t \right) \right) -\psi ^*\left( u^* \right) -\nabla \psi ^*\left( u^* \right) ^T\left( u\left( t \right) -u^* \right) ,\\
			\mathsf{V} _3\left( t \right) =D_h\left( v\left( t \right) ,\lambda ^* \right) =\frac{1}{2}\left\| v\left( t \right) -\lambda ^* \right\| ^2,\\
		\end{cases}
	\end{align*}
Following the similar steps as the proof of \ref{Theorem 2.4}, the proof can be obtained easily. Due to space limitations, we omit the proof. 
\end{proof}

\subsection{APDMD for DEMO \eqref{P3} in smooth case}
In this subsection, to address the DEMO \eqref{P3}, where its objective function is smooth and convex, an accelerated distributed mirror dynamical (ADMD) approach is investigated based on APDMD \eqref{PDM}. Recalling the DEMO \eqref{P3}, the coupled equations $\sum_{i=1}^n{A_ix_i}=\sum_{i=1}^n{d_i}\in \mathbb{R} ^{m}$ can be equivalently decomposed as $\bar{A}x-d+Ly=0\in \mathbb{R} ^{nm}$ with $L=L_n\otimes I_m$ and an auxiliary variable $y=\mathrm{col}\left( y_{1},...,y_{n}\right)\in \mathbb{R} ^{nm}$ by using the properties of the Laplacian matrix of undirected $\mathcal{G}$ (i.e., $\mathrm{ker}\left(L_n \right) =\left\{ \varsigma 1|\varsigma \in \mathbb{R} \right\}$, $\mathrm{rang}\left( L_n \right) =\left\{ \omega \in \mathbb{R}^n|\omega ^T1=0 \right\} $). Therefore, the DEMO \eqref{P3} is equivalent to:
\begin{equation}\label{PP}
	\begin{split}
		&\underset{x\in \mathbb{R} ^{\sum_{i=1}^n{p_i}},y\in \mathbb{R} ^{nm}}{\min} \,\,  f\left( x \right) =\sum_{i=1}^n{f_i\left( x_i \right)},
		\\
		&\mathrm{s}.\mathrm{t}. \,\,\bar{A}x-d+Ly=0, x\in \mathcal{X} =\prod_{i=1}^n{\mathcal{X} _i},
	\end{split}
\end{equation}
where $\bar{A}=\mathrm{bl}\mathrm{diag}\left\{ A_{p_1},A_{p_2},...,A_{p_i}\right\} \in \mathbb{R}^{nm \times \sum_{i=1}^n{p_i}},\,\,d=\mathrm{col}\left( d_1,d_2,...,d_n \right) \in \mathbb{R}^{nm}$, and $\mathcal{X} =\prod_{i=1}^n{\mathcal{X} _i}$ is the Cartesian product of set $\mathcal{X} _i$, $i=1,...,n$.

To deal with the modified DEMO \eqref{PP} with smooth convex objective functions, we propose the following accelerated distributed mirror dynamical (ADMD) approach with $t\geq t_0>0$ and $\alpha \geq 2$ as follows:
\begin{align}\label{ADMD}
	\begin{cases}
		\dot{x}\left( t \right) =\frac{\alpha}{t}\left( \nabla \psi ^*\left( u\left( t \right) \right) -x\left( t \right) \right), \dot{u}\left( t \right) =-\frac{t}{\alpha}\left( \nabla f\left( x\left( t \right) \right) +\bar{A}^Tv\left( t \right) \right), \,\,         \\
		\dot{\lambda}\left( t \right) =\frac{\alpha}{t}\left( v\left( t \right) -\lambda \left( t \right) \right) , \dot{v}\left( t \right) =\frac{t}{\alpha}\left( \bar{A}\psi ^*\left( u\left( t \right) \right) -d-L\lambda \left( t \right) +Lz\left( t \right) \right),\\
		\dot{y}\left( t \right) =\frac{\alpha}{t}\left( z\left( t \right) -y\left( t \right) \right) ,\dot{z}\left( t \right) =-\frac{t}{\alpha}Lv\left( t \right).                            \\  
		x\left( t_0 \right) =x_0,u\left( t_0 \right) =u_0\,\,\mathrm{with}\,\,\nabla \psi ^*\left( u_0 \right) =x_0\in \mathcal{X},
		\\
		\lambda \left( t_0 \right) =\lambda _0,v\left( t_0 \right) =v_0, y\left( t_0 \right) =y_0,z\left( t_0 \right) =z_0, 
	\end{cases}
\end{align}

\begin{theorem}
	Suppose Assumption 3.2 holds and let $\left( x\left( t \right) ,\lambda \left( t \right) ,y\left( t \right) ,u\left( t \right) , \right. 
	\\
	\left. v\left( t \right) ,z\left( t \right) \right) $ be a solution trajectory and $\left( x^*,\lambda ^*,y^*,u^*,v^*,z^* \right) $ be an optimal solution of ADMD \eqref{ADMD}, respectively. Then, for any  $\left( x\left( t_0 \right) ,u\left( t_0 \right) ,\lambda \left( t_0 \right) ,\upsilon \left( t_0 \right) ,y\left( t_0 \right) ,z\left( t_0 \right) \right) \in \mathcal{X} \times \mathbb{R} ^{\sum_{i=1}^n{p_i}}\times \mathbb{R} ^{nm}\times \mathbb{R} ^{nm}\times \mathbb{R} ^{nm}\times \mathbb{R} ^{nm}$, we have the following two statements:	
	
	\textit{(I)}: The trajectories $x\left( t \right)$, $\lambda \left( t \right)$ and $y\left( t \right)$ of ADMD \eqref{ADMD} are bounded for any $t\geq t_0>0$.
	
	\textit{(II)}: Let $\mathbb{L} _{\beta}$ with $\beta=1$, i.e., $\mathbb{L} _{1}\left( x\left( t \right),\lambda\left( t \right) ,y\left( t \right) \right)$ be $f\left( x\left( t \right) \right) -\frac{1}{2}\lambda\left( t \right) ^TL\lambda\left( t \right) +\lambda\left( t \right) ^T$
	\\
	$\left( \bar{A}x\left( t \right)-d-Ly\left( t \right)\right) $, it follows that 
	\begin{subequations}
		\begin{align}
			&\mathbb{L} _{1}\left( x\left( t \right),y^*,\lambda ^* \right) -\mathbb{L}_{1}\left( x^*,y^*,\lambda\left( t \right) \right) \le \frac{\alpha ^2E\left( t_0 \right)}{t^2}, \,\,\lambda \left( t \right) ^TL\lambda \left( t \right) \le \frac{2\alpha ^2E\left( t_0 \right)}{t^2}, \label{conver3,1}
			\\
			&\int_{t_0}^{+\infty}{t\left( \mathbb{L} _1\left( x\left( t \right) ,y^*,\lambda ^* \right) -\mathbb{L} _1\left( x^*,y^*,\lambda \left( t \right) \right) \right) dt}<+\infty. \label{conver3,2}
		\end{align}
	\end{subequations}
\end{theorem}
\begin{proof}
		\textit{(I)}: Designing a Lyapunov function as follows:
		\begin{align}
				E\left( t \right) =&E_1\left( t \right) +E_2\left( t \right) +E_3\left( t \right),
						\end{align}
	with
	\begin{align*}
	\begin{cases}
		E_1\left( t \right) =\frac{t^2}{\alpha ^2}\left( \mathbb{L} _1\left( x\left( t \right) ,y^*,\lambda ^* \right) -\mathbb{L} _1\left( x^*,y^*,\lambda \left( t \right) \right) \right) 
		\\
		\,\,=\frac{t^2}{\alpha ^2}\left( f\left( x\left( t \right) \right) -f\left( x^* \right) +\left( \lambda ^* \right) ^T\left( \bar{A}x\left( t \right) -d-Ly^* \right) +\frac{\beta}{2}\lambda \left( t \right) ^TL\lambda \left( t \right) \right) 
		\\
		E_2\left( t \right) =D_{\psi ^*}\left( u,u^* \right) =\psi ^*\left( u\left( t \right) \right) -\psi ^*\left( u^* \right) -\nabla \psi ^*\left( u^* \right) ^T\left( u\left( t \right) -u^* \right) ,
		\\
		E_3\left( t \right) =D_h\left( v\left( t \right) ,\lambda ^* \right) =\frac{1}{2}\left\| v\left( t \right) -\lambda ^* \right\| ^2,
		\\
		E_4\left( t \right) =D_h\left( z\left( t \right) ,y^* \right) =\frac{1}{2}\left\| z\left( t \right) -y^* \right\| ^2.
	\end{cases}	
	\end{align*}
	%where $D_{\psi ^*}\left( u, u^* \right)$, $D_h\left( v,\lambda^* \right) $ are respectively the Bregman divergence associated with $\psi ^*$ for $u, u^*$, $h$ for $v,\lambda ^*$ and $z, y^*$. 
	Using the same proof steps as \eqref{Theorem 2.4}, we can get
	\begin{equation}
	\begin{split}\label{energy_3}
			\dot{E}\left( t \right) =&\dot{E}_1\left( t \right) +\dot{E}_2\left( t \right) +\dot{E}_3\left( t \right) +\dot{E}_4\left( t \right) 
			\\
		 \leq& -\frac{\left( \alpha -2 \right) t}{\alpha ^2}\left( f\left( x\left( t \right) \right) -f\left( x^* \right) +\left( \lambda ^* \right) ^T\bar{A}\left( x\left( t \right) -x^* \right) +\frac{1}{2}\lambda \left( t \right) ^TL\lambda \left( t \right) \right)\nonumber 
			\\
			=&-\frac{\left( \alpha -2 \right) t}{\alpha ^2}\left( \mathbb{L} _{1}\left( x\left( t \right) ,y^*,\lambda ^* \right) -\mathbb{L} _{1}\left( x^*,y^*,\lambda \left( t \right) \right) \right) \leq0. 		
	\end{split}
	\end{equation}
	% where the first inequality satisfied due to the convexity of $\mathbb{L} _1\left( x\left( t \right) ,y^*,\lambda ^* \right)$  with respect of $ x\left( t \right) $ with fixed $\lambda ^*$ and $\lambda ^*$ on $\mathcal{X} $, and the last inequation is established since $\alpha \geq 2$. 
	Recall that the Lyapunov function $E\left( t \right)$ is radically unbounded and positive with any $t\geq t_0>0$. This implies that the trajectories of $x\left( t \right)$, $\lambda \left( t \right)$ and $y\left( t \right)$ are bounded for any $t\geq t_0>0$. The proof of (I) is therefore completed.
		
	\textit{(II)}: Note that $\dot{E}\left( t \right)\leq0$, i.e., ${E}\left( t \right)$ is nonincreasing on $\left[ t_0,+\infty \right)$. Thus, for any $t\geq t_0>0$, one has $\frac{t^2}{\alpha ^2}\left( \mathbb{L} _1\left( x\left( t \right) ,y^*,\lambda ^* \right) -\mathbb{L} _1\left( x^*,y^*,\lambda \left( t \right) \right) \right) =E_1\left( t \right) \le E\left( t \right) \le E\left( t_0 \right) $, i.e., $\mathbb{L} _1\left( x\left( t \right) ,y^*,\lambda ^* \right) -\mathbb{L} _1\left( x^*,y^*,\lambda \left( t \right) \right) \le \frac{\alpha ^2E\left( t_0 \right)}{t^2}$ holds. Furthermore, we have $\lambda \left( t \right) ^TL\lambda \left( t \right) \le \frac{2\alpha ^2E\left( t_0 \right)}{t^2}$. Thus, the conclusion of \eqref{conver3,1} holds.
		
	Since $\dot{E}\left( t \right) \leq-\frac{\left( \alpha -2 \right) t}{\alpha ^2}\left( \mathbb{L} _1\left( x\left( t \right) ,y^*,\lambda ^* \right) -\mathbb{L} _1\left( x^*,y^*,\lambda \left( t \right) \right) \right) $, thus, integrating it from $t_0$ to $+\infty $ and rearranging it, we have
		\begin{align*}
				&\int_{t_0}^{+\infty}{t\left( \mathbb{L} _1\left( x\left( t \right) ,y^*,\lambda ^* \right) -\mathbb{L} _1\left( x^*,y^*,\lambda \left( t \right) \right) \right)}dt
				\\
				=&\frac{\alpha ^2}{\alpha -2}\left( E\left( t_0 \right) -E\left( +\infty \right) \right) \le \frac{\alpha ^2}{\alpha -2}E\left( t_0 \right) <+\infty.
			\end{align*}
	Therefore, the conclusion of \eqref{conver3,2} holds.
\end{proof}

\section{Optimization approaches for problem \eqref{P1} in the nonsmooth case}
For a vast amount of applications (e.g., signal processing, image processing, machine learning), nonsmooth functions are prevalent. To encompass these practical situations, we need to consider the case that the objective function $f(x)$ is nonsmooth. In order to adapt the APDMD \eqref{PDM} to this nonsmooth case, we will consider a corresponding smoothing accelerated primal-dual mirror  dynamical (SAPDMD) approaches based on smoothing approximation in \eqref{SA} as follows:
\begin{equation}\label{PDM_sm}
	\begin{cases}
		\dot{x}^{\mu}\left( t \right) =\frac{\alpha}{t}\left( \nabla \psi ^*\left( u^{\mu}\left( t \right) \right) -x^{\mu}\left( t \right) \right) ,
		\\
		\dot{u}^{\mu}\left( t \right) =-\frac{t}{\alpha}\left( \nabla _x\hat{f}\left( x^{\mu}\left( t \right) ,\mu \left( t \right) \right) +\beta A^T\left( Ax^{\mu}\left( t \right) -b \right) +A^Tv^{\mu}\left( t \right) \right) ,
		\\
		\dot{\lambda}^{\mu}\left( t \right) =\frac{\alpha}{t}\left( v^{\mu}\left( t \right) -\lambda ^{\mu}\left( t \right) \right) ,
		\\
		\dot{v}^{\mu}\left( t \right) =\frac{t}{\alpha}\left( A\nabla \psi ^*\left( u^{\mu}\left( t \right) \right) -b\left( t \right) \right) ,
		\\
		x^{\mu}\left( t_0 \right) =x_{0}^{\mu},u^{\mu}\left( t_0 \right) =u_{0}^{\mu}\,,\nabla \psi ^*\left( u_{0}^{\mu} \right) =x_{0}^{\mu}, 
		\\
		\lambda ^{\mu}\left( t_0 \right) =\lambda _{0}^{\mu},v^{\mu}\left( t_0 \right) =v_{0}^{\mu}, \mu \left( t_0 \right) =\mu _0,
	\end{cases} 
\end{equation}
where $t\geq t_0>0$, $\beta>0$, $\alpha \geq 2$ and $\mu \left( t \right) \leq \frac{\mu _0}{t^{2\mathrm{\alpha}}}$. The superscript $\mu$ means the trajectories that are obtained by the smoothing dynamical approaches, which is used to distinguish the trajectories obtained by the  dynamical approaches without smoothing approximation, i.e., APDMD \eqref{PDM}, ADPDMD \eqref{ADPDMD} and ADMD \eqref{ADMD}. The $\nabla_x \hat{f}\left( x^{\mu}\left( t \right) ,\mu \left( t \right) \right) $ is the gradient of smoothing function of $\hat{f}\left( x^{\mu}\left( t \right) ,\mu \left( t \right) \right) $ with respect of $x^{\mu}\left( t \right)$. The SAPDMD \eqref{PDM_sm} is illustrated in \eqref{fig:SAPDMD} (left).

Similar to APDMD \eqref{PDM}, the SAPDMD \eqref{PDM_sm} can be equivalent to the following smoothing second-order dynamical approach:
\begin{equation}
	\begin{cases}
		\ddot{x}^{\mu}\left( t \right) +\frac{\alpha +1}{t}\dot{x}^{\mu}\left( t \right) +\nabla ^2\psi ^*\left( \nabla \psi \left( x^{\mu}\left( t \right) +\frac{t}{\alpha}\dot{x}^{\mu}\left( t \right) \right) \right) \left( \nabla _x\hat{f}\left( x^{\mu}\left( t \right) ,\mu \left( t \right) \right) \right. 
		\\
		\,\,\,\,\,\,\,\,\,\,\,\,\,\,\,\,\,\,\,\,\,\,\,\,\,\,\,\,\,\,\,\,\,\,\,\,\,\,\,\,\,\,\,\,\,\,\,\,\left. +\beta A^T\left( Ax^{\mu}\left( t \right) -b \right) +A^T\left( \lambda ^{\mu}\left( t \right) +\frac{t}{\alpha}\dot{\lambda}^{\mu}\left( t \right) \right) \right) \,\,=0,
		\\
		\ddot{\lambda}^{\mu}\left( t \right) +\frac{\alpha +1}{t}\dot{\lambda}^{\mu}\left( t \right) -A\left( x^{\mu}\left( t \right) +\frac{\alpha}{t}\dot{x}^{\mu}\left( t \right) \right) +b=0,
		\\
		x^{\mu}\left( t_0 \right) =x_{0}^{\mu}\in \mathcal{X}, \dot{x}^{\mu}\left( t_0 \right) =\dot{x}_{0}^{\mu}, \lambda ^{\mu}\left( t_0 \right) =\lambda_{0}^{\mu}, \dot{\lambda}^{\mu}\left( t_0 \right) =\dot{\lambda}_{0}^{\mu}.                                  \\
	\end{cases}    
\end{equation}

\begin{figure}[htbp]
	\centering
	\includegraphics[width=15cm,height=7cm]{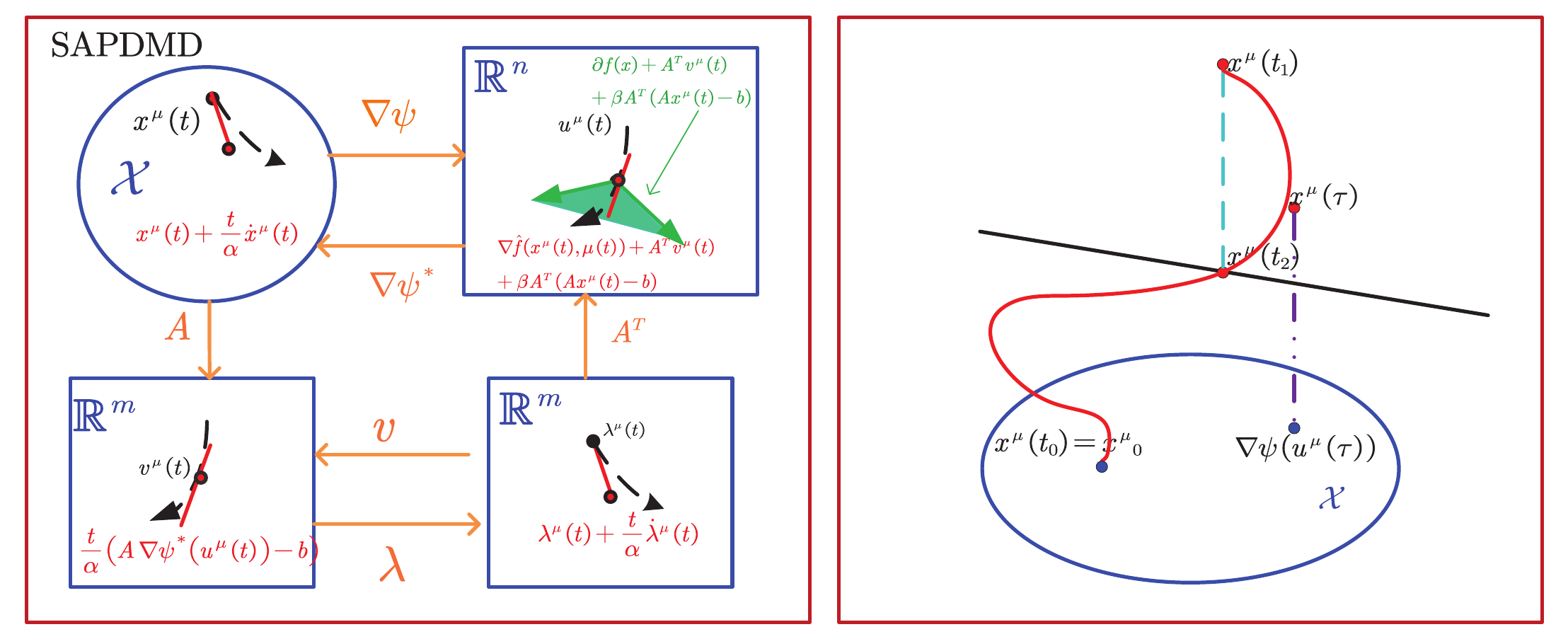}
	\caption{(left) The demonstration of SAPDMD \eqref{PDM_sm}. (right) The demonstration of the proof of feasibility of SAPDMD \eqref{PDM_sm}.}
	\label{fig:SAPDMD}
\end{figure}
\subsection{Nonsmooth dynamical approaches comparison} 
To solve nonsmooth optimization problems, many technologies have been studied and presented in designing dynamical optimized methods, such as, \cite{Cabot2007AsymptoticsFS, Apidopoulos2018TheDI, he2017second}, Moreau-Yosida regularization method \cite{Balavoine2013ConvergenceSO,Attouch2021FastCO} and directional derivative method \cite{Su2016ADE, Fazlyab2017AVA}.

Compared with the dynamical approaches based on techniques mentioned above for solving nonsmooth optimization problems, our proposed SAPDMD \eqref{PDM_sm} based on smoothing approximation in \eqref{SA} has the following differences and preponderances.
\begin{itemize}
	\item The solutions may be different. The SAPDMD \eqref{PDM_sm} has a strong global solution, which is similar to that in \eqref{stro_solu}. In \cite{Apidopoulos2018TheDI, Cabot2007AsymptoticsFS}, some inertial differential inclusion dynamical approaches are investigated, but they have shock solutions. The dynamical differential inclusion approaches proposed in \cite{he2017second, zeng2018distributed} have a Carath\'eodory's solution or Filippov's solution.
	
	\item The SAPDMD \eqref{PDM_sm} can achieve the optimal solution of problem \eqref{P1} in the nonsmooth case when $\mu \rightarrow 0$, which is different from the works in \cite{Nesterov2005SmoothMO,Zhu2016OptimalBR}. In the famous work \cite{Nesterov2005SmoothMO}, Nesterov's pioneered a smoothing method based on the Fenchel conjugate technique to deal with nonsmooth optimization problems and applied it to design accelerated algorithms with a $\varepsilon $-solution, i.e., $f\left( x^{\varepsilon} \right) -\min f\left( x \right) \leq \varepsilon$, where $\varepsilon $ is a small constant, but it's not equal to $0$.
	%It further implies that any optimal solutions of the algorithm based on Nesterov's smoothing technique can only be the $\varepsilon $-optimal solution of the original problem and are not guaranteed to be optimal. 
	
	\item The proposed SAPDMD \eqref{PDM_sm} does not require solving some subproblems. However, the accelerated dynamic approaches in \cite{Balavoine2013ConvergenceSO,Attouch2021FastCO} need to use the Moreau-Yosida approximation $f_{\rho}\left( x \right) =\underset{x\in \mathcal{X}}{\min}\left\{ f\left( \xi \right) +\frac{1}{2\rho}\left\| x-\xi \right\| ^2 \right\}$, the accelerated dynamic aprroaches \cite{Su2016ADE, Fazlyab2017AVA} need to utilize $d\left( x;\dot{x} \right) =\underset{f\in \partial f\left( x \right)}{\arg \max}f^T\dot{x}$, for them in general, there are no closed formulas available. This is undesirable from the point of view of numerical calculation. 	
	
	\item The existence and uniqueness of the global solutions for SAPDMD \eqref{PDM_sm} can be easily guaranteed by the Cauchy-Lipschitz-Picard theorem. However, the uniqueness of the solutions in directional derivative method \cite{he2017second} and differential inclusion method \cite{Su2016ADE, Fazlyab2017AVA} may not be guaranteed.
\end{itemize}

\subsection{Feasibility, existence and uniqueness of strong global solution for SAPDMD}  
In this subsection, we illustrate the feasibility, existence and uniqueness of a strong global solution of $x^{\mu}\left(t\right)$ for the SAPDMD  \eqref{PDM_sm} by the same way as in \eqref{solution1}. 

\begin{lemma}\label{feaibility2}
	For any $\left( x^{\mu}\left( t_0 \right) ,u^{\mu}\left( t_0 \right) ,\lambda ^{\mu}\left( t_0 \right) ,\upsilon ^{\mu}\left( t_0 \right) \right)  \in \mathcal{X} \times \mathbb{R}^n\times \mathbb{R} ^m\times \mathbb{R} ^m$ with $\nabla \psi ^*\left( u_{t_0}^{\mu} \right) =x_{t_0}^{\mu}$, then the variable $x^{\mu}\left( t \right) $ of SAPDMD is always in $\mathcal{X}, \forall \,\,t\geq t_0>0$, i.e., the feasibility of $x\left( t \right)$ is satisfied.
	\begin{proof}	
		The proof follows the same arguments as in \eqref{feasibility1} by the reductio and separation hyperplane theorem, we omit it here due to the limitation of space.
	\end{proof}

\end{lemma}

Now let us turn to the existence of strong global solution for \eqref{P1} with nonsmooth objective functions and we will once again take into account solutions (similar to \eqref{stro_solu}) to this problem. 
\begin{theorem}\label{ex_uni2}	
	For any initial values $\left( x^{\mu}\left( t_0 \right) ,u^{\mu}\left( t_0 \right), \lambda^{\mu} \left( t_0 \right) , \upsilon^{\mu} \left( t_0 \right)\right) \in \mathcal{X} \times \mathbb{R}^n\times \mathbb{R}^m\times \mathbb{R}^m$, there exists a unique strong global solution of SAPDMD \eqref{PDM_sm}.
\end{theorem}
	\begin{proof}
			Let $Y^{\mu}\left( t \right) =\left( x^{\mu}\left( t \right) ,u^{\mu}\left( t \right) ,\lambda^{\mu} \left( t \right) ,v^{\mu}\left( t \right) \right) $, then the  SAPDMD \eqref{PDM_sm} is equivalent to
			\begin{equation}
					\begin{cases}
							\dot{Y}^{\mu}\left( t \right) =H\left( t,Y^{\mu}\left( t \right) \right) \,\,                    \\
							Y^{\mu}\left(t_0 \right) =\left( x^{\mu}\left( t_0 \right) ,u^{\mu}\left( t_0 \right) , \lambda^{\mu} \left( t_0 \right), \upsilon^{\mu} \left( t_0 \right) \right),
						\end{cases}
				\end{equation}
			where 
	\begin{equation}\label{system2}
	\begin{split}		
	H\left( t,Y^{\mu} \right) 
		=&\left( \frac{\alpha}{t}\left( \nabla \psi ^*\left( u^{\mu}\left( t \right) \right) -x^{\mu}\left( t \right) \right) , \right. -\frac{t}{\alpha}\left( \nabla _x\hat{f}\left( x^{\mu}\left( t \right) ,\mu \left( t \right) \right) \right. 
	\\
	&\left. -A^Tv^{\mu}\left( t \right) \right) ,\left. \frac{\alpha}{t}\left( v^{\mu}\left( t \right) -\lambda ^{\mu}\left( t \right) \right) ,\frac{t}{\alpha}\left( A\nabla \psi ^*\left( u^{\mu}\left( t \right) \right) -b \right) \right).
	\end{split}
	\end{equation}
	To prove the existence and uniqueness of the strong global solution  $Y^{\mu}\left( t \right) $ generated by SAPDMD \eqref{PDM_sm} by the Cauchy-Lipschitz-Picard theorem \cite{Bolte2003SurDS}, the following conditions need to be satisfied:
			
			(I): For every $t\in \left[ t_0,+\infty \right) $, the mapping $H\left( t,\cdot \right) $ is $\mathsf{l} \left( t \right) $-Lipschitz continuous and $\mathsf{l} \left( \cdot \right) \in \mathrm{L}_{loc}^{1}\left( \left[ t_0,+\infty \right) \right) $.
			
			(II) For any $Y^{\mu}\in \mathcal{X} \times \mathbb{R}^n\times \mathbb{R}^m\times \mathbb{R}^m$, we have $H\left( \cdot ,Y^{\mu}\right) \in \mathrm{L}_{loc}^{1}\left( \left[t_0,+\infty \right),\mathcal{X} \times \mathbb{R}^n,\right. 
			\\
			\left. \times \mathbb{R}^m\times \mathbb{R}^m \right)$.
			
			The proof of (I). Let $t\in \left[ t_0,+\infty \right)$ be fixed and use the Lipschitz continuous of $\nabla \psi ^*$, $\nabla _xf\left( x,\mu \right) $. Then, for any $Y^{\mu}$, $\hat{Y}^{\mu}$, we have 
			\begin{equation*}
				\begin{split}
			&\,\,\,\,\,\,\left\| F\left( t,Y^{\mu}\left( t \right) \right) -F\left( t,\hat{Y}^{\mu}\left( t \right) \right) \right\| 
			\leq \left( \frac{t^2}{\alpha ^2} \right. \left( 4\beta ^2\delta _{\max}\left( A^TA \right) ^2+\frac{4\ell ^2}{\mu ^2\left( t \right)} \right. 
			\\
			&\left.+\delta _{\max}\left( A^TA \right) \mathfrak{l} _{\psi ^*}^{2} +2\delta _{\max}\left( A^TA \right) \right) \,\,\,\,\left. +\frac{\alpha ^2\left( 6+2\mathfrak{l} _{\psi ^*}^{2} \right)}{t^2} \right) ^{\frac{1}{2}}\left\| Y^{\mu}\left( t \right) -\hat{Y}^{\mu}\left( t \right) \right\|.
				\end{split}
		    \end{equation*}
	By using the notation $\mathsf{l} \left( t \right) =\left( \frac{t^2\left( 4\beta ^2\delta _{\max}\left( A^TA \right) ^2+\frac{4\ell ^2}{\mu ^2\left( t \right)}+\delta _{\max}\left( A^TA \right) \mathfrak{l} _{\psi ^*}^{2}+2\delta _{\max}\left( A^TA \right) \right)}{\alpha ^2} \right.
	\\
	\left. +\frac{\alpha ^2\left( 6+2\mathfrak{l} _{\psi ^*}^{2} \right)}{t^2} \right) ^{\frac{1}{2}}$, one has
	$\left\| H\left( t,Y^{\mu}\left( t \right) \right) -H\left( t,\hat{Y}^{u}\left( t \right) \right) \right\| \leq \mathsf{l} \left( t \right) \left\| Y^{\mu}\left( t \right)-\hat{Y}^{\mu}\left( t \right) \right\|.$ 
			
	Note that $\mathsf{l} \left( t \right) $ is continuous on $\left[ t_0,+\infty \right) $. Hence $\mathsf{l} \left( \cdot \right) $ is integrable on $\left[ t_0,\mathcal{T} \right] $ for all $t_0<\mathcal{T} <+\infty$.
			
	The proof of (II). Let $Y\in \mathcal{X} \times \mathbb{R}^n\times \mathbb{R}^m\times \mathbb{R}^m$ and $t_0<\mathcal{T}<+\infty$, it holds that  
	\begin{equation*}
	\begin{split}	
	&\int_{t_0}^{\mathcal{T}}{\left\| H\left( t,Y^{\mu}\left( t \right) \right) \right\|}dt
	\\
	\leq& \sqrt{\left\| \nabla \psi ^*\left( u\left( t \right) \right) \right\| ^2+\left\| x\left( t \right) \right\| ^2+\left\| \lambda \left( t \right) \right\| ^2+\left\| v\left( t \right) \right\| ^2+\left\| x^* \right\| ^2+\left\| \nabla _x\hat{f}\left( x^{\mu}\left( t \right) ,\mu \left( t \right) \right) \right\| ^2}
	\\
	&\times \left( \frac{8\alpha ^2}{t^2}+\frac{2t^2}{\alpha ^2}\left( 4\delta _{\max}\left( A^TA \right) +4\beta ^2\delta _{\max}\left( A^TA \right) ^2+2 \right) \right),
		\end{split}
    \end{equation*}
and the conclusion holds by employing the continuity of the function 
\begin{equation*}
\begin{split}	
t\rightarrow \left( \frac{8\alpha ^2}{t^2}+\frac{2t^2}{\alpha ^2}\left( 4\delta _{\max}\left( A^TA \right) +4\beta ^2\delta _{\max}\left( A^TA \right) ^2+2 \right) \right) ^{\frac{1}{2}}  \text{on}\,\,\left[ t_0,\mathcal{T} \right].
\end{split}
\end{equation*}

The existence and uniqueness of $Y^{\mu}\left( t \right)$ to the dynamical system \eqref{system2} can be guaranteed by  the Cauchy-Lipschitz-Picard theorem, consequently, the existence and uniqueness of the trajectories of SAPDMD \eqref{PDM_sm} also hold.	
\end{proof}

\subsection{The accelerated convergence of the SAPDMD} In this subsection, we will illustrate the accelerated convergence properties of the SAPDMD \eqref{PDM_sm} based on the Lyapunov analysis method.

A natural question is whether the Lyapunov analysis method in the smooth case is still effective in the nonsmooth case. The answer is affirmative, provided some care is taken in the main three steps of our analysis. First, a time-dependent parameter $\mu \left( t \right)$ needs to be introduced in the Lyapunov function. Second, when taking the time derivative of Lyapunov function, it requires utilizing the full differentiation of the smoothing approximation function, i.e., $\nabla _{\mu}\hat{f}\left( x\left( t \right) ,\mu \left( t \right) \right)$ needs to be considered. The third factor is  boundedness of gradient with respect to $\mu\left( t \right) $. In turn, all the results and estimation we have presented in the previous sections can be transferred to this more general nonsmooth context.

\begin{theorem}\label{Theorem_sapdmd}
	Suppose that Assumption 3.1 holds and the objective function is nonsmooth convex. Let $\left(x\left( t \right), y\left( t \right)\right)$ and $\left(x^*, y^*\right)$ be a solution trajectory and an optimal solution of SAPDMD \eqref{PDM_sm}, respectively. Defining a Lyapunov function as   $\hat{\mathcal{L}}_{\beta}\left( x^{\mu}\left( t \right) ,\lambda ^{\mu}\left( t \right) \right) =\hat{f}\left( x^{\mu}\left( t \right) ,\mu \left( t \right) \right) +\lambda ^{\mu}\left( t \right) ^T\left( Ax^{\mu}\left( t \right) -b \right) +\frac{\beta}{2}\left\| Ax^{\mu}\left( t \right) -b \right\| ^2$ with $\beta \geq 0$, then, for any $\left( x^{\mu}\left( t \right) ,\lambda ^{\mu}\left( t \right) \right) \in \mathcal{X} \times \mathbb{R} ^m$ , we have 
	
	(I): The trajectories of $x^{\mu}\left( t \right)$ and $\lambda^{\mu} \left( t \right)$ are bounded for any $t\geq t_0>0$.
	
	(II): Then, we have 
	\begin{subequations}
		\begin{align}
			&0\le \hat{\mathcal{L}}_{\beta}\left( x^{\mu}\left( t \right) ,\lambda ^* \right) -\hat{\mathcal{L}}_{\beta}\left( x^*,\lambda ^{\mu}\left( t \right) \right) +4\kappa _{\hat{f}}\mu \left( t \right) \le \frac{\alpha ^2\mathrm{V}\left( t_0 \right)}{t^2}, \label{sc1}
			\\
			&\left\| Ax^{\mu}\left( t \right) -b \right\| ^2\le \frac{2\alpha ^2\mathrm{V}\left( t_0 \right)}{\beta t^2}, \label{sc2}
			\\
			&\int_{t_0}^{+\infty}{\left\| Ax^{\mu}\left( t \right) -b \right\| ^2}\le +\infty, \label{sc3}
			\\
			&\int_0^{+\infty}{t\left( \hat{\mathcal{L}}_{\beta}\left( x^{\mu}\left( t \right) ,\lambda ^* \right) -\hat{\mathcal{L}}_{\beta}\left( x^*,\lambda ^{\mu}\left( t \right) \right) +2\kappa _{\hat{f}}\mu \left( t \right) \right) dt}<+\infty, \label{sc4}
		\end{align}
	\end{subequations}
	
	(III): Let $\lambda ^*=\begin{cases}
		0, \,\,\,\,\,\,\,\,\,\,\,\,\,\,\,\,\,\,\,\,\,  if\,\,Ax^{\mu}\left( t \right) -b=0,\\
		\frac{Ax^{\mu}\left( t \right) -b}{\left\| Ax^{\mu}\left( t \right) -b \right\|},if\,\,Ax^{\mu}\left( t \right) -b\ne 0,\\
	\end{cases}$ and for any $t\geq t_0>0$, it follows that 
	\begin{subequations}
		\begin{align}
			&0\leq f\left( x^{\mu}\left( t \right) \right) -f\left( x^* \right) +\left\| Ax^{\mu}\left( t \right) -b \right\| \leq \frac{\alpha ^2V\left( t_0 \right)}{t^2};\label{sc5}
			\\
			&-\frac{\alpha \sqrt{2\mathrm{V}\left( t_0 \right)}}{t\sqrt{\beta}}\le f\left( x^{\mu}\left( t \right) \right) -f\left( x^* \right) \le \frac{\alpha ^2\mathrm{V}\left( t_0 \right)}{t^2}. \label{sc6}
		\end{align}
	\end{subequations}
\end{theorem}
\begin{proof}
	\textit{(I)}: Consider a well-designed smooth Lyapunov function 
	\begin{equation}\label{SV2}
		\mathrm{V}\left( t \right) =\mathrm{V}_1\left( t \right) +\mathrm{V}_2\left( t \right) +\mathrm{V}_3\left( t \right),
		\end{equation}
with 
\begin{align*}
\begin{cases}
\mathrm{V}_1\left( t \right) =\frac{t^2}{\alpha ^2}\left( \hat{\mathcal{L}}_{\beta}\left( x^{\mu}\left( t \right) ,\lambda ^* \right) -\hat{\mathcal{L}}_{\beta}\left( x^*,\lambda ^{\mu}\left( t \right) \right) +4\kappa _{\hat{f}}\mu \left( t \right) \right) 
\\
\,\,\,\,\,\,\,\,\,\,\,\,\,\,\,\,\,=\frac{t^2}{\alpha ^2}\left( \hat{f}\left( x^{\mu}\left( t \right) ,\mu \left( t \right) \right) -f\left( x^*,\mu \left( t \right) \right) \right. +4\kappa _{\hat{f}}\mu \left( t \right) 
\\
\,\,\,\,\,\,\,\,\,\,\,\,\,\,\,\,\,\,\,\,\left. +\frac{\beta}{2}\left\| Ax^{\mu}\left( t \right) -b \right\| ^2+\left( \lambda ^* \right) ^TA\left( x^{\mu}\left( t \right) -x^* \right) \right) 
\\
\mathrm{V}_2\left( t \right) =D_{\psi ^*}\left( u^{\mu}\left( t \right) ,u^* \right) =\psi ^*\left( u^{\mu}\left( t \right) \right) -\psi ^*\left( u^* \right) -\nabla \psi ^*\left( u^* \right) ^T\left( u^{\mu}\left( t \right) -u^* \right) ,
\\
\mathrm{V}_3\left( t \right) =D_h\left( v^{\mu}\left( t \right) ,\lambda ^* \right) =\frac{1}{2}\left\| v^{\mu}\left( t \right) -\lambda ^* \right\| ^2.
\end{cases}
\end{align*}
%where $D_{\psi ^*}\left( u^{\mu}\left( t \right) , u^* \right)$, $D_h\left( v^{\mu}\left( t \right),\lambda^* \right) $ are respectively the Bregman divergence of $\psi ^*$ for $u^{\mu}\left( t \right) , u^*$ and $h$ for $v^{\mu}\left( t \right),\lambda ^*$.
%The non-negative intem $\kappa _{\hat{f}}\mu \left( t \right)$ is introducted in the $\mathrm{V}_1\left( t \right)$.  
Note that the function $\mathrm{V}_1\left( t \right)$ is positive for any $t\geq t_0>0$ since  
	\begin{align*}
	&\hat{\mathcal{L}}_{\beta}\left( x^{\mu}\left( t \right) ,\lambda ^* \right) -\hat{\mathcal{L}}_{\beta}\left( x^*,\lambda ^{\mu}\left( t \right) \right) +4\kappa _{\hat{f}}\mu \left( t \right) 
	\\
	\ge& f\left( x^{\mu}\left( t \right) \right) -f\left( x^* \right) +\left( \lambda ^* \right) ^TA\left( x^{\mu}\left( t \right) -x^* \right) +\frac{\beta}{2}\left\| Ax^{\mu}\left( t \right) -b \right\| ^2+2\kappa _{\hat{f}}\mu \left( t \right) 
	\\
	=& \mathcal{L} _{\beta}\left( x^{\mu}\left( t \right) ,\lambda ^* \right) -\mathcal{L} _{\beta}\left( x^*,\lambda ^{\mu}\left( t \right) \right) +2\kappa _{\hat{f}}\mu \left( t \right) \ge 0,
		\end{align*}
where the first equality holds from $\hat{f}\left( x^{\mu}\left( t \right) ,\mu \left( t \right) \right) -f\left( x^{\mu}\left( t \right) \right) +\kappa _{\hat{f}}\mu \left( t \right) \geq 0$, $f\left( x^* \right) -f\left( x^*,\mu \left( t \right) \right) +\kappa _{\hat{f}}\mu \left( t \right) \geq 0$ in \eqref{Defi_sm} \textit{(iv)}, and the second inequality is satisfied due to \eqref{sd} in \eqref{saddle}. In addition, since $\psi ^*\left( \cdot \right)$ and $h\left( \cdot \right) =\frac{1}{2}\left\| \cdot \right\| ^2$ are convex, such that $D_{\psi ^*}\left( u^{\mu}\left( t \right), u^* \right) \geq 0$, $ D_h\left( v^{\mu}\left( t \right),\lambda ^* \right) \geq0$. Thus, we end up with  $\mathrm{V}\left( t \right) \geq0$.

The derivatives of $\mathrm{V}_1\left( t \right) $, $\mathrm{V}_2\left( t \right) $ and $\mathrm{V}_3\left( t \right) $ along the trajectory of SAPDMD \eqref{PDM_sm} satisfy 
	
\begin{equation}\label{sV1}
\begin{split}
\mathrm{\dot{V}}_1\left( t \right) 
=&\frac{2t}{\alpha ^2}\left( \begin{array}{c}
		\hat{f}\left( x^{\mu}\left( t \right) ,\mu \left( t \right) \right) -f\left( x^*,\mu \left( t \right) \right) +\left( \lambda ^* \right) ^TA\left( x^{\mu}\left( t \right) -x^* \right)\\
	\end{array} \right. 
	\\
	&\left. +4\kappa _{\hat{f}}\mu \left( t \right) +\frac{\beta}{2}\left\| Ax^{\mu}\left( t \right) -b \right\| ^2 \right) +\frac{t}{\alpha}\left( \nabla \psi ^*\left( u^{\mu}\left( t \right) \right) -\nabla \psi ^*\left( u^* \right) \right) ^T
	\\
	&\times \left( \nabla _x\hat{f}\left( x^{\mu}\left( t \right) ,\mu \left( t \right) \right) +\beta A^T\left( Ax^{\mu}\left( t \right) -b \right) +A^T\lambda ^* \right) 
	\\
	&+\frac{t}{\alpha}\left( \nabla _x\hat{f}\left( x^{\mu}\left( t \right) ,\mu \left( t \right) \right) +\beta A^T\left( Ax^{\mu}\left( t \right) -b \right) +A^T\lambda ^* \right) ^T
	\\
	&\times \left( \nabla \psi ^*\left( u^* \right) -x^{\mu}\left( t \right) \right) +\frac{2t^2}{\alpha ^2}\kappa _{\hat{f}}\dot{\mu}\left( t \right) 
	\\
	&+\frac{t^2}{\alpha ^2}\left( \nabla _{\mu}\hat{f}\left( x^{\mu}\left( t \right) ,\mu \left( t \right) \right) +\nabla _{\mu}\hat{f}\left( x^*,\mu \left( t \right) \right) +2\kappa _{\hat{f}} \right) \dot{\mu}\left( t \right),
	\end{split}
	\end{equation}
		\begin{equation}\label{sV2}
			\begin{split}
			\mathrm{\dot{V}}_2\left( t \right) =&-\frac{t}{\alpha}\left( \nabla \psi ^*\left( u^{\mu}\left( t \right) \right) -\nabla \psi ^*\left( u^* \right) \right) ^T\left( A^Tv^{\mu}\left( t \right) -A^T\lambda ^* \right) 
			\\
			&-\left( \nabla _x\hat{f}\left( x^{\mu}\left( t \right) ,\mu \left( t \right) \right) +\beta A^TA\left( x^{\mu}\left( t \right) -x^* \right) +A^T\lambda ^* \right) ^T
			\\
			&\times \frac{t}{\alpha}\left( \nabla \psi ^*\left( u^{\mu}\left( t \right) \right) -\nabla \psi ^*\left( u^* \right) \right),
				\end{split}
			\end{equation}
	
	\begin{align}\label{sV3}
	\mathrm{\dot{V}}_3\left( t \right) =\left( v^{\mu}\left( t \right) -\lambda ^* \right) ^T\dot{v}^{\mu}\left( t \right) =\frac{t}{\alpha}\left( v^{\mu}\left( t \right) -\lambda ^* \right) ^TA\left( \nabla \psi ^*\left( u^{\mu}\left( t \right) \right) -x^* \right) .
	\end{align}	
		
	Combining \eqref{sV1}, \eqref{sV2}, \eqref{sV3} and rearranging it yields
	\begin{equation}\label{sVend}
	\begin{split}
		\mathrm{\dot{V}}\left( t \right) =&\mathrm{\dot{V}}_1\left( t \right) +\mathrm{\dot{V}}_2\left( t \right) +\mathrm{\dot{V}}_3\left( t \right) 
		\\
		\le&\frac{2t}{\alpha ^2}\left( \hat{f}\left( x^{\mu}\left( t \right) ,\mu \left( t \right) \right) -\hat{f}\left( x^*,\mu \left( t \right) \right) +\frac{\beta}{2}\left\| Ax^{\mu}\left( t \right) -b \right\| ^2 \right. +4\kappa _{\hat{f}}\mu \left( t \right) 
		\\
		&\left. +\left( \lambda ^* \right) ^TA\left( x^{\mu}\left( t \right) -x^* \right) \right) -\frac{t}{\alpha}\left( \hat{f}\left( x^{\mu}\left( t \right) ,\mu \left( t \right) \right) -\hat{f}\left( x^*,\mu \left( t \right) \right) \right. 
		\\
		&\left. +\beta \left\| Ax^{\mu}\left( t \right) -b \right\| ^2+\left( \lambda ^* \right) ^TA\left( x^{\mu}\left( t \right) -x^* \right) \right) +\frac{2t^2}{\alpha ^2}\kappa _{\hat{f}}\dot{\mu}\left( t \right) 
		\\
	\le& \frac{\left( 2-\mathrm{\alpha} \right) t}{\alpha ^2}\left( \hat{f}\left( x^{\mu}\left( t \right) ,\mu \left( t \right) \right) -\hat{f}\left( x^*,\mu \left( t \right) \right) ^2 \right. +4\kappa _{\hat{f}}\mu \left( t \right) 
		\\
	&\left. +\left( \lambda ^* \right) ^TA\left( x^{\mu}\left( t \right) -x^* \right) +\frac{\beta}{2}\left\| Ax^{\mu}\left( t \right) -b \right\| \right) 
		\\
			&-\frac{t\beta}{2\alpha}\left\| Ax^{\mu}\left( t \right) -b \right\| ^2+\frac{4t}{\alpha}\kappa _{\hat{f}}\mu \left( t \right) +\frac{2t^2}{\alpha ^2}\kappa _{\hat{f}}\dot{\mu}\left( t \right) 
			\\
			\le& -\frac{\left( \mathrm{\alpha}-2 \right) t}{\alpha ^2}\left( \hat{\mathcal{L}}_{\beta}\left( x^{\mu}\left( t \right) ,\lambda ^* \right) -\hat{\mathcal{L}}_{\beta}\left( x^*,\lambda ^{\mu}\left( t \right) \right) \right. 
			\left. +4\kappa _{\hat{f}}\mu \left( t \right) \right)
			\\
			 &-\frac{t\beta}{2\alpha}\left\| Ax^{\mu}\left( t \right) -b \right\| ^2
				\le 0,\forall \,\,x^{\mu}\left( t \right) \in \mathcal{X} ,t\ge t_0>0,
		\end{split}
		\end{equation}	
			where the first inequality holds since $\nabla _{\mu}\hat{f}\left( x^{\mu}\left( t \right) ,\mu \left( t \right) \right) +\nabla _{\mu}\hat{f}\left( x^*,\mu \left( t \right) \right) +2\kappa _{\hat{f}}\geq 0$ and $\dot{\mu}\left( t \right) \leq 0$, the second inequality is satisfied because of the convexity of $\hat{f}\left( x^{\mu}\left( t \right),\mu \left( t \right) \right)$ of $x^{\mu}\left( t \right) \in \mathcal{X} $ with any fixed $\mu \left( t \right) \in \left[ 0,\bar{\mu} \right] $, the third inequality holds from $\mu \left( t \right) \leq \frac{\mu _0}{t^{2\mathrm{\alpha}}}$ (i.e., $\dot{\mu}\left( t \right) \leq-\mu _0\left( 2\mathrm{\alpha} \right) t^{\left( -2\alpha -1 \right)}$ and $\frac{t^2}{\alpha ^2}\kappa _{\hat{f}}\geq 0$ imply $\frac{2t}{\alpha}\kappa _{\hat{f}}\mu \left( t \right) +\frac{t^2}{\alpha ^2}\kappa _{\hat{f}}\dot{\mu}\left( t \right) \leq \frac{2t}{\alpha}\kappa _{\hat{f}}\mu _0t^{-2\alpha}-\frac{t^2}{\alpha ^2}\kappa _{\hat{f}}\mu _0\left( 2\mathrm{\alpha} \right) t^{-2\alpha -1}=\frac{2t}{\alpha}\kappa _{\hat{f}}\mu _0t^{-2\alpha}-\frac{2t}{\alpha ^2}\kappa _{\hat{f}}\mu _0t^{-2\alpha}=0$), and the last inequality holds since $\mathrm{\alpha}\geq 2$ and $\hat{\mathcal{L}}_{\beta}\left( x^{\mu}\left( t \right) ,\lambda ^* \right) -\hat{\mathcal{L}}_{\beta}\left( x^*,\lambda ^{\mu}\left( t \right) \right) +4\kappa _{\hat{f}}\mu \left( t \right) \geq0$.
			
		Recall that the Lyapunov function $\mathrm{V}\left( t \right)$ is radically unbounded and positive with any $t\geq t_0>0$. This implies that the trajectories of $ x^{\mu}\left( t \right)$ and $\lambda^{\mu} \left( t \right)$ are bounded for any $t\geq t_0>0$. The proof of (I) is therefore completed.
			
		\textit{(II)}: From \eqref{sVend}, we have that $\dot{\mathrm{V}}\left( t \right)\leq0$, i.e., $\mathrm{V}\left( t \right)$ is nonincreasing on $t \in\left[t_0,+\infty \right)$. It yields to $\frac{t^2}{\alpha ^2}\left( \hat{\mathcal{L}}_{\beta}\left( x^{\mu}\left( t \right) ,\lambda ^* \right) -\hat{\mathcal{L}}_{\beta}\left( x^*,\lambda ^{\mu}\left( t \right) \right) +4\kappa _{\hat{f}}\mu \left( t \right) \right) =\mathrm{V}_1\left( t \right) \leq \mathrm{V}\left( t\right) \leq \mathrm{V}\left( t_0 \right)$, i.e., \eqref{sc1} holds. Furthermore, the \eqref{sc1} implies $\left\| Ax^\mu \left( t \right)\left( t \right) -b \right\| ^2$\\$\leq \frac{2\alpha ^2V\left( t_0 \right)}{\beta t^2}$, i.e., \eqref{sc2} is satisfied. 
			
		Note that $\mathrm{\dot{V}}\left( t \right) \leq-\frac{\left( \mathrm{\alpha}-2 \right) t}{\alpha ^2}\left( \hat{\mathcal{L}}_{\beta}\left( x^{\mu}\left( t \right),\lambda ^* \right) -\hat{\mathcal{L}}_{\beta}\left( x^*,\lambda ^{\mu}\left( t \right) \right) +4\kappa _{\hat{f}}\mu \left( t \right) \right)$\\$-\frac{t\beta}{2\alpha}\left\| Ax^{\mu}\left( t \right) -b \right\| ^2 $ in \eqref{sVend}, 
		and integrating it from $t_0$ to $+\infty $ and rearranging it, we have
			\begin{align*}
				\frac{\alpha -2}{\alpha ^2}&\int_{t_0}^{+\infty}{t\left( \hat{\mathcal{L}}_{\beta}\left( x^{\mu}\left( t \right) ,\lambda ^* \right) -\hat{\mathcal{L}}_{\beta}\left( x^*,\lambda ^{\mu}\left( t \right) \right) +4\kappa _{\hat{f}}\mu \left( t \right) \right)}dt
				\\
				=&\mathrm{V}\left( t_0 \right) -\mathrm{V}\left( +\infty \right) \le V\left( t_0 \right) <+\infty,
				\\
				\frac{\beta}{2\alpha}&\int_{t_0}^{+\infty}{t\left\| Ax^{\mu}\left( t \right) -b \right\| ^2}dt=\mathrm{V}\left( t_0 \right) -\mathrm{V}\left( +\infty \right) \le V\left( t_0 \right) <+\infty,
				\end{align*}
			which implies both \eqref{sc3} and  \eqref{sc4} hold.
			
			\textit{(III)}: The equations $\lambda ^*=\begin{cases}
					0, \,\,\,\,\,\,\,\,\,\,\,\,\,\,\,\,\,\,\,\,\,\,\,\,\, if\,\,Ax^{\mu}\left( t \right) -b=0,\\
					\frac{Ax^{\mu}\left( t \right) -b}{\left\| Ax^{\mu}\left( t \right) -b \right\|},if\,\,Ax^{\mu}\left( t \right) -b\ne 0,\\
				\end{cases}$ and \eqref{sc1} give $\left\| Ax^{\mu}\left( t \right) -b \right\| \leq\frac{\alpha \sqrt{2V\left( t_0 \right)}}{\sqrt{\beta}t}$. 
			
			According to the upper-bound of \eqref{sd} in \eqref{saddle} with $\beta =0$, one has 
		\begin{equation*}
			\begin{split}	
					&\mathcal{L} _0\left( x^{\mu}\left( t \right) ,\lambda ^* \right) -\mathcal{L} _0\left( x^*,\lambda ^{\mu}\left( t \right) \right) 
					\\
					=&f\left( x^{\mu}\left( t \right) \right) -f\left( x^* \right) \geq -\left( Ax^{\mu}\left( t \right) -b \right) ^T\lambda ^*, \forall \,\,\left( x^{\mu}\left( t \right) ,\lambda ^{\mu}\left( t \right) \right) \in \mathcal{X} \times \mathbb{R}^m,
				\end{split}
		     \end{equation*}	
	and from  (iv) in \eqref {Defi_sm}, one also has 
	 \begin{equation*}
		\begin{split}
		 &\hat{f}\left( x^{\mu}\left( t \right) ,\mu \left( t \right) \right) -\hat{f}\left( x^*,\mu \left( t \right) \right) +4\kappa _{\hat{f}}\mu \left( t \right) 
		 \\
		 \geq& f\left( x^{\mu}\left( t \right) \right) -f\left( x^* \right) \geq -\left( Ax^{\mu}\left( t \right) -b \right) ^T\lambda ^*
		 \\
		 \geq&-\left\| Ax^{\mu}\left( t \right) -b \right\| \ge -\frac{\alpha \sqrt{2V\left( t_0 \right)}}{t\sqrt{\beta}}, \forall \,\,\left( x^{\mu}\left( t \right), \lambda ^{\mu}\left( t \right) \right) \in \mathcal{X} \times \mathbb{R}^m.
		\end{split}
		 \end{equation*}	
	
	Since $\left\| Ax^{\mu}\left( t \right) -b \right\| \geq 0$ and \eqref{sc1} imply $f\left( x^{\mu}\left( t \right) \right) -f\left( x^* \right) \leq \frac{\alpha ^2\mathrm{V}\left( t_0 \right)}{t^2}$, further, we deduce for any $t\geq t_0>0$ that $-\frac{\alpha \sqrt{2\mathrm{V}\left( t_0 \right)}}{t\sqrt{\beta}}\le f\left( x^{\mu}\left( t \right) \right) -f\left( x^* \right) \le \frac{\alpha ^2\mathrm{V}\left( t_0 \right)}{t^2}$ is true. Therefore, we can obtain the conclusion (III).
	\end{proof}

Similar to the smooth case, the SAPDMD \eqref{PDM_sm} can also be used to address the DCCP \eqref{P2} and DEMO \eqref{P3} with nonsmooth convex objections.

\subsection{SAPDMD for DCCP in the nonsmooth case} Based on the SAPDMD \eqref{PDM_sm}, a smoothing accelerated dirtibuted primal-dual mirror dynamical (SADPDMD) approach for DCCP \eqref{P2} with  nonsmooth convex objective functions is given by:
\begin{align}\label{SADPDMD}
	\begin{cases}
		\dot{x}^{\mu}\left( t \right) =\frac{\alpha}{t}\left( \nabla \psi ^*\left( u^{\mu}\left( t \right) \right) -x^{\mu}\left( t \right) \right) ,
		\\
		\dot{u}^{\mu}\left( t \right)=-\frac{t}{\alpha}\left( \nabla _x\hat{f}\left( x^{\mu}\left( t \right) ,\mu \left( t \right) \right) +\beta Lx^{\mu}\left( t \right)+L\left(v^{\mu}\left( t \right) \right) \right),
		\\
		\dot{\lambda}^{\mu}\left( t \right) =\frac{\alpha}{t}\left( v^{\mu}\left( t \right) -\lambda ^{\mu}\left( t \right) \right), \dot{v}^{\mu}\left( t \right) =\frac{t}{\alpha}L\nabla \psi ^*\left( u^{\mu}\left( t \right) \right) ,
		\\
		x^{\mu}\left( t_0 \right) =x_{0}^{\mu},\nabla \psi ^*\left( u^{\mu}\left( t_0 \right) \right) =x_{0}^{\mu},\lambda ^{\mu}\left( t_0 \right) =\lambda _{0}^{\mu}, v^{\mu}\left( t_0 \right) =v_{0}^{\mu}, \mu _0=\bar{\mu}>0,                             \\
	\end{cases}
\end{align}
where $t\geq t_0>0$, $\beta>0$, $\alpha \geq 2$, and $\mu \left( t \right) \leq \frac{\mu _0}{t^{2\mathrm{\alpha}}}$.

\subsection{The accelerated convergence of the SADPDMD} Now, let's discuss accelerated convergence properties of SADPDMD \eqref{SADPDMD} with the help of Lyapunov analysis tool.
\begin{theorem}
	Suppose that Assumption 3.2 holds and the objective function is nonsmooth. Let $\left(x^{\mu}\left( t \right), \lambda^{\mu}\left( t \right) \right) $ and $\left(x^*, \lambda^*\right) $ be a solution trajectory and an optimal solution of SADPDMD \eqref{SADPDMD}, respectively. Let $\hat{\mathsf{L}}_{\beta}\left( x^{\mu}\left( t \right) ,\lambda ^{\mu}\left( t \right) \right) =\hat{f}\left( x^{\mu}\left( t \right) ,\mu \left( t \right) \right)$\\$+\frac{\beta}{2}\left( x^{\mu}\left( t \right) \right) ^TL\left( x^{\mu}\left( t \right) \right) +\left( \lambda ^{\mu}\left( t \right) \right) ^TLx^{\mu}\left( t \right) $ with $\beta \geq 0$, then, for $\alpha \geq2$, one has 
	
	\textit{(I)}: $ x^{\mu}\left( t \right)$, $\lambda^{\mu} \left( t \right)$, $u^{\mu}\left( t \right)$, $v^{\mu}\left( t \right) $ are bounded for any $t\geq t_0>0$.
	
	(II): Then, we have 
	\begin{subequations}
		\begin{align}
			&0\le \hat{\mathsf{L}}_{\beta}\left( x^{\mu}\left( t \right) ,\lambda ^* \right) -\hat{\mathsf{L}}_{\beta}\left( x^*,\lambda ^{\mu}\left( t \right) \right) +4\kappa _{\hat{f}}\mu \left( t \right) \le \frac{\alpha ^2\mathbb{V} \left( t_0 \right)}{t^2},\label{dsc1}
			\\
			&x^{\mu}\left( t \right) Lx^{\mu}\left( t \right) \le \frac{2\alpha ^2\mathbb{V} \left( t_0 \right)}{\beta t^2},\label{dsc2}
			\\
			&\int_{t_0}^{+\infty}{x^{\mu}\left( t \right) Lx^{\mu}\left( t \right)}\le +\infty,\label{dsc3}
			\\
			&\int_{t_0}^{+\infty}{t\left( \hat{\mathsf{L}}_{\beta}\left( x^{\mu}\left( t \right) ,\lambda ^* \right) -\hat{\mathsf{L}}_{\beta}\left( x^*,\lambda ^{\mu}\left( t \right) \right) +2\kappa _{\hat{f}}\mu \left( t \right) \right) dt}<+\infty, \label{dsc4}
		\end{align}
	\end{subequations}
	
	(III): Let $\lambda ^*=\begin{cases}
		0, \,\,\,\,\,\,\,\,\,\,\,\,\,\,\,\,\,\, if\,\,Lx^{\mu}\left( t \right)=0,\\
		\frac{Lx^{\mu}\left( t \right)}{\left\| Lx^{\mu}\left( t \right)\right\|},if\,\,Lx^{\mu}\left( t \right) \ne 0,\\
	\end{cases}$. For any $t\geq t_0>0$, it follows that 
	\begin{subequations}
		\begin{align}
			&f\left( x^{\mu}\left( t \right) \right) -f\left( x^* \right) +\left\| Lx^{\mu}\left( t \right)\right\| \leq \frac{\alpha ^2 \mathbb{V}\left( t_0 \right)}{t^2};\label{dsc5}
			\\
			&-\frac{\alpha \sqrt{2\mathbb{V}\left( t_0 \right)}}{t\sqrt{\beta}}\le f\left( x^{\mu}\left( t \right) \right) -f\left( x^* \right) \le \frac{\alpha ^2\mathbb{V}\left( t_0 \right)}{t^2}. \label{dsc6}
		\end{align}
	\end{subequations}
\end{theorem}
\begin{proof}
	Construct the candidate smoothing Lyapunov function as  
	\begin{align}\label{energy_SD1}
		\mathbb{V} =\mathbb{V}_1+\mathbb{V}_2+\mathbb{V}_3,
	\end{align}
	with 
	\begin{align*}	
		\begin{cases}
			\mathbb{V} _1\left( t \right) =&\frac{t^2}{\alpha ^2}\left( \hat{\mathsf{L}}_{\beta}\left( x^{\mu}\left( t \right) ,\lambda ^* \right) - \hat{\mathsf{L}}_{\beta}\left( x^*,\lambda ^{\mu}\left( t \right) \right) +4\kappa _{\hat{f}}\mu \left( t \right) \right),
			\\
			\mathbb{V} _2\left( t \right) =&D_{\psi ^*}\left( u^{\mu}\left( t \right) ,u^* \right) =\psi ^*\left( u^{\mu}\left( t \right) \right) -\psi ^*\left( u^* \right) -\nabla \psi ^*\left( u^* \right) ^T\left( u^{\mu}\left( t \right) -u^* \right),
			\\
			\mathbb{V} _3\left( t \right) =&D_h\left( v^{\mu}\left( t \right) ,\lambda ^* \right) =\frac{1}{2}\left\| v^{\mu}\left( t \right) -\lambda ^* \right\| ^2.
		\end{cases}
	\end{align*}
	
	Following similar steps as the proof in \eqref{Theorem_sapdmd}, the conclusions can be obtained easily. Due to space limitations, we omit the proof here. 
\end{proof}

\subsection{SADMD for DEMO in the nonsmooth case} 
To address the DEMO \eqref{P3} where the objective function is nonsmooth, a smoothing accelerated distributed mirror dynamical (SADMD) approach inspired by the SAPDMD \eqref{PDM_sm} and distributed consensus theorem is given by 
\begin{align}\label{SADMD}
	\begin{cases}
		\dot{x}^{\mu}\left( t \right) =\frac{\alpha}{t}\left( \nabla \psi ^*\left( u^{\mu}\left( t \right) \right) -x^{\mu}\left( t \right) \right) ,
		\\
		\dot{u}^{\mu}\left( t \right) =-\frac{t}{\alpha}\left( \nabla _x\hat{f}\left( x^{\mu}\left( t \right) ,\mu \left( t \right) \right) +\bar{A}^Tv^{\mu}\left( t \right) \right) ,
		\\
		\dot{\lambda}^{\mu}\left( t \right) =\frac{\alpha}{t}\left( v^{\mu}\left( t \right) -\lambda ^{\mu}\left( t \right) \right) ,
		\\
		\dot{v}^{\mu}\left( t \right) =\frac{t}{\alpha}\left( \bar{A}\psi ^*\left( u^{\mu}\left( t \right) \right) -d-L\lambda ^{\mu}\left( t \right) +Lz^{\mu}\left( t \right) \right) ,
		\\
		\dot{y}^{\mu}\left( t \right) =\frac{\alpha}{t}\left( z^{\mu}\left( t \right) -y^{\mu}\left( t \right) \right) ,\dot{z}\left( t \right) =-\frac{t}{\alpha}Lv^{\mu}\left( t \right) ,
		\\
		x^{\mu}\left( t_0 \right) =x_{0}^{\mu},u^{\mu}\left( t_0 \right) =u_{0}^{\mu}\,\,\mathrm{with}\,\,\nabla \psi ^*\left( u_0 \right) =x_{0}^{\mu}\in \mathcal{X} ,
		\\
		\lambda ^{\mu}\left( t_0 \right) =\lambda _{0}^{\mu},v^{\mu}\left( t_0 \right) =v_{0}^{\mu},y^{\mu}\left( t_0 \right) =y_{0}^{\mu},z^{\mu}\left( t_0 \right) =z_{0}^{\mu}, \mu _0=\bar{\mu}>0,   
	\end{cases}
\end{align}
where $t\geq t_0>0$, $\beta>0$, $\alpha \geq 2$, and $\mu \left( t \right) \leq \frac{\mu _0}{t^{2\mathrm{\alpha}}}$.

The accelerated convergence properties of SADMD \eqref{SADMD} will be demonstrated in the following theorem.
\begin{theorem}\label{Theorem_sadmd}
	Suppose that Assumption 3.2 holds, except that the objective function is nonsmooth. Let $\left(x^{\mu}\left( t \right), \lambda^{\mu}\left( t \right), y^{\mu}\left( t \right)\right)$ and $\left(x^*, \lambda^*, y^*\right)$ be a solution trajectory and an optimal solution of SADMD \eqref{SADMD}, respectively. Then, for any initial values $\left( x^{\mu}\left( t_0 \right) ,u^{\mu}\left( t_0 \right) ,\lambda ^{\mu}\left( t_0 \right) , \right. 
	\\
	\left. \upsilon ^{\mu}\left( t_0 \right) ,y^{\mu}\left( t_0 \right) ,z^{\mu}\left( t_0 \right) \right) \in \mathcal{X} \times \mathbb{R} ^{\sum_{i=1}^n{p_i}}\times \mathbb{R} ^{nm}\times \mathbb{R} ^{nm}\times \mathbb{R} ^{nm}\times \mathbb{R} ^{nm}$, the following statements are true.
	
	\textit{(I)}: The trajectories of $x^{\mu}\left( t \right)$, $\lambda^{\mu} \left( t \right)$ and $y^{\mu}\left( t \right)$ are bounded for any $t\geq t_0>0$.
	
	(II): Let $$\hat{\mathbb{L}}_{1}\left( x,\lambda ,y \right) =\hat{f}\left( x^{\mu}\left( t \right) ,\mu \left( t \right) \right) -\frac{1}{2}\lambda ^{\mu}\left( t \right) ^TL\lambda ^{\mu}\left( t \right) +\lambda ^{\mu}\left( t \right) ^T\left( \bar{A}x^{\mu}\left( t \right) -d-Ly^{\mu}\left( t \right) \right),$$ then, it follows 
	\begin{subequations}
		\begin{align}
			&\hat{\mathbb{L}}_1\left( x^{\mu}\left( t \right) ,y^*,\lambda ^* \right) -\hat{\mathbb{L}}_1\left( x^*,y^*,\lambda ^{\mu}\left( t \right) \right) +4\kappa _{\hat{f}}\mu \left( t \right) \le \frac{\alpha ^2\mathbb{E}\left( t_0 \right)}{t^2}, \label{SSSc1}
			\\
			&\lambda^{\mu} \left( t \right) ^TL\lambda^{\mu} \left( t \right) \le \frac{2\alpha ^2\mathbb{E}\left( t_0 \right)}{t^2},\label{SSSc2}
			\\
			&\int_{t_0}^{+\infty}{t\left( \hat{\mathbb{L}}_1\left( x^{\mu}\left( t \right) ,y^*,\lambda ^* \right) -\hat{\mathbb{L}}_1\left( x^*,y^*,\lambda ^{\mu}\left( t \right) \right) +4\kappa _{\hat{f}}\mu \left( t \right) \right)}dt<+\infty.\label{SSSc3}
		\end{align}
	\end{subequations}
\end{theorem}

\begin{proof}
	\textit{(I)}: Consider a smoothing Lyapunov function as follows:
	\begin{align}\label{SSV}
		\mathbb{E} \left( t \right) =\mathbb{E} _1\left( t \right) +\mathbb{E} _2\left( t \right) +\mathbb{E}_3\left( t \right),
	\end{align}
	with
	\begin{align*}
		\begin{cases}
			\mathbb{E} _1\left( t \right) =\frac{t^2}{\alpha ^2}\left( \mathbb{L} _1\left( x^{\mu}\left( t \right) \left( t \right) ,y^*,\lambda ^* \right) -\mathbb{L} _1\left( x^*,y^*,\lambda ^{\mu}\left( t \right) \right) +4\kappa _{\hat{f}}\mu \left( t \right) \right),
			\\
			\,\,\,\,\,\,\,\,\,\,\,\,\,\,\,\,=\frac{t^2}{\alpha ^2}\left( \hat{f}\left( x^{\mu}\left( t \right), \mu \left( t \right) \right) -\hat{f}\left( x^*,\mu \left( t \right) \right) \right. +\frac{1}{2}\lambda ^{\mu}\left( t \right) ^TL\lambda ^{\mu}\left( t \right),
			\\
			\left.\,\,\,\,\,\,\,\,\,\,\,\,\,\,\,\,\,\,\,+4\kappa _{\hat{f}}\mu \left( t \right) +\left( \lambda ^* \right) ^T\left( \bar{A}x^{\mu}\left( t \right) -d-Ly^* \right) \right),
			\\
			\mathbb{E}_2\left( t \right) =D_{\psi ^*}\left( u^{\mu}\left( t \right) ,u^* \right) =\psi ^*\left( u^{\mu}\left( t \right) \right) -\psi ^*\left( u^* \right) -\nabla \psi ^*\left( u^* \right) ^T\left( u^{\mu}\left( t \right) -u^* \right),
			\\
			\mathbb{E} _3\left( t \right) =D_h\left( v^{\mu}\left( t \right) \left( t \right), \lambda ^* \right) =\frac{1}{2}\left\| v^{\mu}\left( t \right) -\lambda ^* \right\| ^2,
			\\
			\mathbb{E} _4\left( t \right) =D_h\left( z^{\mu}\left( t \right) ,y^* \right) =\frac{1}{2}\left\| z^{\mu}\left( t \right) -y^* \right\| ^2,
		\end{cases}
	\end{align*}
	where $D_{\psi ^*}\left( u, u^* \right)$ and $D_h\left( v,\lambda^* \right) $ are the Bregman divergences associated with $\psi ^*$ for $u, u^*$, $h$ for $v,\lambda ^*$ and $z, y^*$. 
	
	The time derivative of $\mathbb{E}\left( t \right)$ with $\nabla \psi ^*\left( u^* \right) =x^*$ is 
	\begin{align*}
		\dot{\mathbb{E}}\left( t \right) =&\dot{\mathbb{E}}_1\left( t \right) +\dot{\mathbb{E}}_2\left( t \right) +\dot{\mathbb{E}}_3\left( t \right) +\dot{\mathbb{E}}_4\left( t \right) \nonumber
		\\
		\leq& \frac{2t}{\alpha ^2}\left( \hat{f}\left( x^{\mu}\left( t \right) ,\mu \left( t \right) \right) -\hat{f}\left( x^*,\mu \left( t \right) \right) +\left( \lambda ^* \right) ^T\bar{A}\left( x^{\mu}\left( t \right) -x^* \right) \right. +4\kappa _{\hat{f}}\mu \left( t \right) \nonumber
		\\
			&\left. +\frac{1}{2}\lambda ^{\mu}\left( t \right) ^TL\lambda ^{\mu}\left( t \right) \right) -\frac{t}{\alpha}\left( L\lambda ^{\mu}\left( t \right) \right) ^T\lambda ^{\mu}\left( t \right) +\frac{4t}{\alpha}\kappa _{\hat{f}}\mu \left( t \right) +\frac{2t^2}{\alpha ^2}\kappa _{\hat{f}}\dot{\mu}\left( t \right) \nonumber
		\\
	\end{align*}
		\begin{align}\label{SSSEnd}
			&-\frac{t}{\alpha}\left( \hat{f}\left( x^{\mu}\left( t \right) ,\mu \left( t \right) \right) -\hat{f}\left( x^*,\mu \left( t \right) \right) +\left( \lambda ^* \right) ^T\bar{A}\left( x^{\mu}\left( t \right) -x^* \right) +4\kappa _{\hat{f}}\mu \left( t \right) \right) 
		\\
		\le&-\frac{\left( \alpha -2 \right) t}{\alpha ^2}\left( \hat{f}\left( x^{\mu}\left( t \right) ,\mu \left( t \right) \right) -\hat{f}\left( x^*,\mu \left( t \right) \right) +\left( \lambda ^* \right) ^T\bar{A}\left( x^{\mu}\left( t \right) -x^* \right) \right.\nonumber 
		\\
		& +4\kappa _{\hat{f}}\mu \left( t \right) \left. +\frac{1}{2}\lambda ^{\mu}\left( t \right) ^TL\lambda ^{\mu}\left( t \right) \right) \nonumber
		\\
		=&-\frac{\left( \alpha -2 \right) t}{\alpha ^2}\left( \hat{\mathbb{L}}_1\left( x^{\mu}\left( t \right) ,y^*,\lambda ^* \right) -\hat{\mathbb{L}}_1\left( x^*,y^*,\lambda ^{\mu}\left( t \right) \right) +4\kappa _{\hat{f}}\mu \left( t \right) \right) \le 0, \nonumber
	\end{align}
	where the first inequality holds because of $\nabla _{\mu}\hat{f}\left( x^{\mu}\left( t \right) ,\mu \left( t \right) \right) +\nabla _{\mu}\hat{f}\left( x^*,\mu \left( t \right) \right) +2\kappa _{\hat{f}}\geq 0$, $\dot{\mu}\left( t \right) \leq 0$, the second inequality is satisfied due to the convexity of $\hat{f}\left( x^{\mu}\left( t \right),\mu \left( t \right) \right)$ of $x^{\mu}\left( t \right) \in \mathcal{X} $ with any fixed $\mu \left( t \right) \in \left[ 0,\bar{\mu} \right] $ and the property of $\mu \left( t \right) \leq \frac{\mu _0}{t^{2\mathrm{\alpha}}}$ (i.e., $\dot{\mu}\left( t \right) \leq-\mu _0\left( 2\mathrm{\alpha} \right) t^{\left( -2\alpha -1 \right)}$ and $\frac{t^2}{\alpha ^2}\kappa _{\hat{f}}\geq 0$ implies $\frac{2t}{\alpha}\kappa _{\hat{f}}\mu \left( t \right) +\frac{t^2}{\alpha ^2}\kappa _{\hat{f}}\dot{\mu}\left( t \right) \leq \frac{2t}{\alpha}\kappa _{\hat{f}}\mu _0t^{-2\alpha}-\frac{t^2}{\alpha ^2}\kappa _{\hat{f}}\mu _0\left( 2\mathrm{\alpha} \right) t^{-2\alpha -1}=\frac{2t}{\alpha}\kappa _{\hat{f}}\mu _0t^{-2\alpha}-\frac{2t}{\alpha ^2}\kappa _{\hat{f}}\mu _0t^{-2\alpha}=0$), and the last inequality is established since $\mathrm{\alpha}\geq 2$ and $\hat{\mathbb{L}}_{1}\left( x^{\mu}\left( t \right) ,\lambda ^* \right) -\hat{\mathbb{L}}_{1}\left( x^*,\lambda ^{\mu}\left( t \right) \right) +4\kappa _{\hat{f}}\mu \left( t \right) \geq0$.
	
	Recall that the Lyapunov function $\mathbb{E}\left( t \right)$ is radially unbounded and positive for any $t\geq t_0>0$. This implies that the trajectories of $x^{\mu}\left( t \right)$, $\lambda^{\mu} \left( t \right)$ and $y^{\mu}\left( t \right)$ are bounded for any $t\geq t_0>0$. The proof of (I) is completed.
	
	\textit{(II)}: From \eqref{SSSEnd}, we have that $\dot{\mathbb{E}}\left( t \right)\leq0$, i.e., ${\mathbb{E}}\left( t \right)$ is nonincreasing on $\left[ t_0,+\infty \right)$. Thus, for any $t\geq t_0>0$, one has $\hat{\mathbb{L}} _1\left( x^{\mu}\left( t \right) \left( t \right) ,y^*,\lambda ^* \right) -\hat{\mathbb{L}}  _1\left( x^*,y^*,\lambda ^{\mu}\left( t \right) \right) +4\kappa _{\hat{f}}\mu \left( t \right) =\mathbb{E}_1\left( t \right) \le \mathbb{E}\left( t \right) \le \mathbb{E}\left( t_0 \right)$, i.e., $\hat{\mathbb{L}}  _1\left( x\left( t \right) ,y^*,\lambda ^* \right) -\hat{\mathbb{L}} _1\left( x^*,y^*,\lambda \left( t \right) \right)+4\kappa _{\hat{f}}\mu \left( t \right) \le \frac{\alpha ^2\mathbb{E}\left( t_0 \right)}{t^2}$ \eqref{SSSc1} holds. In addition, the \eqref{SSSc1} implies $\lambda^{\mu} \left( t \right) ^TL\lambda^{\mu} \left( t \right) \le \frac{2\alpha ^2E\left( t_0 \right)}{t^2}$, i.e.,  the \eqref{SSSc2} holds.
	
	Since $\dot{\mathbb{E}}\left( t \right) \le -\frac{\left( \alpha -2 \right) t}{\alpha ^2}\left( \hat{\mathbb{L}}_1\left( x^{\mu}\left( t \right) ,y^*,\lambda ^* \right) -\hat{\mathbb{L}}_1\left( x^*,y^*,\lambda ^{\mu}\left( t \right) \right) +4\kappa _{\hat{f}}\mu \left( t \right) \right) $, thus, integrating it from $t_0$ to $+\infty $ and rearranging it, we have
	\begin{align*}
		&\int_{t_0}^{+\infty}{t\left( \hat{\mathbb{L}}_1\left( x^{\mu}\left( t \right) ,y^*,\lambda ^* \right) -\hat{\mathbb{L}}_1\left( x^*,y^*,\lambda ^{\mu}\left( t \right) \right) +4\kappa _{\hat{f}}\mu \left( t \right) \right)}dt
		\\
		=&\frac{\alpha ^2}{\alpha -2}\left( \mathbb{E} \left( t_0 \right) -\mathbb{E} \left( +\infty \right) \right) \le \frac{\alpha ^2}{\alpha -2}\mathbb{E} \left( t_0 \right) <+\infty,
	\end{align*}
	therefore, the proof of \eqref{SSSc3} is completed.
\end{proof}

\section{Numerical experiment}\label{sec:experiments}
In this section, we give several numerical experiments to illustrate the effectiveness and superiority of the proposed accelerated dynamical approaches.  We use the ode45 function in MATLAB to solve the dynamical approaches in all our numerical experiments.

\subsection{In the smooth case} %\newsiamthm{example}{Example}
\begin{example} \textbf{Logistic regression} \cite{Attouch2021FastCO}: Consider the problem \eqref{P1} as follows:
	\begin{equation} \label{logregress}
		\begin{split}
			&\min  f\left( x \right) =\log \left( 1+\exp \left( -\left( 1,1,1,1 \right) ^Tx \right) \right) 
			\\
			&\mathrm{s}.\mathrm{t}. \,\,Ax=b, x\in \mathcal{X},
		\end{split}
	\end{equation}
	where $A=\left[ \begin{array}{c}
		0.2,1,1,2\\
		0,1,0.5,1\\
	\end{array} \right]$,\,\, $b=\mathrm{col}\left( 1,1\right)$, $\mathcal{X}=\left\{ x\in \mathbb{R}_+ ^4|\sum_{i=1}^4{x_i=1}\right\} $. The objective function is convex (but not strongly convex) and smooth, and it is a very popular regularization in machine learning. Applying APDMD \eqref{PDM_sm} with \textbf{Kullback-Leibler} divergence to address problem \eqref{logregress}. \eqref{fig:APDMD_exp} (left) shows the trajectories of estimates for $x\left(t\right)$ of APDMD \eqref{PDM_sm} versus time; \eqref{fig:APDMD_exp} (middle) and (right) display the error of $\left| f\left( x\left( t \right) \right) -f\left( x^* \right) \right|$ and $\left\| Ax\left( t \right) -b \right\|$ respectively. The numerical results of gaps for objective function and equation constraints are in excellent agreement with our theoretical results, i.e., they both converge at the predicted rates. 
\end{example}
\begin{figure}[!thbp]
	\includegraphics[width=5cm,height=4cm]{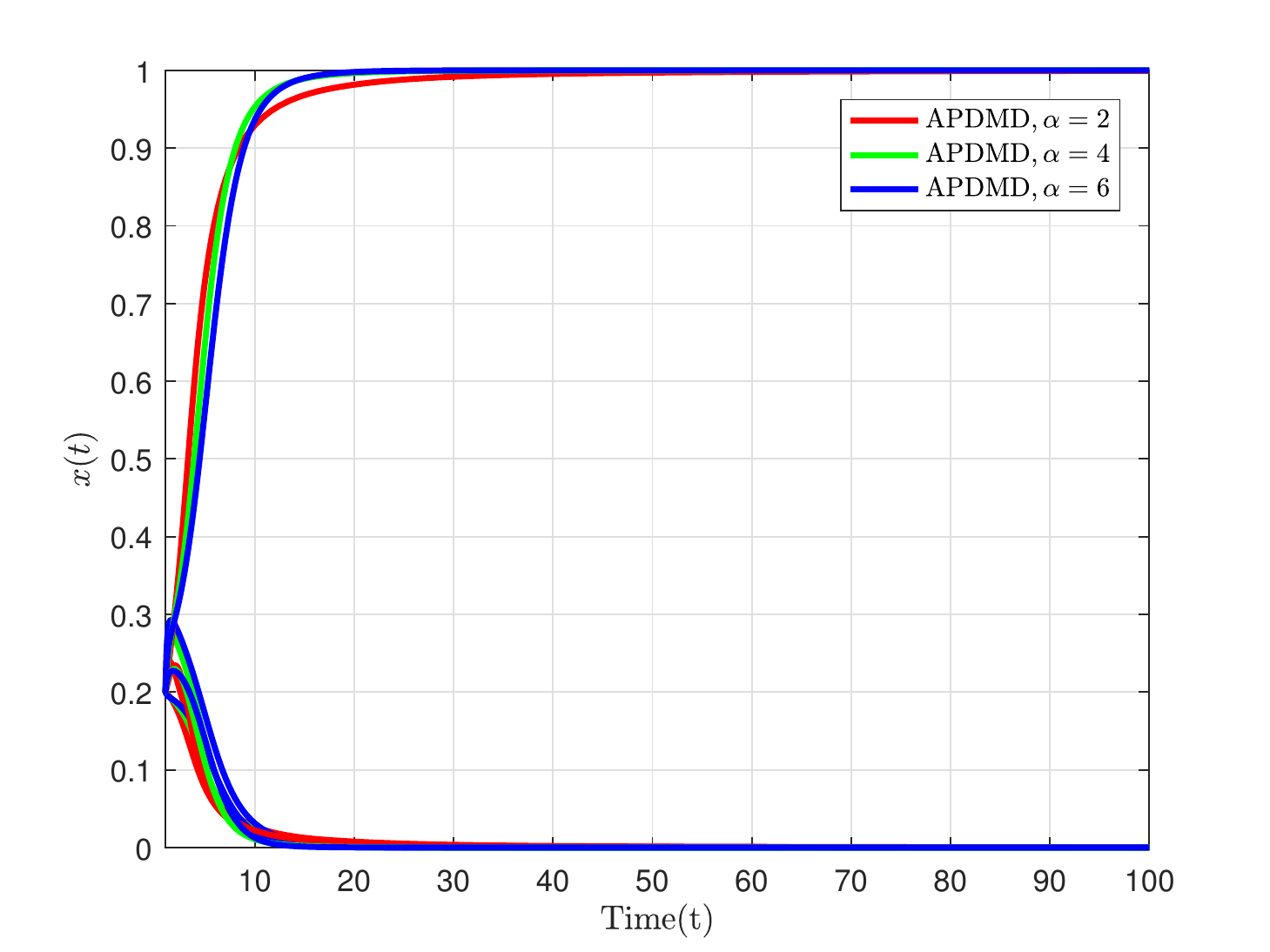}
	\includegraphics[width=5cm,height=4cm]{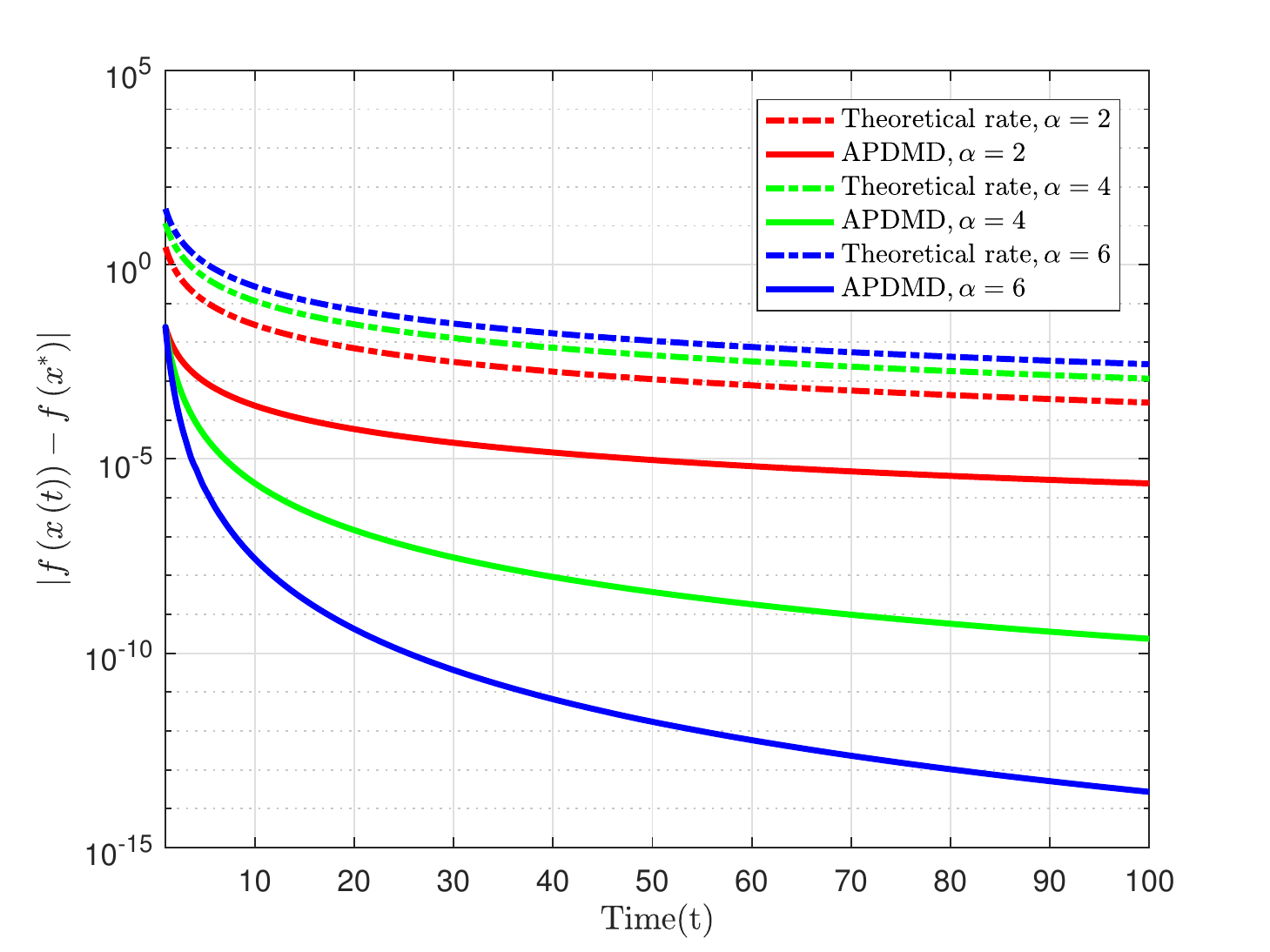}
	\includegraphics[width=5cm,height=4cm]{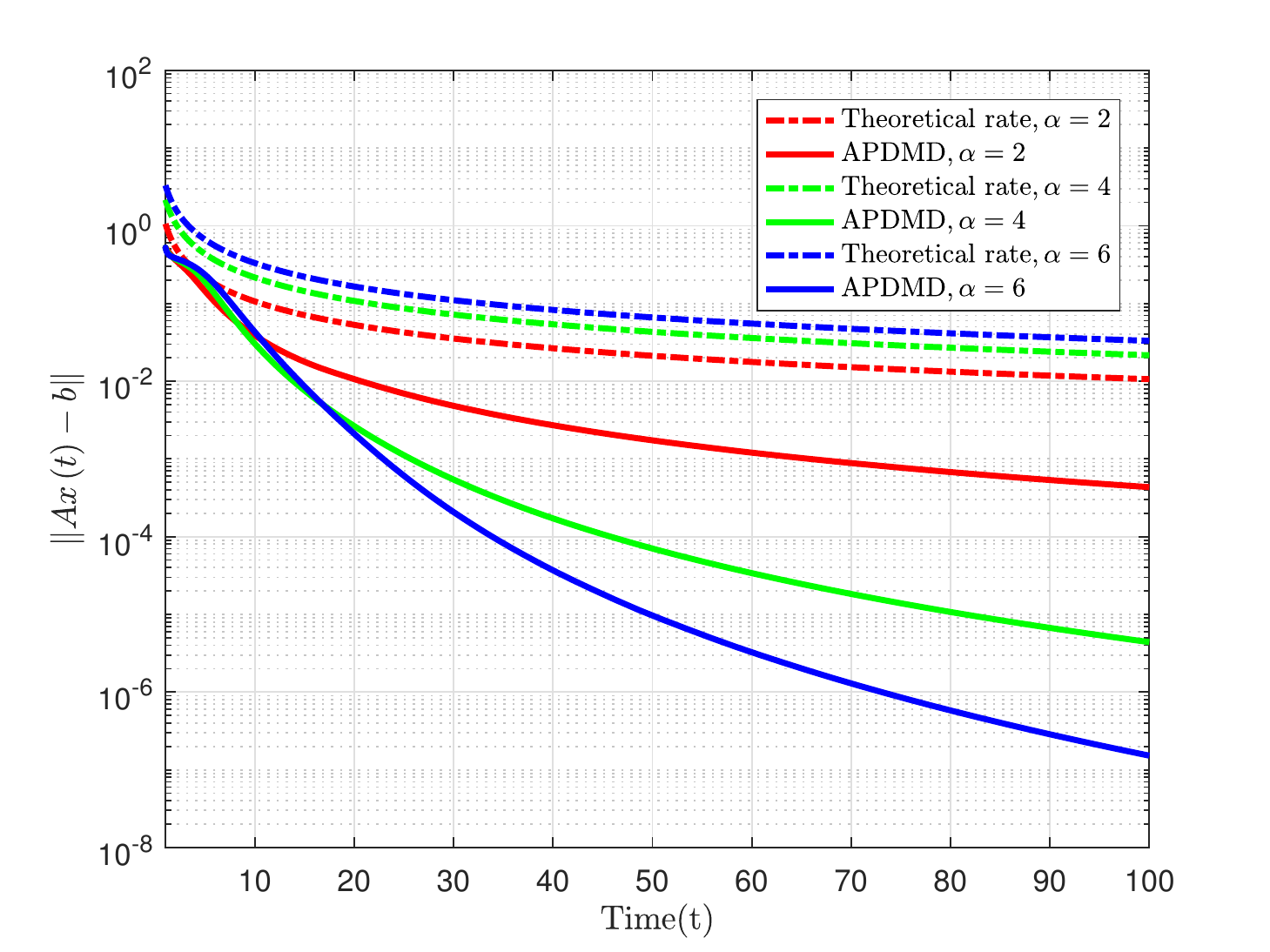}
	\caption{APDMD \eqref{PDM} with $\alpha =2,4,6$. (left) Trajectories of $x\left( t \right)$  for problem \eqref{logregress}. (middle) Error of $\left| f\left( x\left( t \right) \right)-f\left( x^* \right) \right|$ under APDMD \eqref{PDM_sm} for problem \eqref{logregress}. (right) Error of $\left\| Ax\left( t \right) -b \right\|$ under APDMD \eqref{PDM_sm} for problem \eqref{logregress}.}	\label{fig:APDMD_exp}
\end{figure}

\begin{example} \textbf{Distributed Logistic regression}: Consider a distributed problem \eqref{P2} as follows:
	\begin{equation} \label{dis_log}
		\begin{split}
			&\min f\left( x \right) =\sum_{i=1}^n{\log \left( 1+\exp \left( -\left( i-1,\frac{i}{2},i,i+1 \right) ^Tx_i \right) \right)}
			\\
			&\mathrm{s}.\mathrm{t}. \,\,Lx=0, \,\, x\in \prod_{i=1}^n{\mathcal{X} _i},
		\end{split}
	\end{equation}
	where $x=\left( x_1,...,x_n \right) ^T\in \mathbb{R} ^{4n}$. Setting $n=4$ and 
	$\mathcal{X} _1=\left\{ x_1\in \mathbb{R} _{+}^{4}|1^Tx=1 \right\}$, $\mathcal{X}_2=\left\{ x_2\in \mathbb{R} _{+}^{4} \right\}$, $
	\mathcal{X}_3=\left\{x_3\in \mathbb{R}^4|\left\|x_3-\left( 0.1,0.2,0.5,0.8 \right) ^T \right\| \leq 2 \right\}$ and  $\mathcal{X} _4=\left\{ x_4\in \mathbb{R} ^4|\left( 1,\right. \right. 
	\\
	\left. \left.1,1,1 \right) x_4\le 4 \right\} $. Note that the objective function in \eqref{dis_log} is convex (but not strongly convex) and smooth, which satisfies the requirement of problem \eqref{P2} in the smooth case. Applying ADPDMD \eqref{ADPDMD} with 4 agents (connected as a ring) to solve problem \eqref{dis_log}, i.e., for agent 1, using \textbf{Kullback-Leibler} divergence; for agent 2, applying \textbf{Itakura-Saito} divergence; for agent 3, adopting projection operator of \textbf{Sphere} set; for agent 4, utilizing projection operator of \textbf{half-space} set. 
	\eqref{fig:ADPDMD_exp2} (left) shows the trajectories of $x\left(t\right)$ under ADPDMD \eqref{ADPDMD} for problem \eqref{dis_log} are uniformly convergent and globally asymptotically stable, i.e., $x_1=x_2=...=x_5$; \eqref{fig:ADPDMD_exp2} (middle) and (right) display the errors of $\left| f\left( x\left( t \right) \right) -f\left( x^* \right) \right|$ and $\sqrt{x\left( t \right) ^TLx\left( t \right)}$, respectively. The numerical results show $\left| f\left( x\left( t \right) \right) -f\left( x^* \right) \right|$ and $\sqrt{x\left( t \right) ^TLx\left( t \right)}$ and $\sqrt{x\left( t \right) ^TLx\left( t \right)}$ all converge at the predicted rates.
\end{example}
\begin{figure}[!thbp]
	\includegraphics[width=5cm,height=4cm]{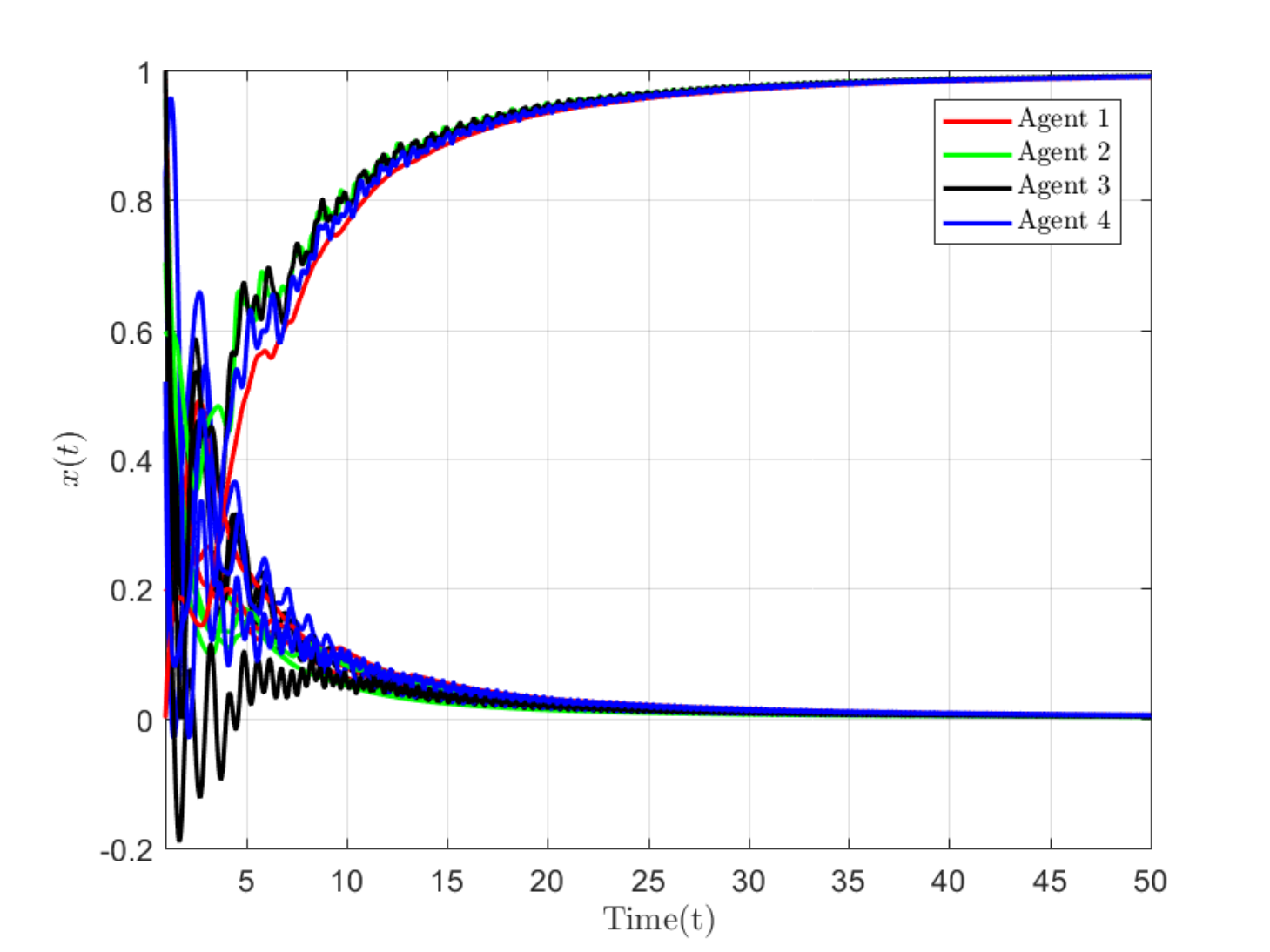}
	\includegraphics[width=5cm,height=4cm]{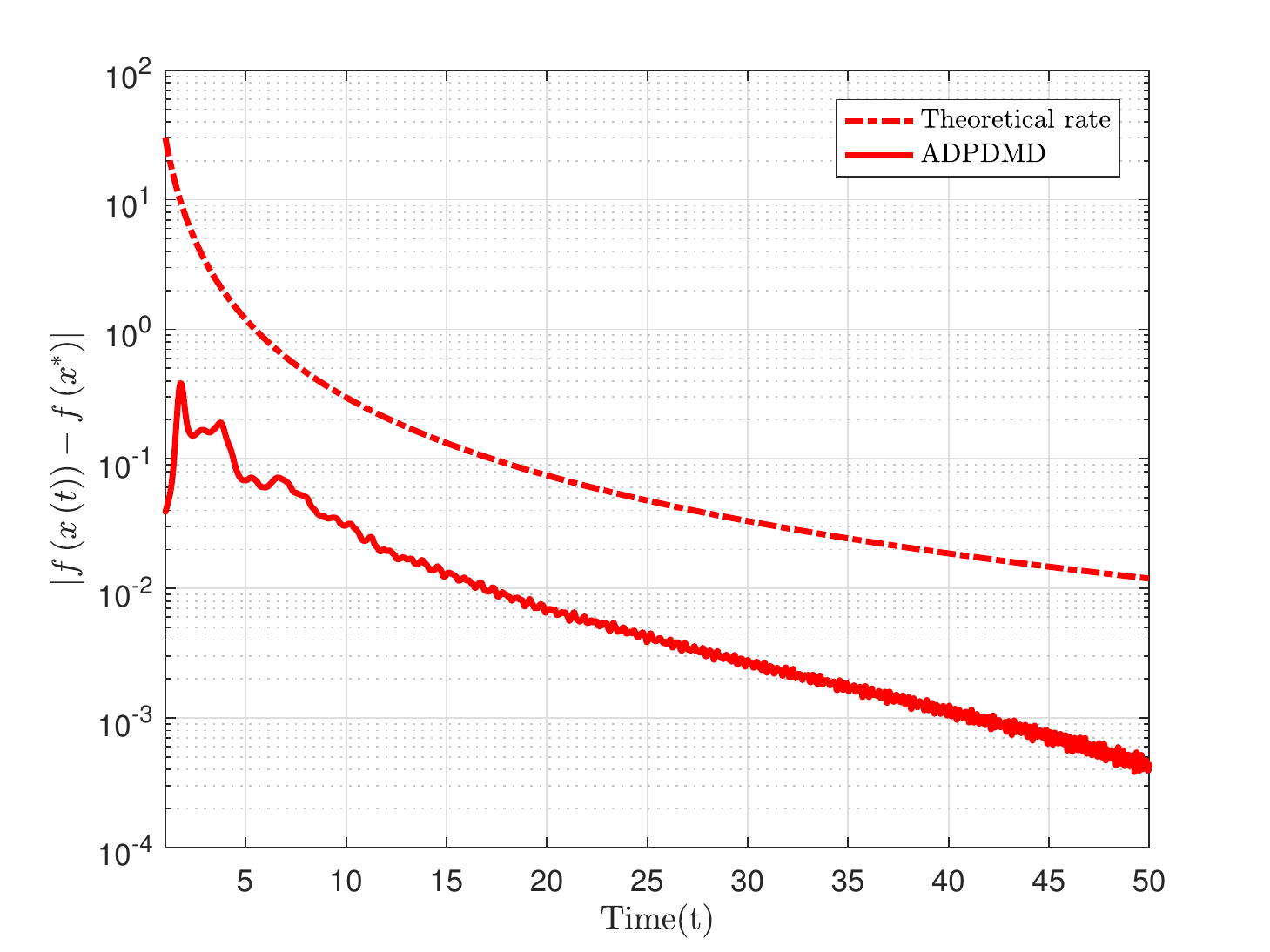}
	\includegraphics[width=5cm,height=4cm]{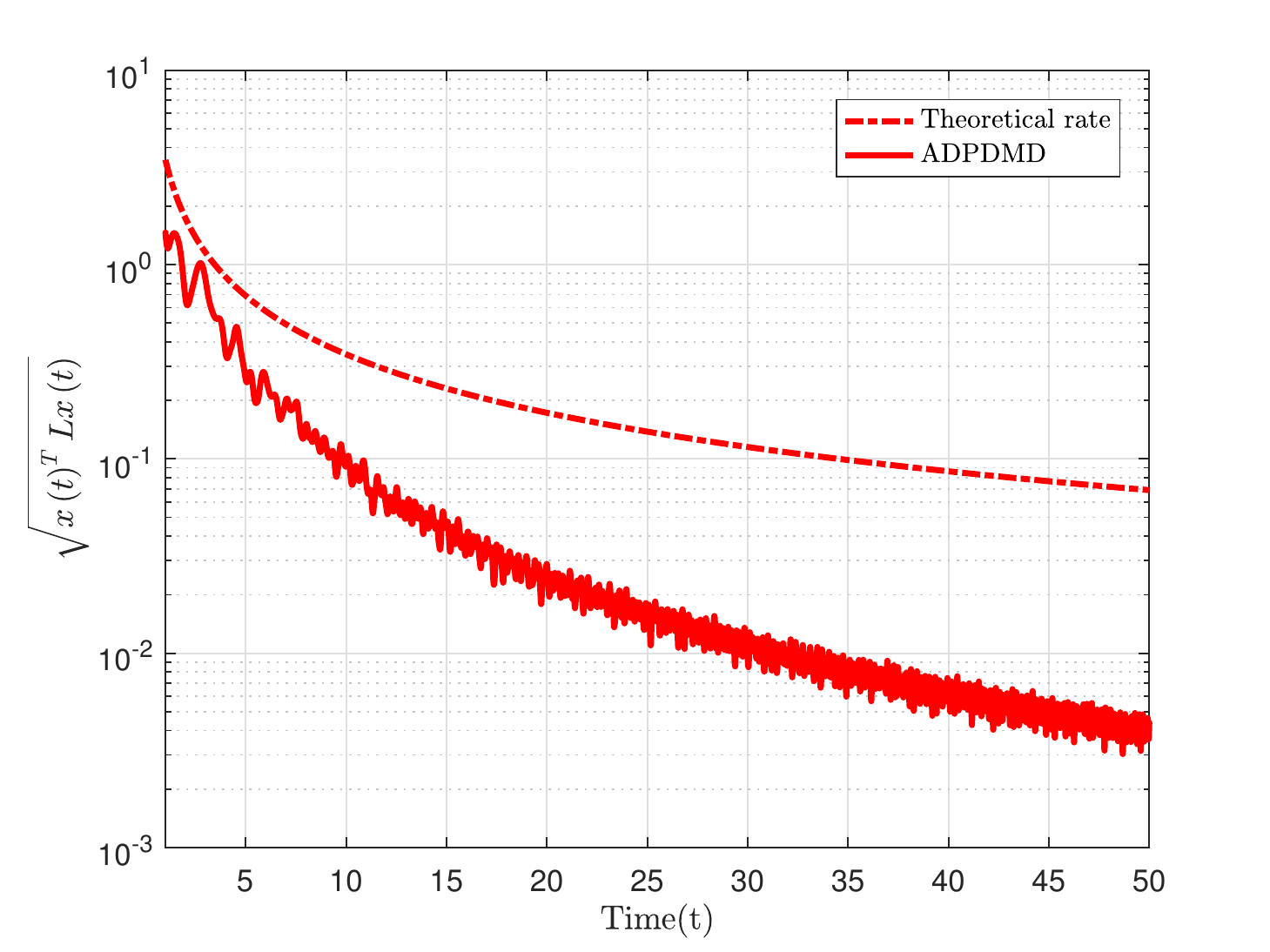}
	\caption{ADPDMD \eqref{ADPDMD} with $\alpha =3$. (left) Trajectories of $x\left( t \right)$  for problem \eqref{dis_log}. (middle) Error of $\left| f\left( x\left( t \right) \right)-f\left( x^* \right) \right|$ under ADPDMD \eqref{ADPDMD} for problem \eqref{dis_log}. (right) Error of$\sqrt{x\left( t \right) ^TLx\left( t \right)}$ under  ADPDMD \eqref{ADPDMD} for problem \eqref{dis_log}.}
	\label{fig:ADPDMD_exp2}
\end{figure}

\begin{example}\textbf{	Distributed quadratic programming}: Consider a special case of DEMO \eqref{P3} in the smooth case as follows:
	\begin{equation}\label{D_SP}
		\begin{split}
			&\min f\left( x \right) =\sum_{i=1}^n{x_{i}^{T}A_ix_i},
			\\
			&\mathrm{s}.\mathrm{t}. \,\,x+Ly=0, x\in \prod_{i=1}^n{\mathcal{X} _i},
		\end{split}
	\end{equation}
	where $\mathcal{X} _i=\left\{ x_i\in \mathbb{R} ^m| i+2\leq x_{i,j}\ll i+3, j=1,...,m \right\} 
	, i=1,...,n$, $b=\mathrm{col}\left( 7,...,7 \right)$$\in \mathbb{R} ^{nm}$ and $y\in \mathbb{R} ^{mn}$ is an auxiliary variable.  Let $n=10$, $m=5$ and $A_i\in \mathbb{R} ^{10\times 10}$ be a positive semi-definite matrix generated by standard Gaussian distribution. Let every part $x_i$ be an agent, and 10 agents are connected as a ring. Applying the ADMD \eqref{ADMD} to solve problem \eqref{D_SP} and the experimental results are shown in \eqref{fig:ADMD_exp3}. In \eqref{fig:ADMD_exp3} (left), the dual variable $\lambda$ is uniformly convergent and globally asymptotically stable, i.e.,  $\lambda_1=\lambda_2=...=\lambda_{10}$; \eqref{fig:ADMD_exp3} (middle and right) illustrate $\left| f\left( x^{\mu}\left( t \right) \right)-f\left( x^* \right) \right|$ and $\sqrt{\lambda \left( t \right) ^TL\lambda \left( t \right)}$ of ADMD \eqref{ADMD} are convergent at the predicted rates.
\end{example}
\begin{figure}[!thbp]
	\includegraphics[width=5cm,height=4cm]{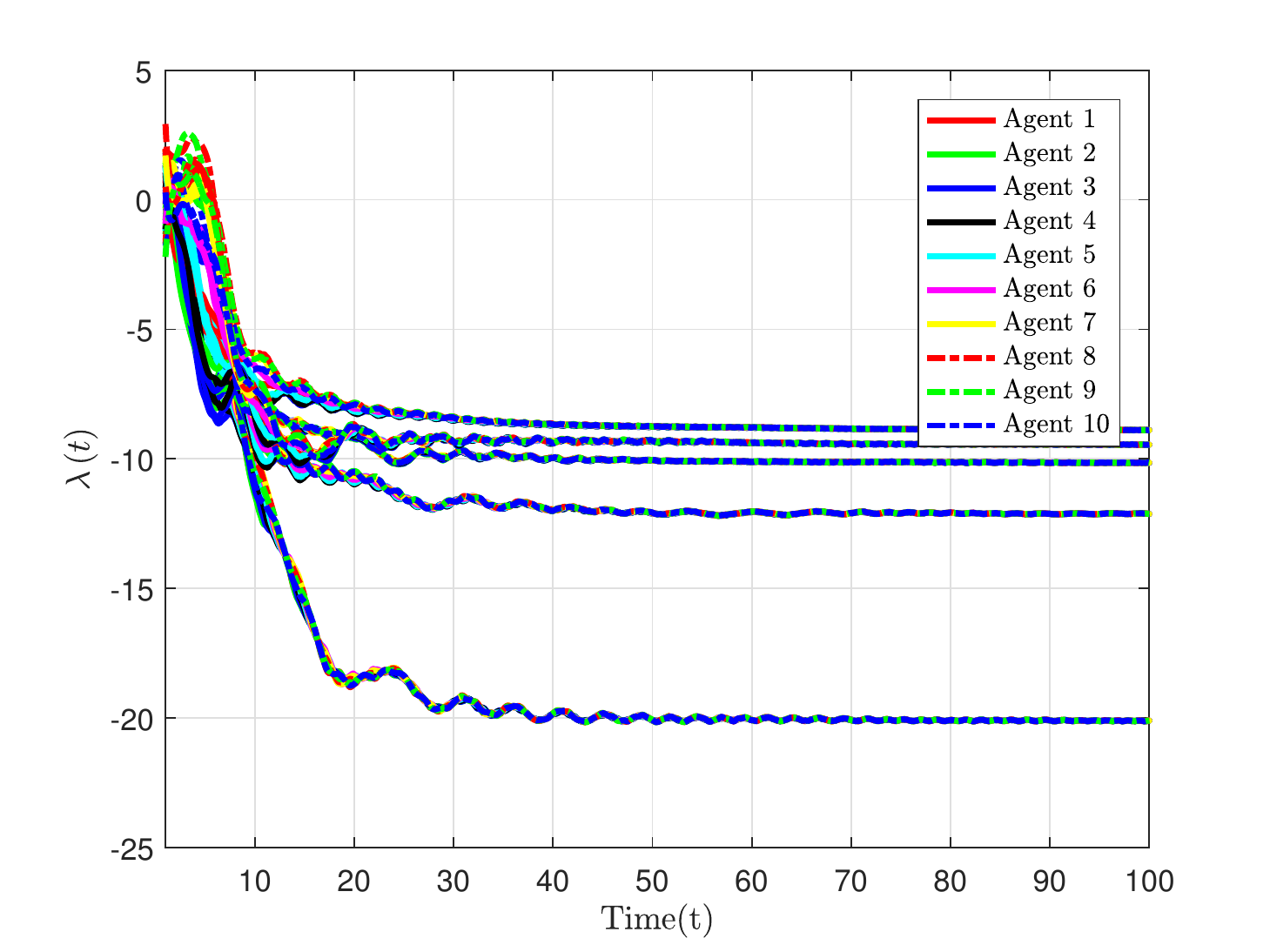}
	\includegraphics[width=5cm,height=4cm]{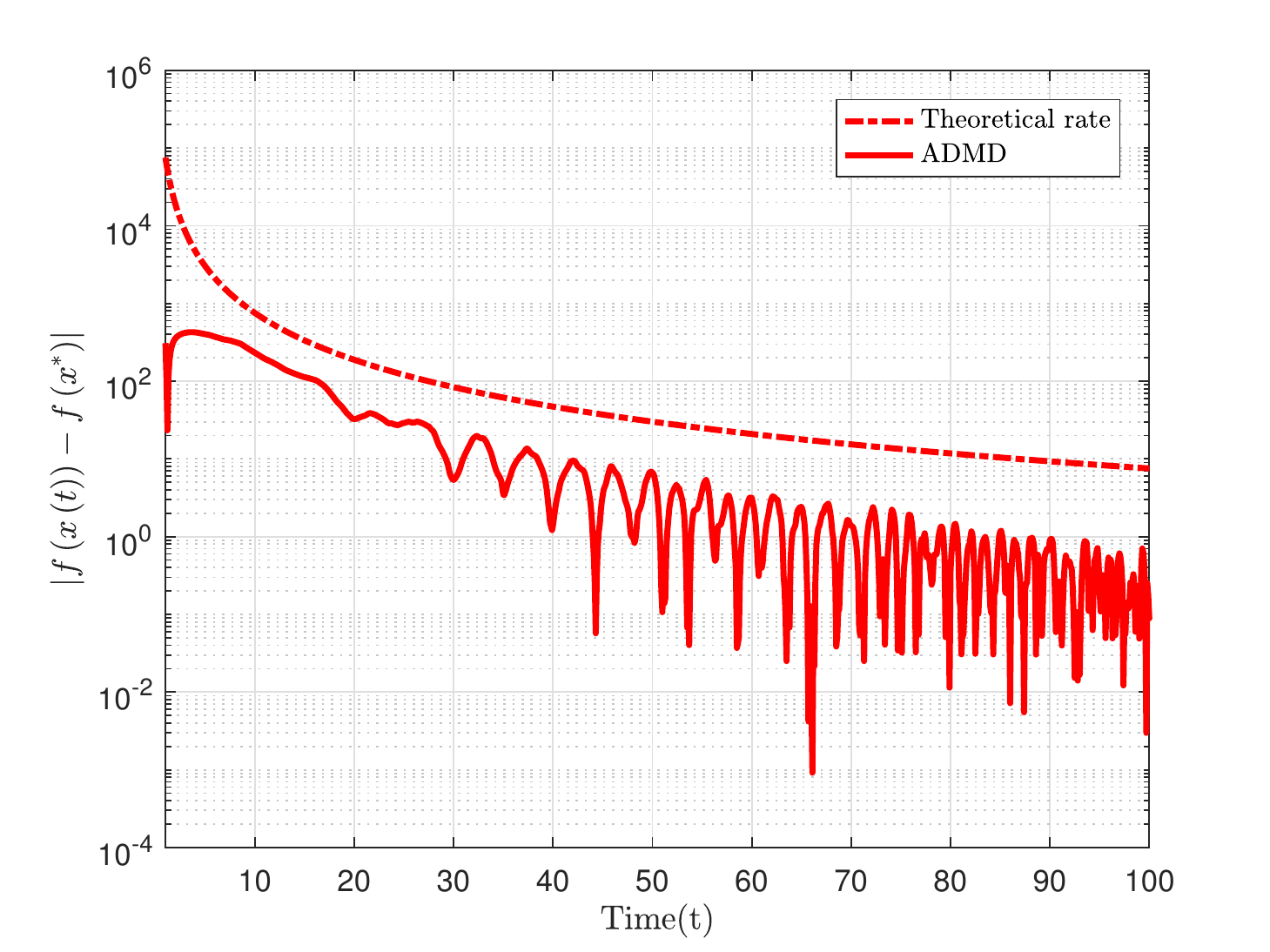}
	\includegraphics[width=5cm,height=4cm]{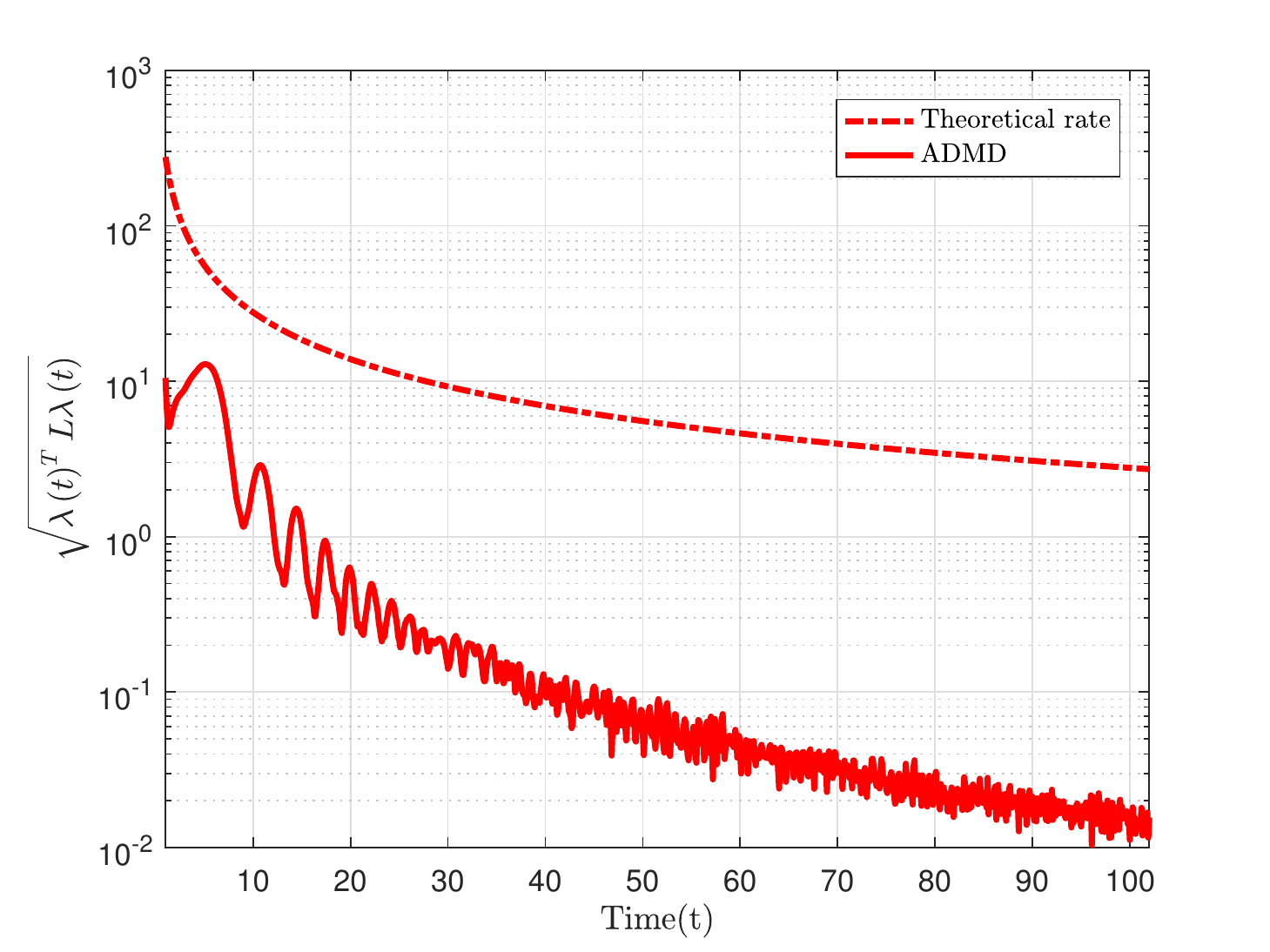}
	\caption {ADMD \eqref{ADMD} with with $\alpha =3$. (left) Trajectories of $\lambda$ of 10 agents. (middle) Error of $\left| f\left( x^{\mu}\left( t \right) \right) -f\left( x^* \right) \right|$ under ADMD for problem \eqref{D_SP}. (right) Error of $\sqrt{\lambda \left( t \right) ^TL\lambda \left( t \right)}$ under ADMD for problem \eqref{D_SP}.}
	\label{fig:ADMD_exp3}
\end{figure}		

\subsection{In the nonsmooth case} \label{non-experiment} 
\begin{example}\textbf{Nonnegative Basis Pursuit (NBP)} in \cite{khajehnejad2010sparse}:  Take into account a NBP as follows:
	\begin{equation} \label{NBP}
		\begin{split}
			&\min  f\left( x \right) =\left\| x \right\| _1,
			\\
			&\mathrm{s}.\mathrm{t}. \,\,Ax=b, x\in \mathcal{X},
		\end{split}
	\end{equation}
	where $\mathcal{X} =\left\{ x\in \mathbb{R} ^{40}|,x_i\geq 0,i=1,...,40 \right\}$ and $A\in \mathbb{R} ^{10\times 40}$ is an orthogonal Gaussian matrix. The objective function in \eqref{NBP} is convex (but not strongly convex) and nonsmooth as required. Applying the SAPDMD \eqref{PDM_sm} and let $\psi \left( x \right) =\sum_{i=1}^n{x_i\ln x_i}$ to solve problem \eqref{NBP}. \eqref{fig:APDMD_exp2} displays the reconstructed sparse signal of $x$ in the left, the error of $\left| f\left( x^{\mu}\left( t \right) \right) -f\left( x^* \right) \right|$ in the middle, and the error of $\left\| Ax^{\mu}\left( t \right) -b \right\|$ in the right. All parameters are complied with the requirement. The numerical results are in good accordance with our theoretical results, where the $\left| f\left( x^{\mu}\left( t \right) \right) -f\left( x^* \right) \right|$ and $\left\| Ax^{\mu}\left( t \right) -b \right\|$ are convergent at the predicted rates.
\end{example}	
\begin{figure}[!thbp]
	\includegraphics[width=5cm,height=4cm]{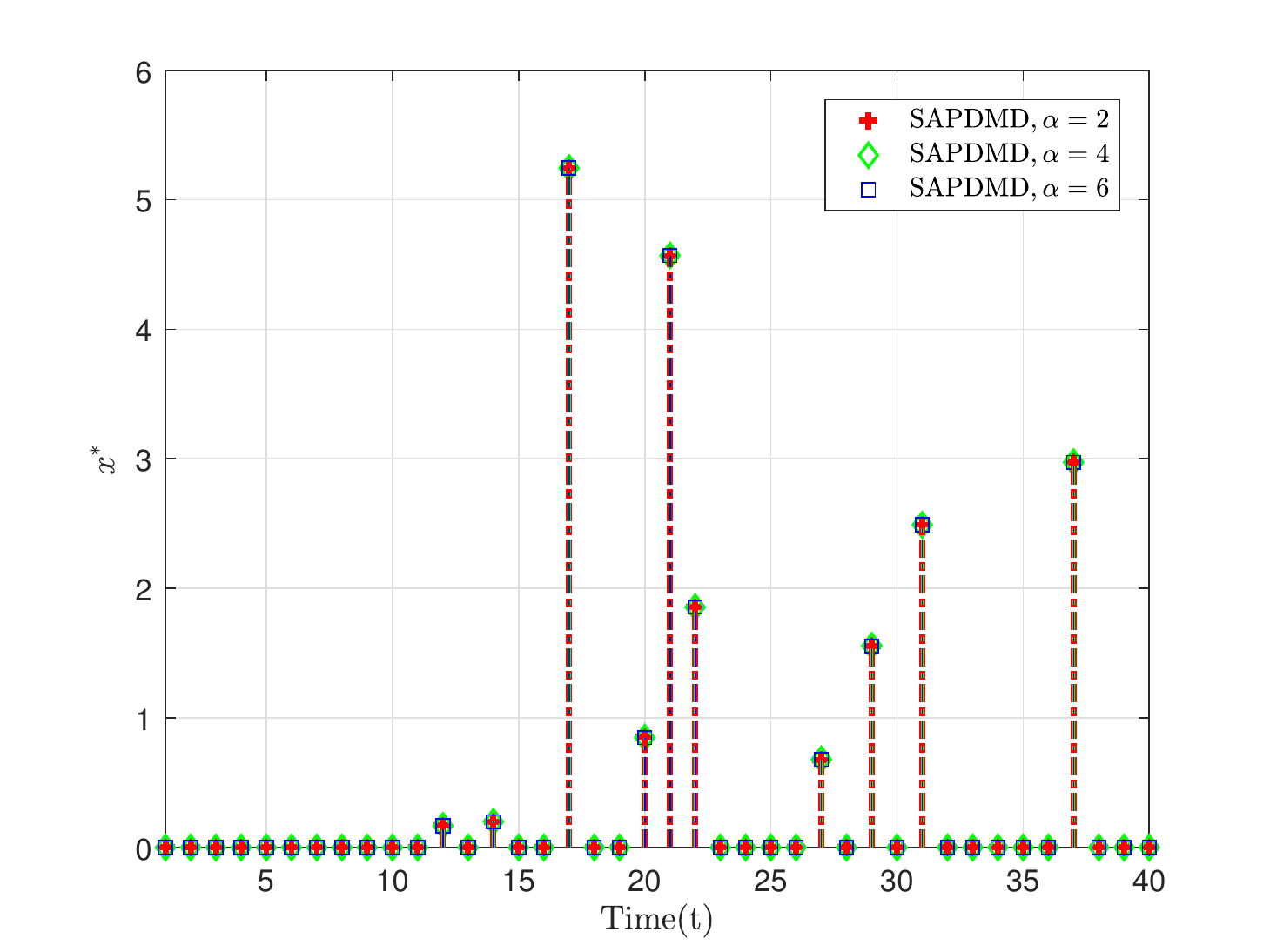}
	\includegraphics[width=5cm,height=4cm]{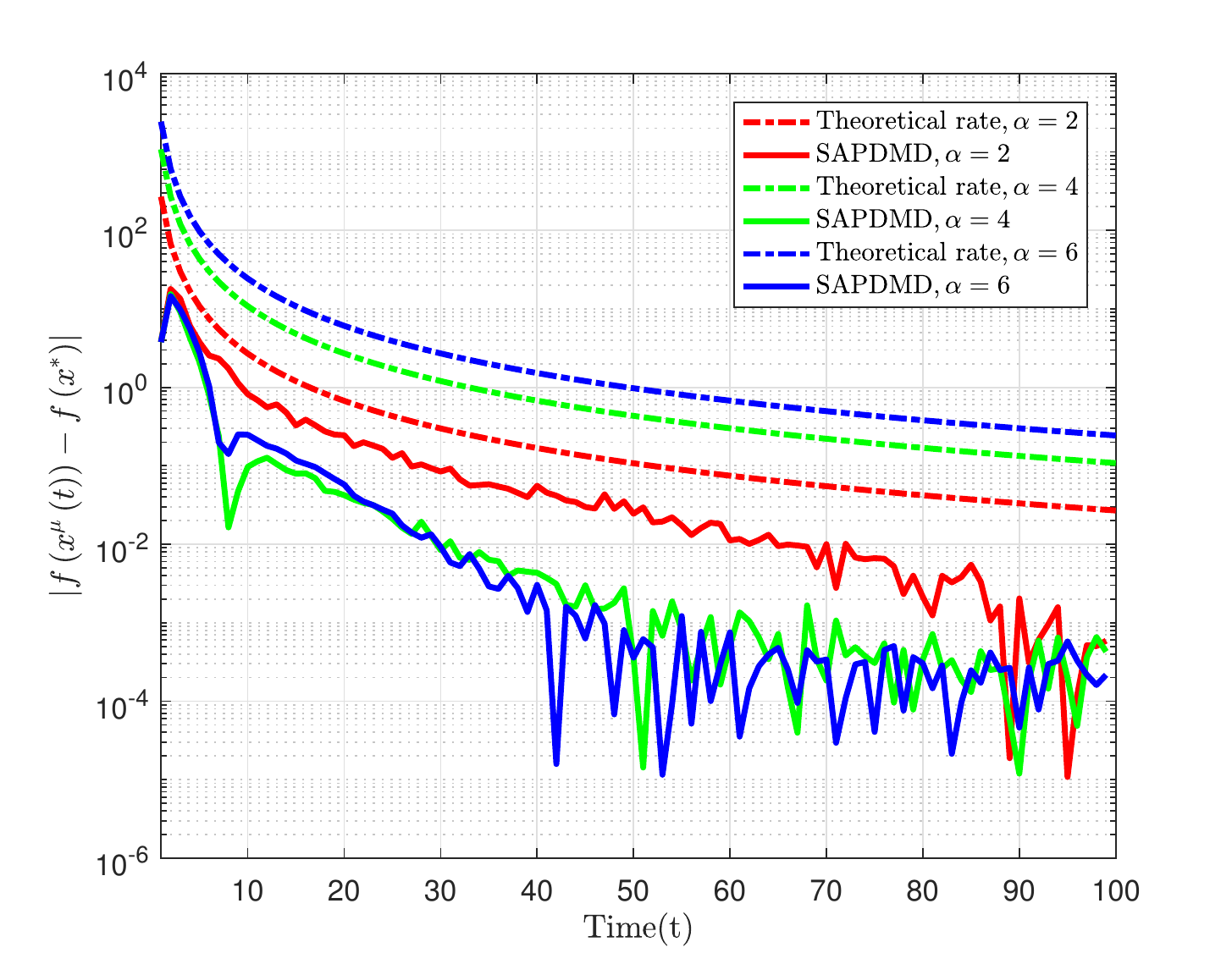}
	\includegraphics[width=5cm,height=4cm]{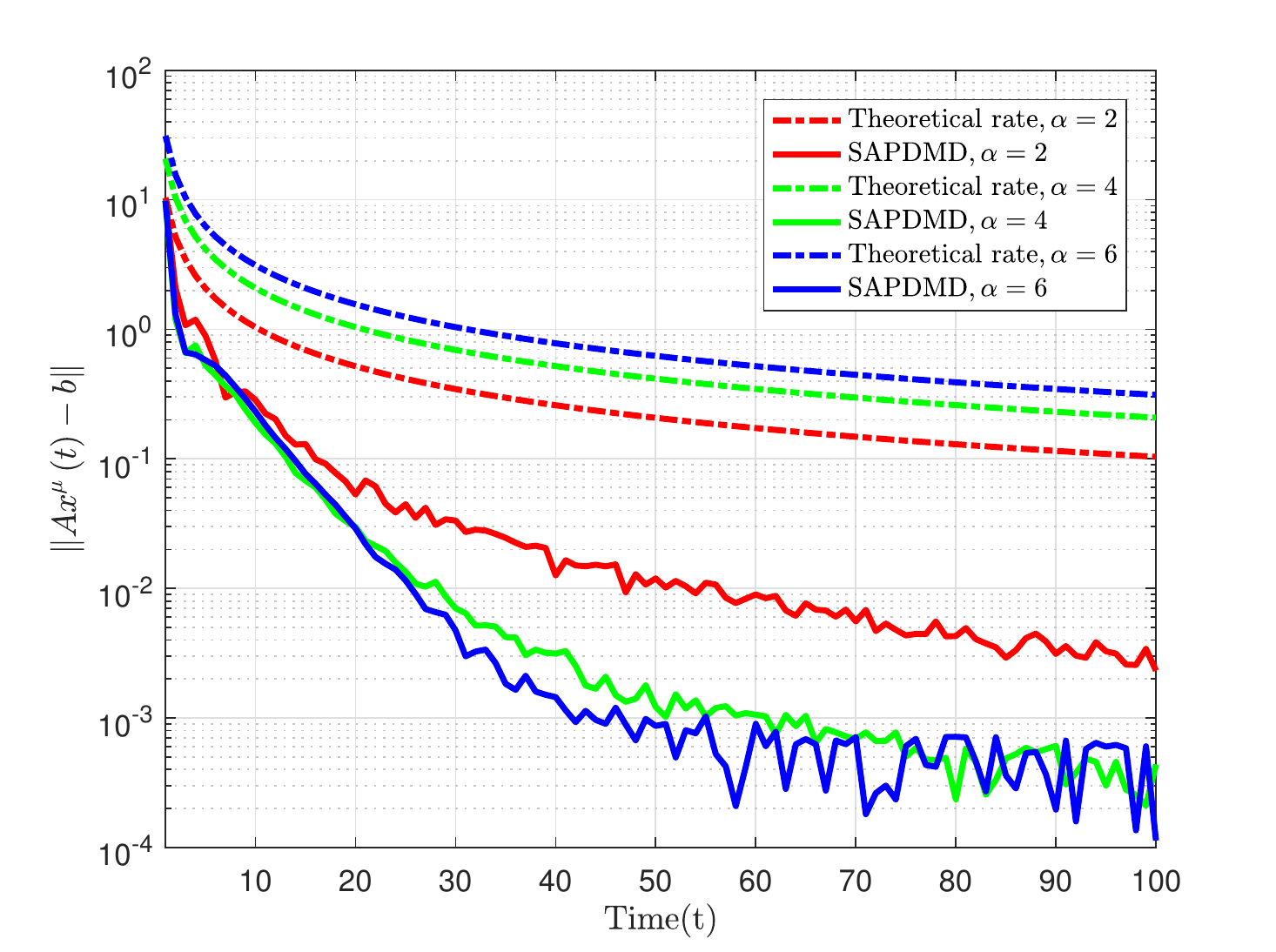}
	\caption{SAPDMD \eqref{PDM_sm} with $\alpha =2,4,6$. (left) Reconstructed sparse signal for problem \eqref{NBP}. (middle) Error of $\left| f\left( x^{\mu}\left( t \right) \right)-f\left( x^* \right) \right|$ for problem \eqref{NBP}. (right) Error of $\left\| Ax^{\mu}\left( t \right) -b \right\|$ for problem \eqref{NBP}.}
	\label{fig:APDMD_exp2}
\end{figure}

\begin{example}\textbf{Distributed basis pursuit in row partition} in \cite{Zhao2021CentralizedAC}: Take a distributed basis pursuit with row partition of sensing matrix as follows:
	\begin{equation}\label{D_BP_R}
		\begin{split}
			&\min  f\left( x \right) =\sum_{i=1}^k{\left\| x_i \right\| _1}
			\\
			&\mathrm{s}.\mathrm{t}.\,Lx=0, x\in \mathcal{X},
		\end{split}
	\end{equation}
	where  $L=L_k\otimes I_n\in \mathbb{R} ^{kn\times kn}$, $\mathcal{X} =\left\{ x\in \mathbb{R} ^{kn}|\hat{A}x=b,\hat{A}=\mathrm{bl}\mathrm{diag}\left\{A_{m_1\times n},...,A_{m_k\times n} \right\} \right.$
	\\
	$\left. \in \mathbb{R} ^{m\times kn},b\in \mathbb{R} ^m \right\} $, $b\in \mathbb{R} ^m$. Note that, the problem \eqref{D_BP_R} is convex (but not strongly convex) and nonsmooth and as a special case of DCCP \eqref{P2} in the nonsmooth case. Setting $m=10$, $k=5$, $n=60$, sparsity be $2$ and every $A_{m_i\times n}, i=1,...,5$  acts as an agent, and 5 agents are connected as a ring. Applying the SADPDMD \eqref{SADPDMD} with $\nabla \psi ^*=P_{\mathcal{X}}$ of affine set $\mathcal{X} =\left\{ \begin{array}{c}x\in \mathbb{R} ^{kn}|\hat{A}x=b\\
	\end{array} \right\} $ in \eqref{proj} to solve the problem \eqref{D_BP_R} and all parameters  comply with the requirement. The numerical results are displayed in \eqref{fig:SDAPDMD_exp}. As can be seen from \eqref{fig:SDAPDMD_exp} (left) that the variable $x$ is uniformly convergent and globally asymptotically stable, i.e., $x^{\mu}_1=x^{\mu}_2=...=x^{\mu}_5$; \eqref{fig:APDMD_exp2} (middle) shows that the SADPDMD can efficiently solve the problem \eqref{D_BP_R} and reconstruct the original sparse signal; the result in  \eqref{fig:SDAPDMD_exp} (middle) is in good accordance with our theoretical results, i.e., $\left| f\left( x^{\mu}\left( t \right) \right)-f\left( x^* \right) \right|$ is convergent at the predicted rates.
\end{example}
\begin{figure}[!thbp]
	\includegraphics[width=5cm,height=4cm]{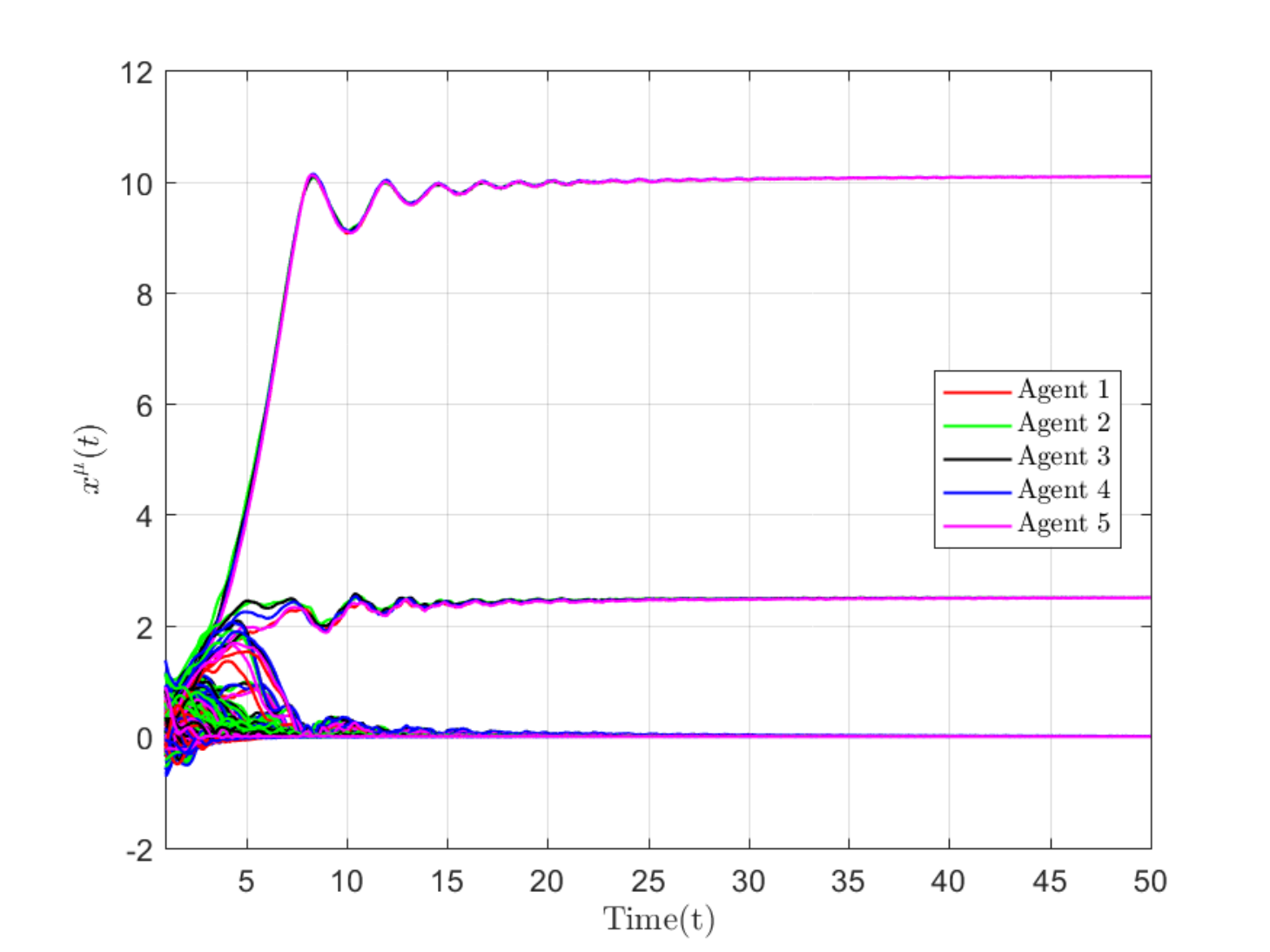}
	\includegraphics[width=5cm,height=4cm]{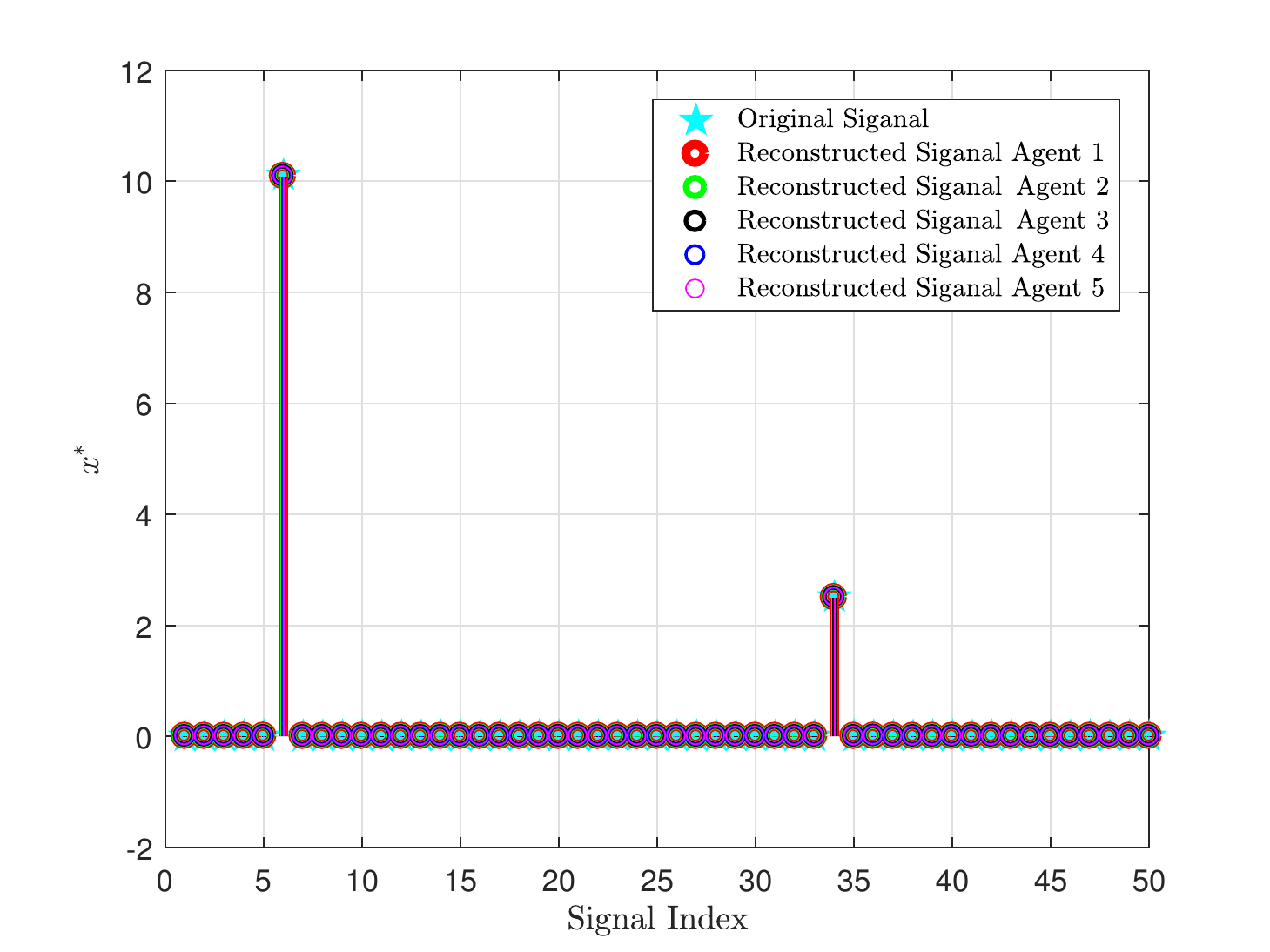}
	\includegraphics[width=5cm,height=4cm]{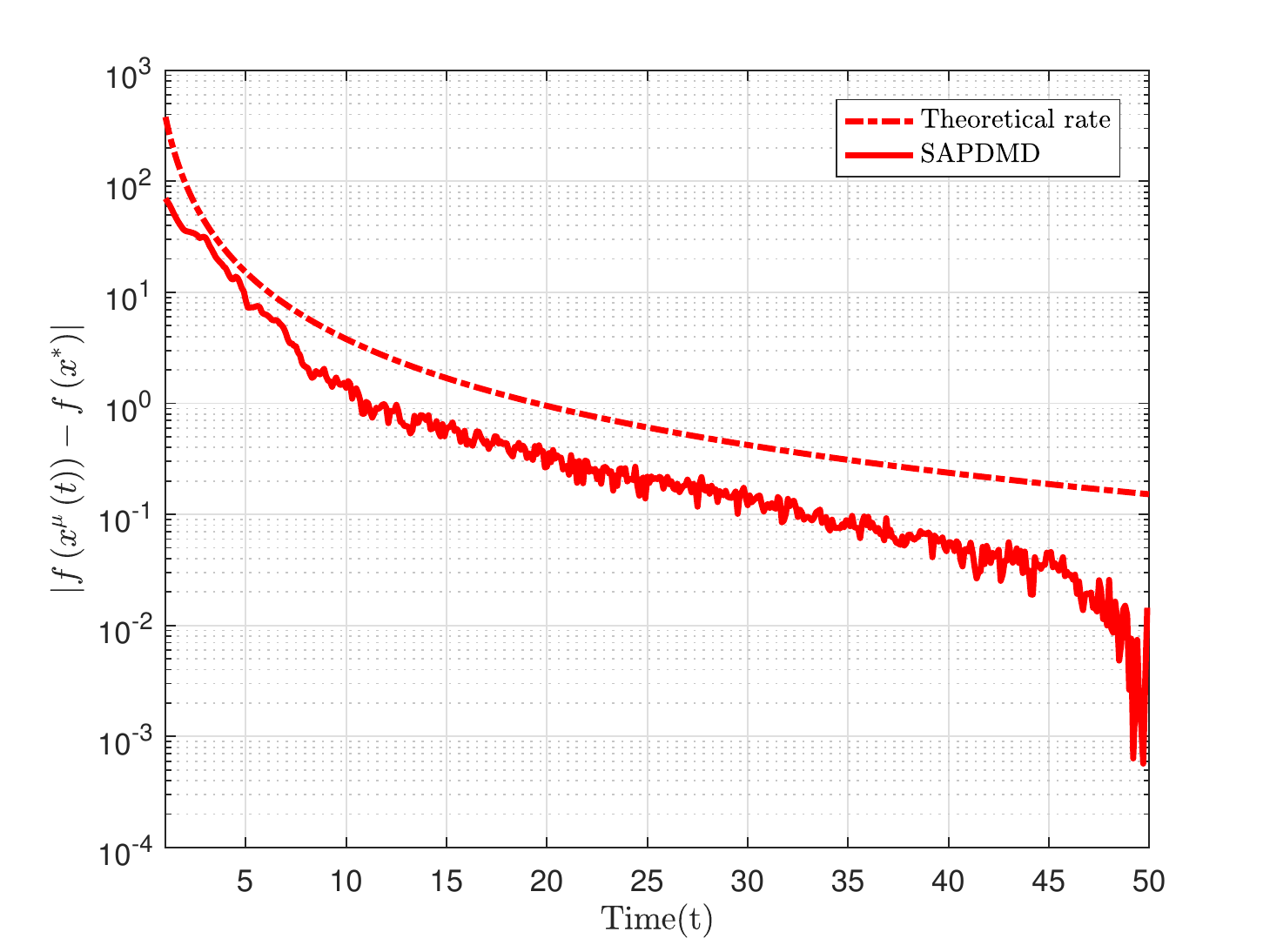}
	\caption{SADPDMD \eqref{SADPDMD} with with $\alpha =3$. (left) Trajectories of $x$ in 5 agents. (middle) Reconstructed sparse signal of 5 agents. (right)  Error of $\left| f\left( x^{\mu}\left( t \right) \right) -f\left( x^* \right) \right|$.}
	\label{fig:SDAPDMD_exp}
\end{figure}

\begin{example}\textbf{Distributed basis pursuit in column partition} in \cite{Zhao2021CentralizedAC}: Consider a distributed basis pursuit with column partition of sensing matrix as follows:
	\begin{equation}\label{D_BP_C}
		\begin{split}
			&\min  f\left( x \right) =\sum_{i=1}^q{\left\| x_i \right\| _1},
			\\
			&\mathrm{s}.\mathrm{t}. \,\,\bar{A}x+Ly=b,
		\end{split}
	\end{equation}
	where $\bar{A}=\mathrm{bl}\mathrm{diag}\left\{ A_{m\times n_1},...,A_{m\times n_q} \right\} \in \mathbb{R} ^{mq\times n}, \sum_{i=1}^q{n_q}=n$, and $y\in \mathbb{R} ^{mq}$ is an auxiliary variable. Note that, the problem \eqref{D_BP_C} is convex (but not strongly convex) and nonsmooth, which is a special case of DEMO \eqref{P3}. Letting $q=10$, and  $A\in \mathbb{R} ^{10\times 60}$, then $\hat{A}\in \mathbb{R}^{100\times 60}$. Let every
	part $A_{m\times n_i}, i=1,...,10$ be an agent, and 10 agents are connected as a ring. Applying the ADMD \eqref{SADMD} to solve the problem \eqref{D_BP_C} and the experimental results are shown in \eqref{fig:DENMO_nonsmooth_exp}. \eqref{fig:DENMO_nonsmooth_exp} (left) displays the dual variable $\lambda$ is uniformly convergent and globally asymptotically stable, i.e.,  $\lambda^{\mu} _1=\lambda^{\mu} _2=...=\lambda^{\mu} _{10}$; \eqref{fig:DENMO_nonsmooth_exp} (middle) describes that the SADMD
	\eqref{SADMD} is able to efficiently solve the problem\eqref{D_BP_C} and recover the sparse signal in a distributed way;  \eqref{fig:DENMO_nonsmooth_exp} (right) illustrates $\left| f\left( x^{\mu}\left( t \right) \right)-f\left( x^* \right) \right|$ of SADMD \eqref{SADMD} converges at the predicted rates.
\end{example}
\begin{figure}[!thbp]
	\includegraphics[width=5cm,height=4cm]{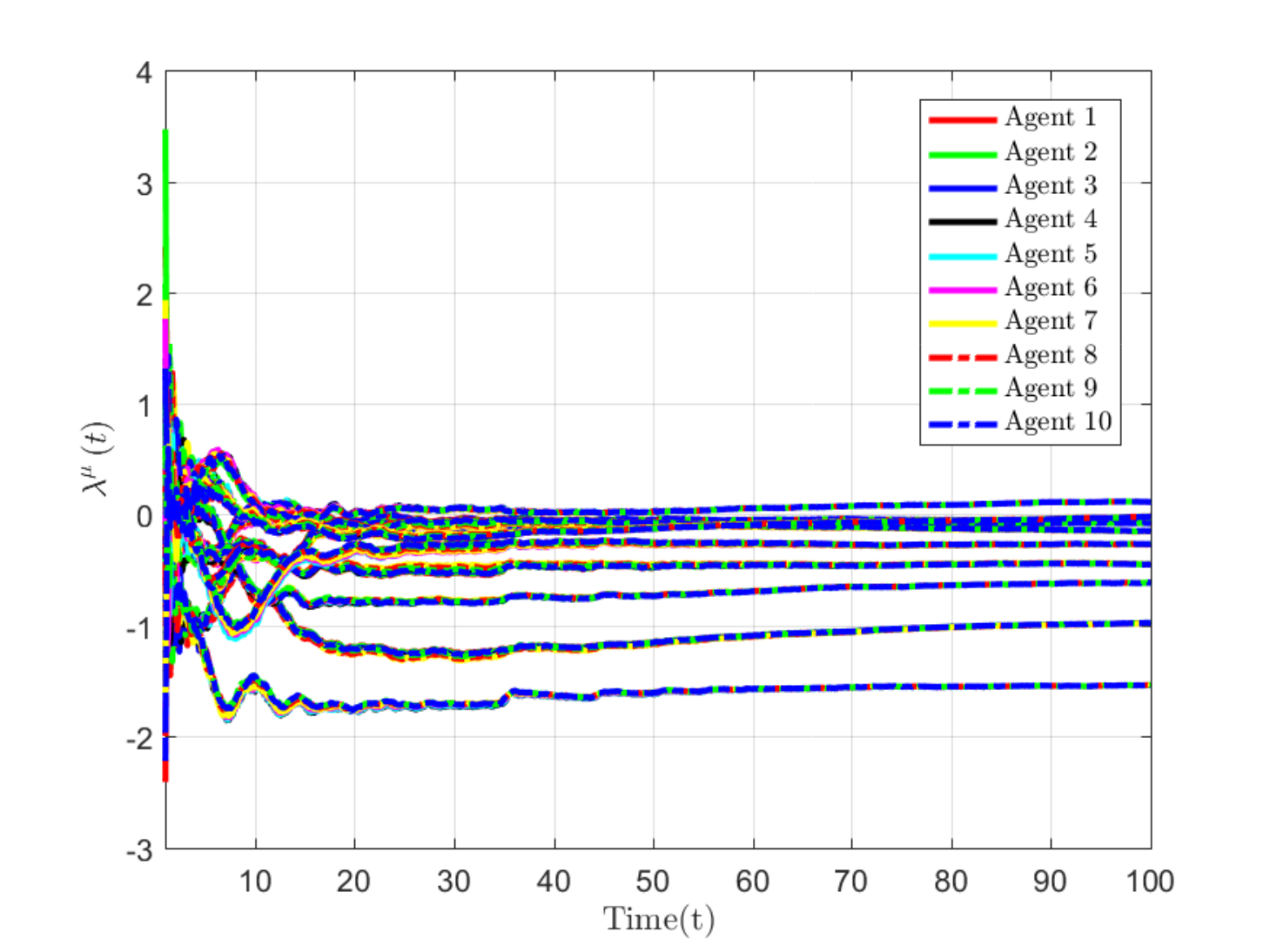}
	\includegraphics[width=5cm,height=4cm]{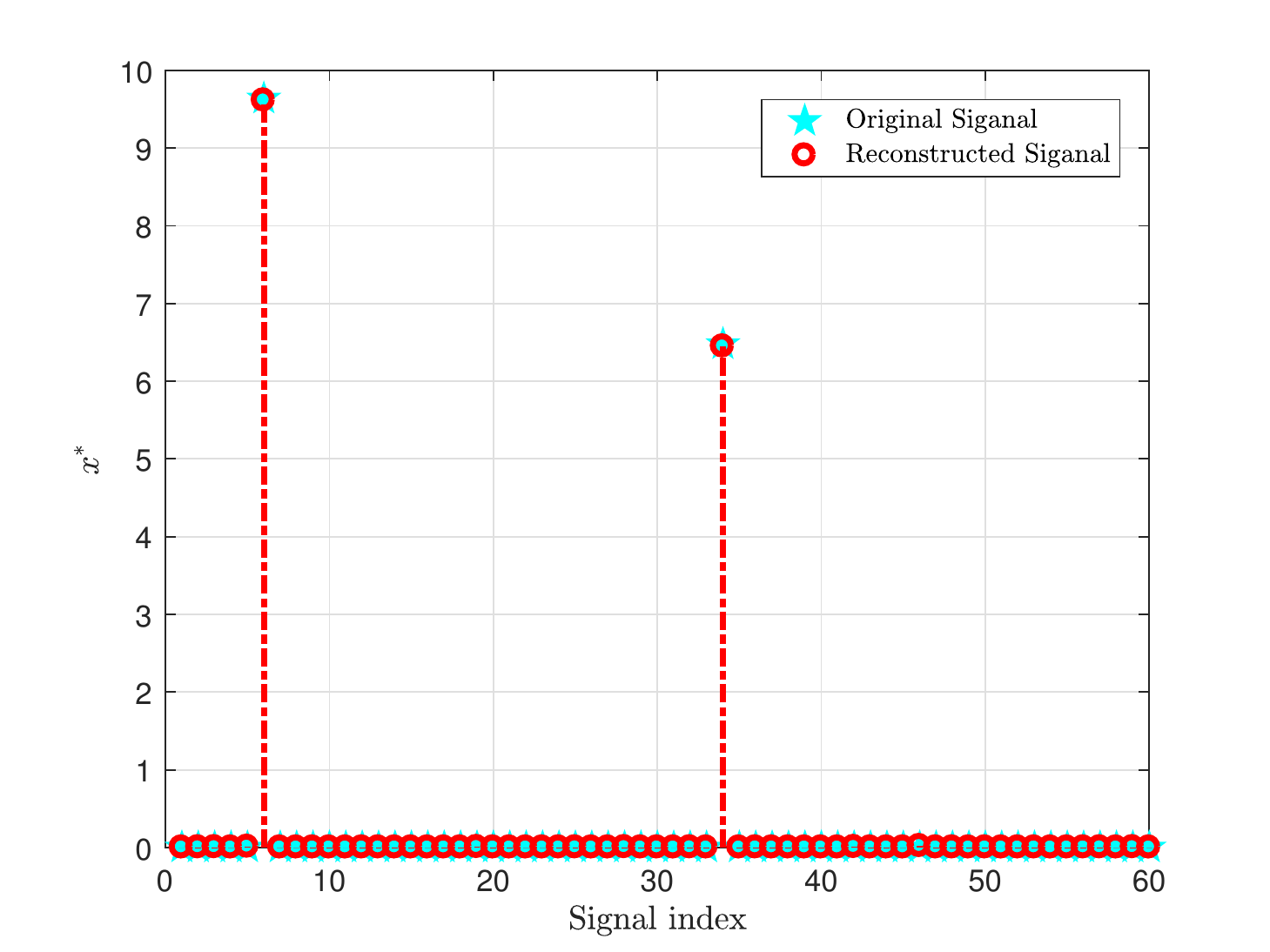}
	\includegraphics[width=5cm,height=4cm]{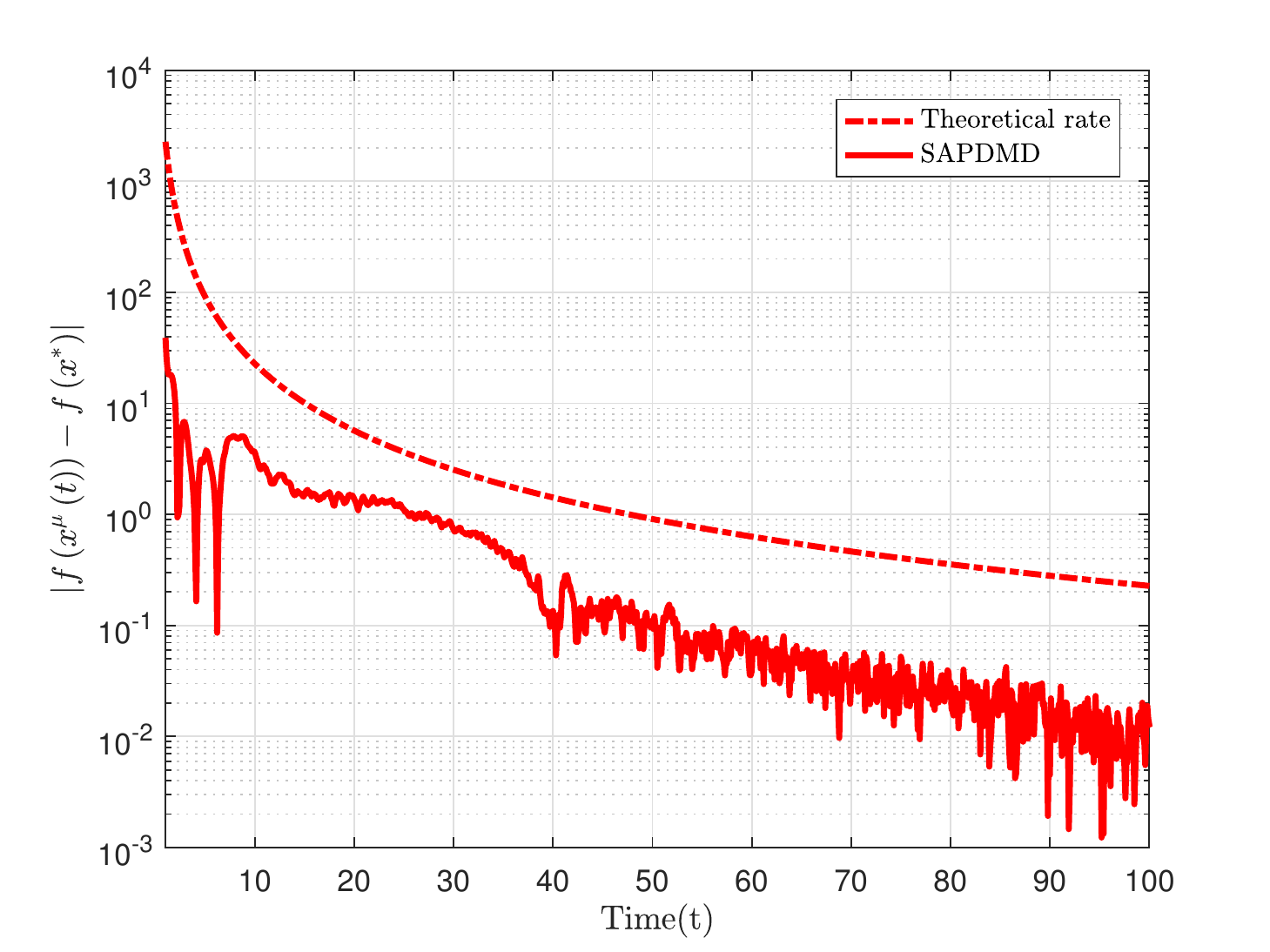}
	\caption {SADMD \eqref{SADMD} with with $\alpha =3$. (left) Trajectories of $\lambda$ of 10 agents. (middle)  Reconstructed sparse signal. (right) Error of $\left| f\left( x^{\mu}\left( t \right) \right) -f\left( x^* \right) \right|$.}
	\label{fig:DENMO_nonsmooth_exp}
\end{figure}

%\blindmathpaper
%
%
%Here is a citation \cite{chow:68}.

% Acknowledgements and Disclosure of Funding should go at the end, before appendices and references

\acks{This work was supported in part by the National Key R$\&$D Program of China (No. 2018AAA0100101), in part by National Natural Science Foundation of China (Grant no. 61932006, 61772434) and in part by the Fundamental Research Funds for the Central Universities (Project No. XDJK2020TY003).}

% Manual newpage inserted to improve layout of sample file - not
% needed in general before appendices/bibliography.

%\newpage
%
%\appendix
%\section*{Appendix A.}
%\label{app:theorem}

% Note: in this sample, the section number is hard-coded in. Following
% proper LaTeX conventions, it should properly be coded as a reference:

%In this appendix we prove the following theorem from
%Section~\ref{sec:textree-generalization}:

%In this appendix we prove the following theorem from
%Section~6.2:
%
%\noindent
%{\bf Theorem} {\it Let $u,v,w$ be discrete variables such that $v, w$ do
%not co-occur with $u$ (i.e., $u\neq0\;\Rightarrow \;v=w=0$ in a given
%dataset $\dataset$). Let $N_{v0},N_{w0}$ be the number of data points for
%which $v=0, w=0$ respectively, and let $I_{uv},I_{uw}$ be the
%respective empirical mutual information values based on the sample
%$\dataset$. Then
%\[
%	N_{v0} \;>\; N_{w0}\;\;\Rightarrow\;\;I_{uv} \;\leq\;I_{uw}
%\]
%with equality only if $u$ is identically 0.} \hfill\BlackBox
%
%\noindent
%{\bf Proof}. We use the notation:
%\[
%P_v(i) \;=\;\frac{N_v^i}{N},\;\;\;i \neq 0;\;\;\;
%P_{v0}\;\equiv\;P_v(0)\; = \;1 - \sum_{i\neq 0}P_v(i).
%\]
%These values represent the (empirical) probabilities of $v$
%taking value $i\neq 0$ and 0 respectively.  Entropies will be denoted
%by $H$. We aim to show that $\fracpartial{I_{uv}}{P_{v0}} < 0$....\\
%
%{\noindent \em Remainder omitted in this sample. See http://www.jmlr.org/papers/ for full paper.}

\vskip 0.2in
\bibliography{sample}

\end{document}